\documentclass[phd,abstractpage,appendix,dedicationpage,copyrightpage,notables,nofigures]{uithesis4}
\usepackage{amsfonts}
\usepackage{amsmath}
\usepackage{amssymb}
\usepackage{amscd}
\usepackage{color}
\usepackage{mathrsfs}
\usepackage{mathtools}
\usepackage{calc} 
\usepackage{verbatim}
\usepackage{amsthm}
\usepackage{djgcommands}
\usepackage{enumitem}
\usepackage[all]{xypic}

\newcommand{\LM}{\textup{LM}}
\newcommand{\RM}{\textup{RM}}
\newcommand{\UM}{\textup{UM}}
\newcommand{\LC}{\textup{LC}}
\newcommand{\RC}{\textup{RC}}
\newcommand{\UC}{\textup{UC}}
\newcommand{\DRM}{\text{DRM}}
\newcommand{\DUM}{\text{DUM}}
\newcommand{\DLC}{\text{DLC}}
\newcommand{\DRC}{\text{DRC}}
\newcommand{\DUC}{\text{DUC}}
\newcommand{\AM}{\text{AM}}
\newcommand{\GM}{\text{GM}}
\newcommand{\cad}{c\`{a}dl\`{a}g }
\newcommand{\hstar}{\, \widehat{\star} \,}
\newcommand{\tstar}{\, \widetilde{\star} \,}

\title{Unital Dilations of Completely Positive Semigroups}

\author{David J.\ Gaebler}
\dept{Mathematics}
\advisor{Professor Paul S.\ Muhly}
\memberOne{Paul S.\ Muhly}
\memberTwo{Palle E.\ T.\ Jorgensen}
\memberThree{Raul E.\ Curto}
\memberFour{Ionut Chifan}
\memberFive{Charles Frohman}
\submitdate{May 2013}
\copyrightyear{2013}

\dedication{thesisDedication}

\ackfile{thesisAck}
\abstractfile{thesisAbstract}

\begin{document}

\frontmatter


\chapter*{Preface: Background and Terminology}
This thesis is intended to be readable by a graduate student with
a working knowledge of the fundamentals of functional analysis and
operator algebras, but without prior exposure to the theory of
completely positive maps or of operator semigroups.  For instance,
the preparation provide by \cite{MacCluer} and \cite{Zhu} should be
adequate, and that by \cite{KadisonRingrose1} ample.

Following \cite{Sakai}, we distinguish between W$^*$-algebras, which
are abstractly defined as C$^*$-algebras having a Banach-space predual (necessarily unique, as it turns out), and
von Neumann algebras, which are concretely defined as weakly closed
self-adjoint subalgebras of $B(H)$ for some Hilbert space $H$.
In this convention, every von Neumann algebra is also a W$^*$-algebra
(with predual equal to a quotient of the predual $B(H)_* \simeq L^1(H)$),
whereas every W$^*$-algebra is isomorphic to some von Neumann algebra
(\cite{Sakai} 1.16.7).  We depart somewhat from Sakai in
referring to the weak-* topology on a W$^*$-algebra as
the \textbf{ultraweak topology}, which he calls the $\sigma$-topology
or weak topology, and the topology induced by the seminorms
$x \mapsto \phi(x^* x)$ for positive weak-* continuous functionals
$\phi$ as the \textbf{ultrastrong topology}, which he calls
the strong topology or s-topology.
In the case of a von Neumann algebra, these topologies coincide
with the ultraweak and ultrastrong operator topologies
as usually defined (\cite{Sakai} 1.15.6), and hence also with
the weak and strong operator topologies on bounded subsets
(\cite{Sakai} 1.15.2).  Because of this latter fact, we sometimes drop the ``ultra'' and refer merely to the \textbf{weak} and \textbf{strong} topologies when working on a bounded subset of a W$^*$-algebra.
We shall also make (rare) use of the \textbf{ultrastrong-* topology},
in which $x_\nu \to x$ iff $x_\nu \to x$ strongly and $x_\nu^* \to
x^*$ strongly.  Among the properties of these topologies that we will need are the following:
\begin{itemize}
    \item Multiplication is separately continuous in both the ultraweak and ultrastrong topologies.  However, it is jointly continuous in neither.  On bounded sets, multiplication is jointly strongly and strong-* continuous, but not jointly weakly continuous.
    \item The adjoint map $x \mapsto x^*$ is ultraweakly continuous,
    but not ultrastrongly nor even strongly continuous.
    \item On bounded subsets, one may relate the weak and
    strong topologies as follows: $x_\nu \to x$ strongly
    iff $x_\nu \to x$ weakly and $x_\nu^* x_\nu \to x^* x$
    weakly. 
    \item The \textbf{Kaplansky density theorem}: If $A$ is a W$^*$-algebra and $A_0 \subset A$ an ultraweakly dense
        *-subalgebra, then the unit ball of $A_0$ is
        strong-* dense in the unit ball of $A$.  In the case
        of a von Neumann algebra, the hypothesis of ultraweak
        density may be replaced by WOT-density.
\end{itemize}

A linear map between W$^*$-algebras
which is continuous with respect to their ultraweak topologies
is called \textbf{normal}; if the map in question
is positive, this is equivalent
to the property of preserving
upward-convergent nets (in this case weak and strong convergence
are equivalent) of
positive elements, that is, a positive linear map is normal iff
$\phi(x_\alpha) \uparrow \phi(x)$ whenever $x_\alpha \uparrow x$
(\cite{ConwayOperator} Corollary 46.5).  A C$^*$-isomorphism
between two W$^*$-algebras is automatically normal, but a *-homomorphism or
completely positive map need not be.

We refer to a W$^*$-algebra $\AAA$ as \textbf{separable} if its
predual $\AAA_*$ is a separable Banach space; this can be shown to be equivalent to numerous other conditions, including the separability of
either $\AAA$ or its unit ball in either the weak or strong topologies, and the existence of a faithful normal representation of $\AAA$ on a separable Hilbert space.  A related but strictly weaker property is that of
\textbf{countable decomposability}, which can be defined as the property
that every mutually orthogonal family of nonzero projections in $\AAA$
is at most countable; this is equivalent to the existence
of a faithful state, the existence of a faithful
normal state, or the strong metrizability of
the unit ball (\cite{Blackadar} III.2.2.27).

Additional background material, such as free probability and Hilbert C$^*$-modules, will be addressed in the chapters where these topics first appear.

Throughout, we use the boldface symbol $\one$ to denote the unit
of an algebra, while $1$ will denote the natural number.

\chapter{Introduction to Completely Positive Semigroups} \label{chapintroduction}

\section{Completely Positive Maps, Completely Positive Semigroups, and Endomorphism Semigroups}
In this section we introduce the basic objects of study.

\begin{definition} \label{defncpmap}
Let $A, B$ be C$^*$-algebras and $\phi: A \to B$ a linear map.  We
say that $\phi$ is
\begin{enumerate}
    \item \textbf{positive} if it maps positive elements of $A$
    to positive elements of $B$,
    \item \textbf{$n$-positive} if the map $I_n \otimes \phi:
    M_n(\com) \otimes A \to M_n(\com) \otimes B$ is positive, and
    \item \textbf{completely positive} if $\phi$ is $n$-positive
    for all $n \geq 1$.
\end{enumerate}
\end{definition}

We record here without proof some of the important properties of completely
positive maps.
\begin{itemize}
    \item If either $A$ or $B$ is commutative, the map $\phi: A \to B$
    is positive iff it is completely positive.  (\cite{Paulsen} Theorems 3.9 and 3.11)
    \item Every positive linear map is a *-map, that is, has
    the property that $\phi(a)^* = \phi(a^*)$ for all $a \in A$.  (\cite{Paulsen} Exercise 2.1)
    \item If $\phi$ is 2-positive (so, in particular, if $\phi$ is
    completely positive), then $\phi(a)^* \phi(a) \leq \phi(a^* a)$
    for all $a \in A$.  This is known as the \textbf{Schwarz
    inequality for 2-positive maps}.  (\cite{Paulsen} Proposition 3.3)
    \item If $A$ and $B$ are W$^*$-algebras, a completely positive
    map $\phi: A \to B$ is normal iff it is strongly continuous.  (\cite{Blackadar} Proposition III.2.2.2).  Strong continuity is equivalent to ultrastrong because of the boundedness of the map.
    \item If $\phi: A \to B(H)$ is a completely positive map, there
    exists a triple $(K, V, \pi)$, unique up to isomorphism, such that
    \begin{enumerate}
        \item $K$ is a Hilbert space
        \item $V: H \to K$ is a linear map such that $\|\phi\| = \|V\|^2$
        \item $\pi: A \to B(K)$ is a *-homomorphism such that
        $V^* \pi(a) V = \phi(a)$ for all $a \in A$
    \end{enumerate}
    and with the additional minimality property that $\overline{\pi(A) VH} = K$.  The triple $(H, V, \pi)$ is called the \textbf{minimal Stinespring
    dilation} of $\phi$.  If $\phi$ is unital, $V$ is an isometry; if $\phi$ is normal, so is $\pi$.  This is known as
    \textbf{Stinespring's dilation theorem} (\cite{Stinespring},
    \cite{Paulsen} Theorem 4.1,
    \cite{Blackadar} Theorems II.6.9.7 and III.2.2.4).
\end{itemize}

\begin{definition} \label{defcpsemigroup}
 Let $\AAA$ be a C$^*$-algebra (resp. W$^*$-algebra).
 \begin{enumerate}
    \item A \textbf{cp-semigroup} on $\AAA$ is a family
    $\{\phi_t: t \in [0,\infty)\}$ of (normal) completely positive contractive linear
    maps  $\phi_t: \AAA \to \AAA$ such that $\phi_0 = \dss{\text{id}}{\AAA}$ and
    \[
    \phi_t \circ \phi_s = \phi_{t+s}
    \]
    for all $s,t \geq 0$.

    \item An \textbf{e-semigroup} on $\AAA$ is a cp-semigroup in
    which each $\phi_t$ is a *-endomorphism.

    \item Capital letters (\textbf{CP-semigroup}, \textbf{E-semigroup}) indicate
    that for each $a \in \AAA$, $t \mapsto \phi_t(a)$ is a
    continuous function from $[0,\infty)$ to $\AAA$, where $\AAA$
    is given the norm (resp. ultraweak) topology.  We refer to this
    continuity property of the semigroup as \textbf{strong continuity}
    or \textbf{point-norm continuity} in the C$^*$ case, and
    \textbf{point-weak continuity} in the W$^*$ case.

    \item A subscript of 0 (\textbf{cp$_0$-semigroup}, \textbf{CP$_0$-}, \textbf{e$_0$-}, \textbf{E$_0$-})
    indicates that $\AAA$ contains a unit $\one$
    and that $\phi_t(\one) = \one$ for all $t \geq 0$.
 \end{enumerate}
\end{definition}

\begin{remark} \label{remQMP}
The term \textbf{quantum Markov process} or \textbf{quantum Markov
semigroup} is sometimes used in the literature to describe cp-, cp$_0$-, CP, or CP$_0$-semigroups; however, the usage is nonuniform as to which of these is indicated, so we adhere to the more precise notation above.
\end{remark}

\begin{remark} \label{remCPdefinition}
In the case where $\AAA$ is a W$^*$-algebra,
the definition of cp-semigroups and CP-semigroups remain unchanged
when stated in terms of the strong topology rather than the weak
topology.  That is, each map $\phi_t$ is normal iff it is strongly
continuous, as noted above; and, as we shall show in more detail
below, the map $t \mapsto \phi_t(a)$ for fixed $a$ is continuous
with respect to the weak topology on bounded subsets
of $\AAA$ iff it is continuous
with respect to the strong topology (that is, point-weak continuity is equivalent to point-strong continuity).
\end{remark}

\begin{definition} \label{definvariantstate}
Let $\phi = \{\phi_t\}$ be a cp-semigroup on $A$.  An
\textbf{invariant state} for $\phi$ is a state $\omega: A \to \com$
with the property
\[
\forall t \geq 0: \qquad \omega \circ \phi_t = \omega.
\]
\end{definition}

\section{Dilation}
In this section we introduce the ways in which cp-semigroups and e-semigroups may be related to each other.

\begin{definition} \label{defembeddingretraction}
Let $A, B$ be C$^*$-algebras.
\begin{enumerate}
    \item A \textbf{conditional expectation} on $A$ is a linear
    map $E: A \to A$ such that $E^2 = E$, $E(x^*) = E(x)^*$ for all
    $x \in A$, and $\|E\| = 1$.
    \item An \textbf{embedding from $A$ to $B$} is an injective
    *-homomorphism from $A$ to $B$.
    \item Given an embedding $i: A \to B$, a \textbf{retraction
    with respect to $i$} is a completely positive map $e: B \to A$
    such that $e \circ i = \text{id}_{A}$.
\end{enumerate}
\end{definition}

\begin{remark}
A linear map $E: A \to A$ is a conditional expectation iff it
is a completely positive contraction and is a bimodule map over its range, i.e.\ has the property that $E(E(a) x) =  E(a) E(x) = E(aE(x))$ for
all $a,x \in A$; this is known as \textbf{Tomiyama's theorem}
(\cite{Tomiyama}).  As a result, if $i: A \to B$ is an embedding
and $e: B \to A$ a corresponding retraction, then $i \circ e$
is a conditional expectation on $B$ with range $i(A)$.  Hence,
the distinction between a retraction and a conditional expectation is precisely the distinction between \emph{identifying} $A$ as a subalgebra of $B$, and \emph{explicitly writing an inclusion map} from $A$ to $B$.  The difference is a matter of taste; we generally follow the latter approach.
\end{remark}

\begin{definition}
Let $\phi = \{\phi_t\}$ be a cp-semigroup on a C$^*$-algebra $\AAA$.
 An \textbf{e-dilation} of $(A, \phi)$ is a tuple $(\Aa, i, \E, \sigma)$ where $\Aa$ is a C$^*$-algebra, $i: A \to \Aa$ an embedding,
$\E: \Aa \to A$ a retraction with respect to $i$,
and $\sigma = \{\sigma_t\}$ an e-semigroup on $\Aa$, satisfying
\[
\forall t \geq 0: \qquad \phi_t = \E \circ \sigma_t \circ i.
\]
We summarize the relationship in the following diagram:
\[ \xymatrixcolsep{5pc}\xymatrix{
\Aa \ar[r]^{\sigma_t} & \Aa \ar[d]^\E \\
A \ar[u]^i \ar[r]_{\phi_t} & A
} \]
We call $(\Aa, i, \E, \sigma)$
a \textbf{strong e-dilation} if it satisfies $\E \circ \sigma_t
= \phi_t \circ \E$, corresponding to the diagram
\[ \xymatrixcolsep{5pc}\xymatrix{
\Aa  \ar[d]_\E \ar[r]^{\sigma_t} & \Aa \ar[d]^\E \\
A \ar[r]_{\phi_t} & A
} \]
Note that this implies
\[
\phi_t = \phi_t \circ \E \circ i = \E \circ \sigma_t \circ i
\]
so that every strong dilation is a dilation, but the converse does
not always hold.
An e$_0$-dilation of a cp$_0$-semigroup is said to be \textbf{unital}
if $i(\one) = \one$.
\end{definition}

\section{Motivation and Examples}
\begin{example}
Let $H$ be a Hilbert space and $\{T_t\}$ a semigroup of contractions
on $H$.  Then the maps $\phi_t: B(H) \to B(H)$ defined by
\[
\phi_t(X) = T_t^* X T_T
\]
form a cp-semigroup.  It is a cp$_0$-semigroup iff all the $T_t$
are isometries, an e-semigroup iff all the $T_t$ are
coisometries, and hence an e$_0$-semigroup iff all the $T_t$
are unitaries.  If $T_t$ is strongly continuous, in that
$t \mapsto T_t$ is continuous with respect to the strong operator
topology on $B(H)$, then $\{\phi_t\}$ is a CP-semigroup.

A theorem of Cooper (\cite{Cooper}) states that, given
a strongly continuous contraction semigroup $\{T_t\}$ on
$H$, there exist a Hilbert space $K$, an isometry
$V: H \to K$, and a strongly continuous group $\{U_t\}$ of
unitaries on $K$ such that
\[
T_t = V^* U_t V.
\]
If the $T_t$ are isometries, one obtains the stronger condition
\[
V T_t = U_t V.
\]
Given the Cooper dilation of the semigroup $\{T_t\}$,
one can then define
\begin{enumerate}
    \item the E$_0$-semigroup $\{\alpha_t\}$ on $B(K)$ by $\alpha_t(Y) = U_t^* Y U_t$
    \item the non-unital embedding $i: B(H) \to B(K)$ by $i(X) = V X V^*$
    \item the retraction $\E: B(K) \to B(H)$ by $\E(Y) = V^* Y V$
\end{enumerate}
Then $(B(K), i, \E, \{\alpha_t\})$ is an E$_0$-dilation of
$(B(H), \{\phi_t\})$.

This example
plays a role in the general theory; for instance, Evans and Lewis
prove their dilation theorem (\cite{EvansLewis}) by showing that certain more general
semigroups are equivalent to those of the form $X \mapsto T_t^* X T_t$,
and then applying Cooper dilation.
\end{example}

\begin{example}
In (one of the axiomatizations of) quantum mechanics, every physical system corresponds to a von Neumann algebra $\AAA$, with states of the system
corresponding to positive elements of $\AAA$ of trace 1.  A physical transformation of the system must map states to states and hence, in
particular, must be a positive map; a continuous-time evolution of the system corresponds therefore to a semigroup of positive maps.  If the system is entangled with an environment, a physical transformation of the composite system  must map composite states to composite states, which implies complete positivity of the restriction to the original system; hence, a continuous-time evolution of such an \textbf{open quantum system} is represented by a semigroup of completely positive maps.  Continuity requirements are also natural to impose in this setting as one of the physical axioms.

Actually, the representation of such a system as a completely positive semigroup represents an approximation to a more general \textbf{master equation}, which approximation holds under various simplifying physical assumptions such as those of ``weak coupling'' or a ``singular reservoir.''  Completely positive semigroups arise, for instance, in quantum thermodynamics, where the environment may be regarded as an infinite ``heat bath'' whose self-interactions are much faster than those of the system under study.  For more on these matters see \cite{Haake}, \cite{DaviesMME}, \cite{GKS}, \cite{Lindblad}, \cite{Davies}, \cite{EvansLewis}, and \cite{Attal2}.
In the thermodynamic context one typically assumes the existence of a
normal $\phi$-invariant state $\omega$ on $\AAA$, representing a thermodynamic
equilibrium of the system; correspondingly, one is interested in
dilations $(\Aa, i, \E, \sigma)$ for which there exists a
normal $\sigma$-invariant state $\varpi$ on $\Aa$, which dilates
$\omega$ in the sense that $\varpi \circ i = \omega$.  In the
case of a strong dilation this is automatic, as one can simply
define $\varpi = \omega \circ \E$, and it follows that
\[
\varpi \circ \sigma_t = \omega \circ \E \circ \sigma_t
= \omega \circ \phi_t \circ \E = \omega \circ \E =\varpi.
\]

In this setting, dilation is a way of relating the dynamics
of an open (or ``dissipative'') system to the dynamics of
a closed (or ``non-dissipative'') system containing it.
\end{example}

\begin{example} \label{exmarkovdilation}
Let $\AAA$ be a commutative unital C$^*$-algebra, and let $S$
be the maximal ideal space of $\AAA$, so that $\AAA \simeq C(S)$.
Let $\{P_t\}$ be a CP$_0$-semigroup on $\AAA$.  By Riesz representation
we obtain for each $t \geq 0$ and each $x \in S$ a measure $p_{t,x}$
characterized by the property
\[
\forall f \in C(S): \qquad \int_S f(y) \, dp_{t,x}(y) = (P_t f)(x).
\]
Moreover, since $P_t f$ is a continuous function, the family $\{p_{t,x}\}$
varies weak-* continuously in $x$.  The property $P_0 = \text{id}$
implies that $p_{0,x}$ is the point mass at $x$, and the
semigroup property $P_{s+t} = P_s P_t$ implies the property
\[
p_{t+s,x}(E) = \int_S p_{s,y}(E) dp_{t,x}(y),
\]
sometimes referred to as the \textbf{Chapman-Kolmogorov equation}.

Let $\SSSS$ denote the \textbf{path space}
$S^{[0,\infty)}$, and $\Aa =
C(\SSSS)$.  We have the embedding $i: \AAA \to \Aa$ given
by $i(f)(\pP) = f(\pP(0))$.
By the Stone-Weierstrass theorem, the *-subalgebra
$\Aa_0 \subset \Aa$ consisting of finite sums of functions of
the form $f_1^{(t_1)} \cdots f_n^{(t_n)}$, where for a path
$\pP \in \SSSS$ the value of $f_i^{(t_i)}$ depends only
on $\pP(t_i)$, is dense in $\Aa$.  We define a unital
linear map $\E_0: \Aa_0 \to \AAA$ on $\AAA_0$ by
\[
\E_0 [f_1^{(t_1)} \cdots f_n^{(t_n)}]
= f_n P_{t_n-t_{n-1}} \Big( f_{n-1} P_{t_{n-1}- t_{n-2}} \Big( \cdots
P_{t_2-t_1} \Big(f_1 \Big) \Big) \cdots \Big).
\]
Clearly $E_0 \circ i = \text{id}_\AAA$.
We will show shortly that $\E_0$ is well-defined and contractive,
so that it extends to a unital contractive
(hence positive, hence completely positive)
linear map $\E: \Aa \to \AAA$ which satisfies $\E \circ i = \text{id}_\AAA$
and is therefore a retraction with respect to $i$.

We define for each $t \geq 0$ the continuous maps $\lambda_t:
\SSSS \to \SSSS$ by $(\lambda_t \pP)(s) = \pP(s+t)$, and
the corresponding *-endomorphisms $\sigma_t: \Aa \to \Aa$ by
$\sigma_t f = f \circ \lambda_t$.  It is immediate from
the above that $\E \circ \sigma_t \circ i = P_t$, so that
we have obtained a unital e-dilation of our CP$_0$-semigroup.

Given any regular Borel probability measure $\mu_0$ on $S$,
we obtain through Riesz representation a regular Borel probability
measure $\mu$ on $\SSSS$ characterized by the property
\[
\forall f \in \Aa: \qquad \int_\SSSS f \, d\mu
= \int_S (\E f) \, d\mu_0.
\]
This then implies that
\[
\forall f \in \AAA: \qquad (P_t f)(x)
= \EE \Big[f(\pP(t)) \Big| \pP(0) = x\Big]
\]
where $\EE$ denotes conditional expectation in the probabilistic
sense, so that we have constructed a Markov process $\{\pP(t)\}$
with specified transition probabilities.
We thus obtain a C$^*$-algebraic version of the classical
\textbf{Daniell-Kolmogorov construction}, at least
in the context of \textbf{Feller processes}
rather than general Markov processes.

We now consider an alternate perspective on the same construction, which
enables us easily to prove that $\E_0$ is well-defined and contractive, and simultaneously offers a preview of the techniques used in
this thesis.
For each nonempty finite subset $\gamma \subset [0,\infty)$ let
$\AAA_\gamma$ denote a tensor product of $|\gamma|$ copies of
$C(S)$.  For $\beta \subset \gamma$ we obtain unital embeddings
$\AAA_\beta \to \AAA_\gamma$ as follows: Writing $\gamma$
as a disjoint union $\beta \cup \gamma'$, identify
$\AAA_\gamma$ with $\AAA_\beta \otimes \AAA_{\gamma'}$
and embed via $f \mapsto f \otimes \one$.  This yields
an inductive system and, using the general fact that
$C(X \times Y) \simeq C(X) \otimes C(Y)$ for compact
Hausdorff spaces $X$ and $Y$, we see that
$\displaystyle\lim_{\rightarrow} \AAA_\gamma$ is isomorphic
to $\Aa$.  The domain of $\E_0$ is the union of the images of
all the $\AAA_\gamma$ inside $\Aa$, and the well-definedness and
contractivity of $\E_0$ reduce, by induction, to the well-definedness
and contractivity of the maps $\theta_t: C(S) \otimes C(S)$ given on
simple tensors by $\theta_t (f \otimes g) = (P_t f) g$.  But
such a map $\theta_t$ may
be equivalently defined as
\[
(\theta_t F)(x) = \int_S F(y,x) dp_{t,x}(y)
\]
which obviously yields a well-defined contraction
on $C(S) \otimes C(S)$.

We note that the e-semigroup $\{\sigma_t\}$ is not continuous,
even when the original semigroup $\{P_t\}$ is; that is,
we obtain only an e$_0$-dilation, not an E$_0$-dilation, of
a CP$_0$-semigroup.  We shall return to this point in chapter
\ref{chapcontinuous}.
\end{example}

\begin{remark} \label{remwhy}
We view the last two examples as representing the two major streams of thought which motivate the study of the dilation theory of completely positive semigroups.  On the one hand, in the physics setting such a dilation corresponds to an embedding of an open quantum system inside some closed quantum system.  On the other hand, we have seen that
dilating a CP$_0$-semigroup defined an a \emph{commutative} C$^*$-algebra amounts to construction of a Markov process; hence, we may think of
dilations of general CP$_0$-semigroups as a way
of constructing ``noncommutative Markov processes.''
\end{remark}

\section{Continuity Properties of Semigroups}
In this section we examine in greater detail the continuity properties
of completely positive semigroups, beginning with more general
considerations regarding contraction semigroups on Banach spaces.

\subsection{C$_0$-Semigroups} \label{secC0semigroups}
We recount here some of the essentials of the theory of contraction
semigroups on Banach spaces, which can be found in \cite{HillePhillips},
\cite{DunfordSchwarz1}, \cite{BratteliRobinson1}, and \cite{EngelNagel}.

A semigroup $\{T(t)\}_{t \geq 0}$ of contractions on a Banach space
$\Xx$ is called a \textbf{contraction semigroup}.  Such a semigroup
is said to be
\begin{enumerate}
    \item \textbf{uniformly continuous} if $t \mapsto T(t)$ is
    continuous with respect to the norm topology on $B(\Xx)$; that
    is, if $\displaystyle\lim_{t \to t_0} \|T(t)-T(t_0)\|_{B(\Xx)} = 0$
    \item \textbf{strongly continuous} if, for each
    $x \in \Xx$, $t \mapsto T(t) x$ is continuous with respect
    to the norm topology on $\Xx$; that is, if
    $\displaystyle\lim_{t \to t_0} \|T(t)x - T(t_0)x\|_\Xx = 0$ for each $x \in \Xx$
    \item \textbf{weakly continuous} if, for each $x \in \Xx$,
    $t \mapsto T(t) x$ is continuous with respect to the weak topology
    on $\Xx$; that is, if $\displaystyle\lim_{t \to t_0} \ell \big( T(t) x - T(t_0) x \big) = 0$ for each $x \in \Xx$ and each
    $\ell \in \Xx^*$
\end{enumerate}
In case $\Xx$ is the dual of some other Banach space $\Xx_*$, we
define the semigroup to be
\begin{enumerate}[resume]
    \item \textbf{weak-* continuous} if, for each $x \in \Xx$,
    $t \mapsto T(t) x$ is continuous with respect to the weak-* topology
    on $\Xx$; that is, if $\displaystyle\lim_{t \to t_0} \ell \big( T(t) x - T(t_0) x \big) = 0$ for each $x \in \Xx$ and each
    $\ell \in \Xx_*$
\end{enumerate}

These modes of continuity can, of course, be defined for other
families $\{T(t)\}$ of operators which are not necessarily contractions
and do not necessarily form a semigroup.  In the case of contraction
semigroups, however, it turns out that strong and weak continuity
are equivalent (\cite{EngelNagel} Theorem 1.1.6).  Furthermore, uniform continuity is too stringent a hypothesis to be attainable in most applications of interest, so that the bulk of the study of contraction semigroups revolves around strongly continuous contraction semigroups, also known as \textbf{contractive C$_0$-semigroups} .  More generally,
one can study strongly continuous semigroups of bounded operators,
but it can be easily shown that these may all be written as scalar-valued exponential functions times contraction semigroups, so that one reduces to the contractive case.

The most important object associated with a contractive C$_0$-semigroup
is its \textbf{generator}, the operator $\LL$ on $\Xx$ defined by
the formula
\[
\LL x = \lim_{t \to 0} t\inv [T(t) x - x].
\]
This is in general a closed densely defined unbounded operator,
and in fact
is bounded iff the semigroup is uniformly continuous.  Furthermore,
the generator satisfies the resolvent growth condition $\|(\lambda \one - \LL)\inv\| \leq \lambda\inv$ for all $\lambda > 0$.  The
\textbf{Hille-Yosida theorem} provides a converse, stating that
every closed densely defined operator satisfying this resolvent growth
condition is the generator of some C$_0$-semigroup.  Intuitively,
this semigroup is given by $T(t) = e^{t \LL}$, but this exponential
functional cannot be defined through the usual power series when
$\LL$ is unbounded; one can, however, write
\[
T(t) x = \lim_{n \to \infty} \left(\one - \frac{t}{n} \LL \right)^{-n} x
\]
which is known as the \textbf{Post-Widder inversion formula}
for C$_0$-semigroups.  We thus have a bijection between
contractive C$_0$-semigroups and closed densely defined operators satisfying
a resolvent growth condition, with explicit formulas for both
directions of the bijection.

A notable consequence of the semigroup property is the equivalence between certain notions of continuity and measurability.  We
define a family $\{T(t)\}$ of operators on $\Xx$,
equivalently viewed as a function $T: [0,\infty) \to B(\Xx)$, to be
\begin{enumerate}
    \item \textbf{uniformly measurable} if $T$ is the a.e. norm
    limit of a sequence of countably-valued functions from
    $[0,\infty)$ to $B(\Xx)$
    \item \textbf{strongly measurable} if, for each $x \in \Xx$,
    $t \mapsto T(t) x$ is the a.e. norm limit of a sequence
    of countably-valued functions from $[0,\infty)$ to $\Xx$
    \item \textbf{weakly measurable} if, for each $x \in \Xx$
    and $\ell \in \Xx^*$, $t \mapsto \ell(T(t) x)$ is
    a measurable function from $[0,\infty)$ to $\com$
\end{enumerate}
In case $\Xx$ is the dual of another Banach space $\Xx_*$,
we also define $\{T(t)\}$ to be
\begin{enumerate}[resume]
    \item \textbf{weak-* measurable}  if, for each $x \in \Xx$
    and $\ell \in \Xx_*$, $t \mapsto \ell(T(t) x)$ is
    a measurable function from $[0,\infty)$ to $\com$.
\end{enumerate}
One might ask why we do not instead define the different types of measurability using the Borel $\sigma$-algebras generated by the corresponding continuity types; the short answer is that a better integration theory results from the definitions given here (the Bochner integral in the case of uniform measurability, the Pettis integral for the others).

It turns out that weak and strong measurability are equivalent
when $\Xx$ is separable (\cite{HillePhillips} Corollary 2, p.\ 73)
and, when $\{T(t)\}$ is a contraction semigroup, both are
equivalent to strong and weak continuity at times $t > 0$ (\cite{HillePhillips} Theorem 10.2.3).  This latter result is analogous to the fact that measurable solutions to the Cauchy functional equation
$f(x+y) = f(x)f(y)$ on $\re$ are exponentials, and hence are continuous.  However, strong measurability at $t = 0$ is not enough
to infer strong continuity at $t = 0$, but requires the additional
hypothesis that $\bigcup_{t > 0} T(t) \Xx$ be dense in
$\Xx$ (\cite{HillePhillips} Theorem 10.5.5).

A contraction semigroup $\{T(t)\}$ on $\Xx$ induces an
\textbf{adjoint semigroup} $\{T(t)^*\}$ on $\Xx^*$ by the
formula $(T(t)^* f)(x) = f(T(t) x)$.  If $\Xx$ is the dual
of $\Xx_*$ and if each $T(t)$ is weak-* continuous,
one obtains also a \textbf{pre-adjoint semigroup} $\{T(t)_*\}$
through the same formula; since the weak-* topology is
of much more interest than the weak topology for spaces
having a predual, this is usually referred to in the literature as the
adjoint semigroup (and of course is the restriction of the adjoint
semigroup to $\Xx_* \subset \Xx^*$).  Weak-* continuity and
measurability of $\{T(t)\}$ are equivalent to weak continuity and measurability of $\{T(t)_*\}$, so that in particular they
are equivalent to each other at times $t > 0$
if $\Xx_*$ is separable.

The last topic to consider for contraction semigroups
is the passage from separate to joint continuity.  We
summarize the results in the following theorem.

\begin{theorem}[Joint Continuity of C$_0$-Semigroups] \label{thmjointcontinuityC0} \
\begin{enumerate}
    \item Let $\Xx$ be a Banach space and $\{T(t)\}_{t \geq 0}$
    a contractive C$_0$-semigroup.  Then $T(t)(x)$ is
    jointly continuous in $t$ and $x$; that is, the map
    $[0,\infty) \times \Xx \sa{T} \Xx$ is continuous
    with respect to the norm topology on $\Xx$.
    \item Let $\Xx$ be a Banach space with separable predual
    $\Xx_*$, and
    $\{T(t)\}_{t \geq 0}$ a weak-* continuous semigroup
    of weak-* continuous contractions on $\Xx$.
    Then $T(t)(x)$ is jointly weak-* continuous
    in $t$ and $x$ on bounded subsets of $\Xx$.
    That is, the map
    $[0,\infty) \times \Xx_1 \sa{T} \Xx_1$
    is continuous with
    respect to the weak-* topology on $\Xx_1$.

    \item Let $\AAA$ be a W$^*$-algebra and
    $\{\phi_t\}_{t \geq 0}$ a C$_0$-semigroup
    of strongly continuous contractions on $\AAA$.
    Then $\phi_t(a)$ is jointly strongly continuous
    in $t$ and $a$ at nonzero times.  That is,
    the map $(0,\infty) \times \AAA_1 \sa{\phi} \AAA_1$
    is continuous with
    respect to the strong topology on $\AAA_1$.
\end{enumerate}
\end{theorem}

\begin{proof} \
\begin{enumerate}
    \item By the triangle inequality and the contractivity
    of the semigroup,
    \[
    \|T(s)(y)-T(t)(x)\| \leq \|T(s)(y-x)\|
    + \|T(s)x - T(t)x\| \leq \|y-x\| + \|T(s)(x)-T(t)(x)\|
    \]
    which tends to zero as $(s,y) \to (t,x)$.

    \item By Alaoglu's theorem, $\Xx_1$ is weak-* compact, and since $\Xx_*$ is assumed to be separable, another standard result implies that $\Xx_1$ is weakly metrizable (\cite{ConwayFunctional} V.5.1).  Joint weak-*
        continuity at $(t,a)$ with $t > 0$ is therefore a special case of Theorem 4 in \cite{ChernoffMarsden}.  Joint weak-*
        continuity at $(0,a)$ is more complicated to establish, but is a consequence of Corollary 3.3 of \cite{Lawson}. For the purposes
        of self-containment,
        we sketch here the relevant arguments.
\begin{enumerate}
    \item If $T$ is a Baire space, $X$ a metric space,
    and $f: T \times X \to X$ a separately continuous function,
    then for each $x \in X$ there exists a dense $G_\delta$
    subset $T_0 \subseteq T$ such that, for all $t_0 \in T_0$,
    $f$ is jointly continuous at $(t_0, x)$.  This
    standard result appears as Exercise XI.10.11 in
    \cite{Dugundji} and as Exercise 7.41 in
    \cite{Royden}.  Given $x_0 \in X$,
    one defines for each $m,n \in \N$ the closed subset
    \[
    F_{m,n} = \left\{t \in T \mid \forall x \in B_{1/m}(x_0):
    \ d(f(t,x), f(t,x_0)) < \frac{1}{m} \right\},
    \]
    the union of all which is $T$.  One then defines the open
    dense subsets
    \[
    \OO_m = \bigcup_{n=1}^\infty F_{m,n}^\circ
    \]
    of $T$, and the intersection $T_0 = \bigcap_m \OO_m$ is therefore
    also dense by the Baire Category Theorem.  The sets
    are constructed in such a way that $f$ is jointly
    continuous at $(t_0, x_0)$ for all $t_0 \in T_0$.

    \item In the case where $T = [0,\infty)$ and $f$ is
    additive in the first variable in the sense
    that $f(t+s,x) = f(t,f(s,x))$ and $f(0,x) = x$,
    one can conclude further
    that $f$ is jointly continuous at $(t, x)$ for all $x \in X$
    and all $t > 0$.  This is theorem 4 of \cite{ChernoffMarsden},
    and is proved as follows: Let $x \in X$ and $t > 0$.
    Choose $T_0 \subset [0,\infty)$ dense such that
    $f$ is jointly continuous at all $(t_0, x)$ with
    $t_0 \in T_0$.  Because $T_0$ is dense in $[0,\infty)$,
    it must contain some element $t_0 < t$.  Then for $t'$
    sufficiently close to $t$ we will have $t' > \min(t_0, t-t_0)$, and can
    write for each $x' \in X$
    \[
    f(t', x') = f(t-t_0, f(t_0+t'-t,x')).
    \]
    Now as $t' \to t$ and $x' \to x$, $f(t_0+t'-t,x')
    \to f(t_0, x)$ by joint continuity at $(t_0, x)$.
    It follows that $f(t-t_0, f(t_0+t'-t,x')) \to
    f(t-t_0, f(t_0,x)) = f(t,x)$ by separate continuity.

    \item To establish joint continuity at points
    $(0,x)$, we add the assumption that $X$ is compact.
    For each $x \in X$, let $G$ be an open neighborhood
    of $x$.  For each $y \in X \setminus G$,
    separate continuity implies $f(t,x) \to x$ and
    $f(t,y) \to y$ as $t \to 0$.  Since $T_0$ defined
    as above is dense, one can therefore
    find a $t_0 \in T_0$ with the property $f(t_0, x)
    \neq f(t_0, y)$.  Some straightforward calculations
    then imply that there exist open sets $W_y \ni 0$,
    $U_y \ni x$, and $V_y \ni y$ such that $f(W_y \times
    U_y)$ is disjoint from $V_y$.  As the $V_y$ form
    an open cover of the compact set $X \setminus G$, there
    exists a finite subcover; taking $W$ and $U$ to be the
    corresponding finite intersections of the $W_y$
    and $U_y$, one has $(0, x) \in W \times U$
    and $f(W \times U) \subset G$.
\end{enumerate}

    \item  For strong continuity, we follow the same
proof, using the fact that
$\AAA_1$ is also strongly metrizable (\cite{Blackadar}
III.2.2.27).  Since
$\AAA_1$ is not strongly compact, however, we cannot
infer joint continuity at $(0,a)$.
\end{enumerate}
\end{proof}

\subsection{Completely Positive Semigroups} \label{secCPsemigroupcontinuity}
So far we have considered semigroups of contractions on Banach spaces.  When the Banach space happens to be a W$^*$-algebra, and the contractions happen to be normal completely positive maps, some stronger continuity results hold than are true in the more general setting.  Here we note two such results.  First, recall that a CP-semigroup was defined by the property
 of point-weak continuity.  It turns out that such a
semigroup is automatically \textbf{point-strongly continuous}.
This is Theorem 3.1 of
\cite{ShalitMarkiewicz}.

Our second continuity result which is specific to completely positive
semigroups is an improved statement of joint continuity.

\begin{theorem}[Joint Continuity for CP-Semigroups]  \label{thmjointcontinuityCP} \

Let $\AAA$ be a separable W$^*$-algebra and $\{\phi_t\}_{t \geq 0}$
a CP-semigroup on $\AAA$.
\begin{enumerate}
    \item $\phi_t(a)$ is jointlyweakly continuous in $t$
    and $a$; that is, the map
    $[0,\infty) \times \AAA_1 \sa{\phi} \AAA_1$ is
    continuous with respect to the weak topology on $\AAA_1$.

    \item $\phi_t(a)$ is jointly strongly continuous
    in $t$ and $a$; that is, the map
    $[0,\infty) \times \AAA_1 \sa{\phi} \AAA_1$ is
    continuous with respect to the strong topology on $\AAA_1$.
\end{enumerate}
\end{theorem}

\begin{proof}
\
\begin{enumerate}
  \item This follows from Theorem
    \ref{thmjointcontinuityC0}; we mention it here in order to observe that a considerably simpler proof is available in this special case, which appears as Proposition 2.23 of \cite{SeLegue} and as
    Proposition 4.1(2) of \cite{MS02}.
   \item This is an improvement on Theorem \ref{thmjointcontinuityC0} because of the joint continuity at time 0, which
       we shall need later.
       Assume that $\AAA \subset B(H)$, with $H$ separable.
   Let $t_n \to t$ be a convergent sequence in $[0,\infty)$
   and $a_n \to a$ an SOT-convergent sequence in $\AAA_1$.
   (We can use sequences rather than nets because
   $\AAA_1$ is SOT-metrizable.)  By the first part of this
   theorem, $\phi_{t_n}(a_n) \to \phi_t(a)$ in WOT.
   Now for any $h \in H$,
   \begin{align*}
   \left\| \phi_{t_n}(a_n) h - \phi_t(a) h \right\|^2
   &= \left\|\phi_{t_n}(a_n) h \right\|^2
   - 2 \realp \left\la \phi_{t_n}(a_n) h,
   \phi_t(a) h \right\ra + \left\| \phi_t(a) h \right\|^2\\   &= \la \phi_{t_n}(a_n)^* \phi_{t_n}(a_n) h, h \ra
   - 2 \realp \left\la \phi_{t_n}(a_n) h,
   \phi_t(a) h \right \ra + \left\| \phi_t(a) h \right\|^2\\
   &\leq \la \phi_{t_n}(a_n^* a_n) h, h \ra
   - 2 \realp \left\la \phi_{t_n}(a_n) h,
   \phi_t(a) h \right\ra + \left\| \phi_t(a) h \right\|^2
   \end{align*}
   where we have used the Schwarz inequality for 2-positive
   maps plus the fact that $a_n^* a_n \to a^* a$ in WOT whenever $a_n \to a$ in SOT.  Taking the limsup as $n \to \infty$, we see
   that $\phi_{t_n}(a_n) \to \phi_t(a)$ in SOT.

   This appears as Lemma 4 in \cite{VincentSmith} and
    as Lemma 6.4 in \cite{ShalitE0Dilation}.
\end{enumerate}
\end{proof}

\section{Survey of Extant Results}
The first results concerning the existence of dilations
for cp-semigroups date from the 1970's and pertain to uniformly continuous semigroups.  Recall that a contraction semigroup is uniformly continuous
iff its generator is bounded; \cite{ChristensenEvans}, preceded
in special cases by \cite{GKS} and \cite{Lindblad}, showed that the generator of a uniformly continuous CP-semigroup on a W$^*$-algebra must have the form $a \mapsto \Psi(a) + k^* a + ak$ for some element
$k \in \AAA$ and completely positive map $\Psi: \AAA \to \AAA$.
This structure theorem was used by \cite{EvansLewis} to prove that
a uniformly continuous CP-semigroup on a W$^*$-algebra has an
E-dilation.  However, attempts to prove the existence of dilations for point-weakly continuous, or even point-norm continuous CP-semigroups were unsuccessful.

Dilations were shown to exist in special cases (for instance,
on semigroups having specific forms, on semigroups satisfying additional hypotheses such as the existence of a faithful normal invariant state, in the case of discrete-time semigroups, or using a weaker sense of the word ``dilation'') by
\cite{EmchMinimalDilations},  \cite{AccardiFrigerioLewis}, \cite{VincentSmith}, \cite{KummererMarkovDilations}, and others.  However, progress on the general problem required a new insight.  This insight was the notion of a \textbf{product system of Hilbert spaces}, developed by Arveson (\cite{ArvesonContinuousAnalogues1}, \cite{ArvesonContinuousAnalogues2}, \cite{ArvesonContinuousAnalogues3}, \cite{ArvesonContinuousAnalogues4}).  We shall say more about product systems in chapter \ref{chapproductsystems}; briefly, there is an equivalence of categories between E$_0$-semigroups on $B(H)$ and product systems of Hilbert spaces, so that the problem of constructing E$_0$-dilations reduces in some sense to the problem of building a product system out of a CP$_0$-semigroup.  Variants of this strategy were used in \cite{BhatIndexTheory} and \cite{SeLegue}
to show that every CP$_0$-semigroup on $B(H)$ has an E$_0$-dilation, a result known as \textbf{Bhat's theorem}, and the corresponding result
for separable W$^*$-algebras was established in \cite{ArvesonDynamics}.  Later, the more general notion of a \textbf{product system of Hilbert modules} was introduced, leading to new proofs of these theorems in \cite{BhatSkeide} and \cite{MS02}.  More recently, product systems have been used to study families of completely positive maps indexed by semigroups other than $[0,\infty)$, with the existence of dilations depending on an additional hypothesis known as strong commutativity (\cite{ShalitE0Dilation}).

A different approach to dilation theory, standing outside this narrative, was proposed by Jean-Luc Sauvageot in \cite{Sauvageot}, \cite{SauvageotFirstExitTimes}, and \cite{SauvageotDirichletProblem}.  Writing during the nascence of free probability  (shortly after
the publication of \cite{VoiculescuSymmetries}, for instance), Sauvageot developed a modified version of the free product appropriate for use in dilation theory.  Since the Daniell-Kolmogorov construction (Example \ref{exmarkovdilation}) can be built using tensor products, which are the coproduct in the category of commutative unital C$^*$-algebras, and since free products play the corresponding role in the category of unital C$^*$-algebras, this is an attractively functorial way to conceptualize a noncommutative Markov process.  Using his version of the free product,
Sauvageot proved that every cp$_0$-semigroup on a C$^*$-algebra has a unital e$_0$-dilation.  This dilation theorem was then used to solve a Dirichlet problem for C$^*$-algebras, much as classical Brownian motion can be used to solve the classical Dirichlet problem (\cite{Kakutani}).

Sauvageot's theorem stands virtually alone in achieving a unital dilation; at some point, all the other dilation strategies mentioned here rely upon  the non-unital embedding of $B(H)$ into $B(K)$ for Hilbert spaces $H \subset K$.  However, although \cite{Sauvageot} asserts that his dilation technique can be modified to yield continuous dilations on W$^*$-algebras, little detail is given, and later authors indicate some uncertainty about this modification (e.g. \cite{SkeideDilationsProductSystems}).  Hence, given a CP$_0$-semigroup, it seems that one may be forced to choose either a \emph{unital} e$_0$-dilation or a \emph{continuous} (that is, E$_0$-) dilation.  The present thesis will expound Sauvageot's dilation techniques in order to demonstrate the possibility of achieving both objectives together (Theorem \ref{thmkahuna}).

\chapter{Liberation} \label{chapliberation}

\section{Introduction}
Free probability theory was introduced by Voiculescu in \cite{VoiculescuSymmetries}, as a tool to address the free group factor problem.  Free probability has since blossomed into its own area of study; its development has been an important success, even though the free group factor problem remains unresolved.  Sauvageot's \emph{ad hoc} modification of free probability, in contrast, does not appear to have inspired further pursuit beyond his first paper.  This could be due in part to the relevant free independence property remaining implicit in that paper, appearing only in the midst of the proof of Proposition 1.7.

In this chapter, Sauvageot's version of free independence, which I refer to as \textbf{liberation} (meant to suggest something similar to freeness; not to be confused with Voiculescu's use of
the same word in \cite{VoiculescuFisher6}) is studied in its own right.  As yet the only nontrivial liberated system I know of is the one originally used by Sauvageot in application to dilation theory.  However, I still consider it advantageous to separate this part of the exposition, both (i) to clarify the \emph{combinatorial} aspects of dilation, in contrast to its algebraic and analytic features, and (ii) to suggest possibilities for further investigation of connections with standard free probability theory.

\section{Background: Free Independence and Joint Moments}
We recall some of the basic notions of free
probability, which can be found in
references such as \cite{VoiculescuSymmetries}, \cite{FRV}, and \cite{NicaSpeicher}.

A \textbf{noncommutative probability space} is a pair $(\AAA, \phi)$ where $\AAA$ is a unital complex algebra
and $\phi: \AAA \to \com$ a unital linear functional.
Subalgebras $\{A_i\}_{i \in I}$ of $\AAA$ are
said to be \textbf{freely independent} with respect to $\phi$ if $\phi(a_{i_1} a_{i_2} \dots a_{i_n}) = 0$
whenever
\begin{itemize}
    \item $i_1, \dots, i_n$ are elements of $I$ such that adjacent indices are not equal, i.e. for $k = 1, \dots, n-1$
    one has $i_k \neq i_{k+1}$; this condition is abbreviated as $i_1 \neq i_2 \neq \dots \neq i_n$
    \item $a_{i_k} \in A_{i_k}$ for each $k = 1, \dots, n$
    \item $\phi(a_{i_k}) = 0$ for each $k = 1, \dots, n$.
\end{itemize}
Given noncommutative probability spaces $\{(A_i, \phi_i)\}$, a construction known as the
\textbf{free product} of unital algebras yields, in a universal (i.e.\ minimal) way, a noncommutative probability
space $(\AAA, \phi)$ and injections $f_i: A_i \to \AAA$ satisfying $\phi \circ f_i = \phi_i$,
such that the images $f_i(A_i)$ are freely independent with respect to $\phi$.  Furthermore, this
construction on unital algebras can be ``promoted'' to a construction on unital *-algebras or C$^*$-algebras; in the latter
case it is related to the free product of Hilbert spaces.

One implication of free independence which is essential
for our present purposes is that it determines
the value of $\phi$ on the subalgebra generated by
$\{A_i\}$.  Given $i_1 \neq i_2 \neq \dots \neq i_n$
and elements $a_{i_k} \in A_{i_k}$, one can
compute the \textbf{joint moment}
$\phi(a_{i_1} \dots a_{i_n})$ as follows:
\begin{itemize}
    \item \textbf{Center} each term $a_{i_k}$;
    that is, rewrite it as $\mathring{a}_{i_k}
    + \phi(a_{i_k}) \one$, where we define
    $\mathring{x} = x - \phi(x) \one$.
    \item \textbf{Expand} the product
    $(\mathring{a}_{i_1} + \phi(a_{i_1}) \one)
    \cdots (\mathring{a}_{i_n} + \phi(a_{i_n}) \one)$,
    thus obtaining a sum of $2^n$ words.
    \item \textbf{Simplify} by pulling out scalars:
    rewrite, for instance, $\mathring{a}_{i_1}
    \big( \phi(a_{i_2}) \one \big) \mathring{a}_{i_3}$
    as $\phi(a_{i_2}) \mathring{a}_{i_1} \mathring{a}_{i_3}$.
    \item After simplification, the only
    remaining word of length $n$ is the
    centered word
    $\mathring{a}_{i_1} \dots \mathring{a}_{i_n}$.
    Applying the procedure iteratively to all the smaller words that have been generated, one can rewrite
    the original word as a sum of many centered words, plus a word of length 0, i.e.\ a scalar.  Since $\phi$
    vanishes on centered words and is unital, its
    value at the original word is therefore whatever
    scalar is left when this iterative procedure terminates.
\end{itemize}
Using this outline, one can calculate
$\phi(a_{i_1} \dots a_{i_n})$ whenever
$i_1 \neq i_2 \neq \dots \neq i_n$.  Of course,
no generality is lost by this hypothesis, as
neighboring terms belonging to the
same subalgebra can be combined.

For use in proofs, it will be convenient to formalize
the above procedure in terms of a recursive definition.
I have not seen such a formalization
in the literature, so I present the following.
\begin{itemize}
    \item We use subset notation to
    indicate sub-tuples of an ordered tuple;
    thus, $(1,3) \subset (1,2,3,4,5)$
    and $(1,2,3,4,5) \setminus (1,3) = (2,4,5)$.
    We also use $[n]$ to denote the
    tuple $(1,2, \dots, n)$.  For a set $S$
    we use $S^\sharp = \bigcup_{n=1}^\infty S^n$
    to denote the set of all finite ordered
    tuples from $S$.

    \item Given a set $I$ and
    a tuple $\vec{\iota} \in I^n$,
    the \textbf{consecutivity tuples}
    of $\vec{\iota}$ are the maximal consecutive
    sub-tuples of $[n]$ such that $i_j$ is
    the same for all $j$ in such a tuple.  For
    instance, if $\vec{\iota} = (1,2,2,1,1,3,1)$ then
    the consecutivity tuples are
    $\vec{c}_1 = (1)$, $\vec{c}_2 = (2,3)$,
    $\vec{c}_3 = (4,5)$, $\vec{c}_4 = (6)$,
    and $\vec{c}_5 = (7)$.  We say that
    $\vec{\iota}$ is \textbf{nonstammering}
    if all its consecutivity tuples have
    length 1; this is another way of stating
    the condition $i_1 \neq i_2 \neq \dots
    \neq i_n$.

    \item Given a noncommutative probability
    space $(\AAA, \phi)$ and a nonempty
    set $I$, we recursively define two moment
    functions, the \textbf{alternating moment
    function} $\AM: \AAA^\sharp \to \com$
    and the \textbf{general moment function}
    $\GM: (I \times \AAA)^\sharp \to \com$.

    \item In the base case $n=1$
    we define $\AM(a_1) = \phi(a_1)$.

    \item Given $\vec{\iota} \in I^n$ with
    consecutivity tuples $\vec{c}_1,
    \dots, \vec{c}_\ell$, define
    \[
    \GM(\vec{\iota}; \vec{a}) = \AM
    \left(\prod_{j \in \vec{c}_1} x_j,
    \dots, \prod_{j \in \vec{c}_\ell} x_j \right).
    \]
    Note that this defines a general moment in
    terms of alternating moments of at most
    the same length.

    \item For $n > 1$ define
    \[
    \AM(a_1, \dots, a_n) = \sum_{\vec{j} \subsetneq [n]}
    \GM \Big(\vec{j}; (\mathring{a}_{j_1},
    \dots, \mathring{a}_{j_{|\vec{j}|}}) \Big)
    \prod_{k \in [n] \setminus \vec{j}}
    \phi(a_k)
    \]
    where $|\vec{j}|$ denotes the length of $\vec{j}$.
    By substituting the above definition
    of the general moment, we see that
    this defines an alternating moment in
    terms of alternating moments of strictly
    shorter words.

    \item Note that $\AM(a_1, \dots, a_n) = 0$
    whenever $\phi(a_1) = \dots = \phi(a_n) = 0$,
    as the sum is over proper subsets
    $\vec{j} \subsetneq [n]$ and hence the
    product $\displaystyle\prod_{k \in [n] \setminus \vec{j}}
    \phi(a_k)$ is always nonempty and therefore zero.

    \item Theorem: Let $\AAA, \phi, I$
    be as above, and $\{A_i\}_{i \in I}$ a
    $\phi$-freely independent family
    of unital subalgebras of $\AAA$.
    For $n \in \N$ and $(\vec{\iota}, \vec{a}) \in (I \times \AAA)^n$,
    we say that $\vec{\iota}$ \textbf{locates}
    $\vec{a}$ if $a_k \in A_{i_k}$ for
    each $k = 1, \dots, n$.  Let
    $\AAA_{\text{loc}} \subset (I \times \AAA)^\sharp$ consist of all $(\vec{\iota}, \vec{a})$
    such that $\vec{\iota}$ locates $\vec{a}$,
    and $\AAA_{\text{ns}} \subset \AAA^\sharp$
    consist of all $\vec{a}$ such that there
    exists a nonstammering tuple
    which locates $\vec{a}$.  Then
    \begin{enumerate}
        \item for any $(\vec{\iota}, \vec{a})
        \in \AAA_{\text{loc}}$,
        \[
        \phi(a_1 \dots a_n) = \GM(\vec{\iota}; \vec{a}).
        \]

        \item for any $\vec{a} \in \AAA_{\text{ns}}$,
        \[
        \phi(a_1 \dots a_n) = \AM(\vec{a}).
        \]
    \end{enumerate}
    The proof is by induction on $n$;  using the center-expand-simplify
    procedure outlined above, one
    can see that $\phi(a_1 \cdots a_n)$
    satisfies the same recurrence and initial
    conditions
    as (the relevant restrictions of)
    the functions $\GM$ and $\AM$.
\end{itemize}

\section{Defining Liberation}
We now develop two variations on
free independence, which will be
of use in dilation theory.

\begin{definition} \label{defliberated}
Let $\CC$ be a unital algebra, $\nu: \CC \to \com$
a unital linear functional, $\eE: \CC \to \CC$
a linear map.
Given a triple $(A, B, \rho)$ consisting of
unital subalgebras $A, B \subseteq \CC$
and a unital linear map $\rho: A \to B$
between them satisfying $\eE \circ \rho= \eE$,
we introduce the notation
$\mathring{a} = a - \rho(a)$ for elements
$a \in A$; note that in general $\mathring{a}$
 is neither an element of $A$ nor of $B$.
 We say the triple $(A, B, \rho)$ is:
\begin{enumerate}
    \item \textbf{right-liberated}
    (with respect to $\nu$ and $\eE$) if
    $\eE$ is a $B$-bimodule map,
    i.e. $\eE[b_1 x b_2] = b_1 \eE[x] b_2$
    for all $b_1, b_2 \in B$ and $x \in \CC$,
    and for every $n \geq 1$, every
    $a_1, \dots, a_n \in A$, and every
    $b_1, \dots, b_{n-1} \in B$ satisfying
    $\nu(b_1) = \dots = \nu(b_{n-1}) = 0$,
    \[
    \eE \Big[ \mathring{a}_1 b_1
    \dots \mathring{a}_n  \Big]=0.
    \]

    \item \textbf{left-liberated} if
    $\eE$ is an $A$-bimodule map and
    for every $n \geq 1$, every $a_1, \dots, a_{n-1} \in A$,
    and every $b_1, \dots, b_n \in B$
    satisfying $\nu(b_0) = \dots = \nu(b_n) = 0$,
    \[
    \eE \Big[ b_1 \mathring{a}_1
    b_2 \dots \mathring{a}_{n-1} b_n
    \Big] =0.
    \]
\end{enumerate}
\end{definition}

We note that the criteria in these definitions
resemble free independence, in that
the alternating product of centered terms
is centered.  The key difference, however,
 is that the centering takes place with respect
to several different maps---elements of $B$
are centered with respect to $\nu$, elements
of $A$ with respect to $\rho$, and the
alternating product with respect to $\eE$.

In some cases it will be useful to generalize
this definition.

\begin{definition} \label{defliberatingrep}
Let $A, B$ be unital algebras and $\rho: A \to B$
a unital linear map.  A \textbf{right-liberating representation} of the triple $(A, B, \rho)$
is a quintuple $(\AAA, f, g, \eE, \nu)$ where
\begin{itemize}
    \item $\AAA$ is a unital algebra
    \item $f: A \to \CC$ and $g: B \to \AAA$ are
    unital homomorphisms
    \item $\nu: \AAA \to \com$ is a unital linear functional
    \item $\eE: \AAA \to \AAA$ is a (not necessarily unital) linear map
\end{itemize}
satisfying the following criteria:
\begin{enumerate}
    \item $\eE \circ g \circ \rho = \eE \circ f$

    \item $\eE$ is a $g(B)$-bimodule map

    \item for every $n \geq 1$, every
    and $a_1, \dots, a_n \in A$,
    and every $b_1, \dots b_{n-1} \in B$ such
    that $\nu(g(b_1)) = \dots \nu(g(b_{n-1}))=0$,
    \[
    \eE \Big[ \big(f(a_1) - g(\rho(a_1)) \big)
    g(b_1) \cdots g(b_{n-1}) \big( f(a_n) - g(\rho(a_n)) \big)
     \Big] = 0
    \]
\end{enumerate}
and a \textbf{left-liberating representation} is such a quintuple satisfying
\begin{enumerate}
    \item $\eE \circ g \circ \rho = \eE \circ f$

    \item $\eE$ is an $f(A)$-bimodule map

    \item for every $n \geq 1$, every
    and $a_1, \dots, a_{n-1} \in A$,
    and every $b_1, \dots b_n \in B$ such
    that $\nu(g(b_1)) = \dots \nu(g(b_n))=0$,
    \[
    \eE \Big[ g(b_1) \big(f(a_1) - g(\rho(a_1)) \big)
     \cdots \big( f(a_{n-1}) - g(\rho(a_{n-1})) \big)
     g(b_n) \Big] = 0.
    \]
\end{enumerate}

\end{definition}

\begin{remark} \label{remreduceliberatingrep}
If there exists a map $\tilde{\rho}: f(A) \to g(B)$
with the property $\tilde{\rho} \circ f = g \circ \rho$,
then Definition (\ref{defliberatingrep}) reduces
to the statement that $(f(A), g(B), \tilde{\rho})$
is liberated in the appropriate sense from
Definition (\ref{defliberated}).  Such a
map $\tilde{\rho}$ need not exist in general,
but it does in two important special cases:
\begin{enumerate}
    \item If $f$ is injective, one may define
    $\tilde{\rho} = g \circ \rho \circ f\inv$.
    \item If $(\eE \circ g)$ is injective,
    one may define $\tilde{\rho} = g\circ
    (\eE \circ g)\inv \circ \eE$.  Then
    \[
    \tilde{\rho} \circ f = g \circ (\eE \circ g)\inv
    \circ \eE \circ f = g \circ (\eE \circ g)\inv \circ \eE \circ g \circ \rho = g \circ \rho.
    \]
\end{enumerate}
Both of these cases will be used subsequently.
\end{remark}

\section{Right Liberation and Joint Moments}
Like free independence, liberation is a property
that implies an algorithm.  The idea is the same---by centering, expanding, and simplifying, one can write any word as a centered word plus shorter words---but since the centering takes place with respect to three different maps, the details of the procedure are
more complicated.

Suppose $\CC$ is an algebra in which
$(A, B, \rho)$ is right-liberated
with respect to $\eE, \nu$ as above.  We continue
to use the notation $\mathring{a} = a - \rho(a)$
for $a \in A$.  Let $\la A, B \ra$ denote the subalgebra of $\CC$ generated by $A$
and $B$.  We consider two types of words
in $\la A, B \ra$:
\begin{enumerate}
    \item A word of the \textbf{first type} is
    of the form $b_0 a_1 b_1 \dots b_{\ell-1} a_\ell b_\ell$
    for some $a_1, \dots, a_\ell \in A$ and
    $b_0, \dots, b_\ell \in B$.

    \item A word of the \textbf{second type}
    is of the form $b_0 \mathring{a}_1
    b_1 \dots b_{\ell-1} \mathring{a}_\ell b_\ell$ for
    some $a_1, \dots, a_\ell \in A$
    and $b_0, \dots, b_\ell \in B$.
\end{enumerate}
Since $B$ is unital (so that we can take $b_0 = \one$ and/or $b_\ell = \one$), words
of the first type span $\la A, B \ra$.  However,
we will have use for words of both types.
We refer to the number
$\ell$ above as the \textbf{length} of the word; hence
a word of length zero is simply an element of $B$.
We say that a word of either type is in
\textbf{standard form} if $\nu(b_1) = \dots
= \nu(b_{\ell-1}) = 0$.

To calculate $\eE$ on a word
$b_0 a_1  \cdots  a_\ell b_\ell$ of first type,
we proceed thus:
\begin{itemize}
    \item Center the $b_i$ for $0 < i < \ell$,
    expand, and simplify.  Here ``center'' means
    to write $b_i$ as $\tilde{b}_i +
    \nu(b_i) \one$.

    The result of this step is a sum of standard-form
    words of the first type, each with length
    at most $\ell$.  The lengths of some words are less,
    because $a_i (\nu(b_i) \one) a_{i+1}$ is
    an element of $A$.

    \item For each of the resulting words, center
    the $a_i$, expand, and simplify.  Here ``center''
    means to write $a_i$ as $\mathring{a}_i
    + \rho(a_i)$.  Simplification can result
    in shorter words because $b_i \rho(a_{i+1})
    b_{i+1}$ is an element of $B$.

    The result of this step is a sum of words
    of the second type.  Not all of these words
    are in standard form, because simplification
    can create non-centered elements of $B$.
    However, all words are of length at most
    $\ell$, and (crucially) the only
    word of length $\ell$ is in standard
    form.

    \item The resulting words which are not
    in standard form can be rewritten as
    sums of words of the first type, by
    un-centering the $\mathring{a}_i$
    (that is, writing $\mathring{a}_i$
    as $a_i - \rho(a_i)$), expanding,
    and simplifying.  In the resulting
    sum of words of the first type, all
    have length strictly less than $\ell$.

    \item By iterating, this procedure
    allows us to write our original
    word $b_0 a_1 b_1 \dots a_\ell b_\ell$
    as a sum of standard-form words of the second
    type, plus words of length zero, which are
    just elements of $B$.  Since $\eE$
    vanishes on standard-form words of the
    second type and is a $B$-bimodule map,
    this determines $\eE[b_0 a_1 \dots a_\ell b_\ell]$
    in terms of $\eE[\one]$.
\end{itemize}

As in the case of ordinary free independence,
we can parlay this algorithm into a recursive
expression for joint moments.  As the procedure
is more complicated, however, we end up defining
three ``moment functions'' rather than two.

\begin{itemize}
    \item Given a noncommutative probability space $(\AAA, \nu)$,
    subalgebras $A,B \subset \AAA$, and a linear map $\rho: A \to B$,
    let
    \[
    \WW_\ell = \{(b_0, a_1, b_1, \dots, a_\ell, b_\ell)
    \mid a_1, \dots, a_\ell \in A; \ b_0, \dots, b_\ell \in B\}, \qquad \ell \geq 0
    \]
    denote the alternating tuples of length $2\ell+1$ which start and
    end with an element of $B$, and $\WW_I = \bigcup_{\ell=0}^\infty \WW_\ell$.
    This
    corresponds to
    the set of type I words as described above.
    More precisely, if we define the product
    function $\Pi: \WW_I \to \AAA$ by
    $\Pi(b_0, a_1, b_1, \ldots, a_\ell, b_\ell)
    = b_0 a_1 b_1 \cdots a_\ell b_\ell$, then
    its range $\Pi(\WW_I)$ is the set of
    type I words.

    \item Given $\ell \geq 1$,
    $\vec{x} \in \WW_\ell$, and a subset
    $\vec{\iota} \subset [\ell-1]$, 
    we define the \textbf{left collapse of $\vec{x}$ determined by $\vec{\iota}$},
    denoted $\LC(\vec{x}; \vec{\iota})$,
    as follows:  Let $\vec{\iota} = (i_1, \dots, i_m)$.  For each $k = 1, \dots, m+1$ define
    $\alpha_k = \prod\limits_{j=i_{k-1}+1}^{i_k} a_j$, where we adopt the convention $i_0 = 0$
    and $i_{m+1} = \ell$.

    Then $\LC(\vec{x}; \vec{\iota}) = (b_0, \alpha_1, b_{i_1}
    - \nu(b_{i_1}) \one, \dots, \alpha_m, b_{i_m} - \nu(b_{i_m}) \one, \alpha_{m+1}, b_\ell)$.
    This is the vector that results by taking $\vec{x}$, replacing those $b_j$ with $0 < j < \ell$
    and $j \in \vec{\iota}$ by $b_j - \nu(b_j) \one$,
    replacing those $b_j$ with $0 < j < \ell$
    and $j \notin \vec{\iota}$
    by $\one$, and then multiplying together adjacent terms from $A$.  It is an element
    of $\WW_{|\vec{\iota}|+1}$.

    \item Similarly, given $\vec{x} \in \WW_\ell$ and
    a subset $\vec{\iota} \subset [\ell]$, we define the
    \textbf{right collapse of $\vec{x}$ determined by $\vec{\iota}$}, an element of $\WW_{|\vec{\iota}|}$
    denoted $\RC(\vec{x}; \vec{\iota})$, as follows: Let
    $\vec{\iota} = (i_1, \dots, i_m)$.  For each
    $k = 0, \dots, m$ define $\beta_k = b_{i_k} \prod\limits_{j=i_k+1}^{i_{k+1}-1} \rho(a_j) b_j$,
    where we retain the convention $i_0 = 0$ but
    now set $i_{m+1} = \ell+1$.
    Then $\RC(\vec{x}; \vec{\iota}) = (\beta_0, a_{i_1}, \beta_1, \dots, a_{i_m}, \beta_m)$.
    This is the vector that results by taking $\vec{x}$, replacing those $a_j$ with
    $j \notin \vec{\iota}$ by $\rho(a_j)$, and then multiplying together adjacent
    terms from $B$.

    \item Given $\vec{y} \in \WW_\ell$ and
    a subset $\vec{\iota} \subset [\ell]$, we define the
    \textbf{un-collapse of $\vec{x}$ determined by $\vec{\iota}$}, an element of $\WW_{|\vec{\iota}|}$ denoted
    $\UC(\vec{x}; \vec{\iota})$, as follows: Let $\vec{\iota} = (i_1, \dots, i_m)$.
    For each $k = 0, \dots, m$ define $\beta_k= b_{i_k} \prod\limits_{j=i_k+1}^{i_{k+1}-1}
    (-\rho(a_j)) b_j$, where we continue to interpret
    $i_0$ as $0$ and $i_{m+1}$ as $\ell+1$.  Then
    $\UC(\vec{x}; \vec{\iota}) = (\beta_0, a_{i_1}, \beta_1, \dots, a_{i_m}, \beta_m)$.
    This is the vector that results by taking $(b_0, a_1,
    b_1, \dots, a_\ell, b_\ell)$, replacing each $a_j$
    with $j \notin \vec{\iota}$ by $-\rho(a_j)$, and then multiplying together
    adjacent terms from $B$.

    \item We define three functions $\LM, \RM, \UM: \WW_I \to B$,
    which we call the left-centering moment function, the right-centering moment
    function, and the un-centering moment function, by
    \begin{align}
        \LM(b_0) &= \RM(b_0) = \UM(b_0) = b_0,  \label{momentsfirstline}\\
    \LM(\vec{x}) &= \sum_{\vec{\iota} \, \subseteq [\ell]} \RM(\LC(\vec{x}; \vec{\iota})) \prod_{j \in [\ell] \setminus \vec{\iota}} \nu(x_{2j+1}), \\
    \RM(\vec{x}) &= \sum_{\vec{\iota} \, \subsetneq [\ell+1]} \UM(\RC(\vec{x}; \vec{\iota})),\\
    \UM(\vec{x}) &= \sum_{\vec{\iota} \, \subseteq [\ell+1]} \LM(\UC(\vec{x}; \vec{\iota})), \label{momentslastline}
    \end{align}
    for $\vec{x} \in \WW_{\ell+1}$.
    This produces a well-defined recursion because of the strict subset
    inclusion $\vec{\iota} \subsetneq [\ell+1]$ in the definition of $\RM$, so that
    one obtains moments of strictly shorter words.
\end{itemize}

\begin{remark} \label{remnumberterms}
Note that evaluating the $\LM$ function
    on a word of length $\ell$ returns a sum of $2^{\ell-1}$ evaluations of
    the $\RM$ function on words of length up to $\ell$; on a word of
    size $k \leq \ell$, the $\RM$ function returns a sum of $2^{k-1}$ evaluations
    of the $\UM$ function on words of length strictly less than $k$; and on a word
    of size $j < k$, the $\UM$ function returns a sum of $2^j$ evaluations of
    the $\LM$ function on words of length up to $j$.  This implies that the number
    of terms in the evaluation of the $\LM$ on words of length $\ell$ is
    bounded above by the sequence $\{s_\ell\}$ determined by $s_0 = 1$
    and $s_{\ell+1} = 8^{\ell+1} s_\ell$, which has
    the closed form $s_\ell = 8^{\ell(\ell+1)/2}$.
    Of course the actual number of terms is considerably less, due both to cancellation and
    to the fact that this estimate treats all words of length less than $\ell$
    as if they had length $\ell$.
\end{remark}

\begin{theorem} \label{thmrightlibmoments}
Let $(A, B, \rho)$ be right-liberated in $\AAA$
with respect to $\eE, \nu$.  Then
for any $\vec{x} \in \WW_I$,
\[
\eE\left[ \Pi(\vec{x}) \right] = \LM(\vec{x}) \eE[\one].
\]
\end{theorem}

\begin{proof}
We prove this for $\vec{x} \in \AAA_\ell$
by induction on $\ell$.  The base case
$\ell=0$ is trivial, and in general,
the center-expand-simplify procedure
shows that $\eE[ \Pi( \vec{x})]$ satisfies
the same recursion as $\LM(\vec{x}) \eE[\one]$.
The strict subset inclusion in the definition
of $\RM$ arises because the $\vec{\iota}=[\ell]$
term corresponds to a centered word, which
vanishes when $\eE$ is applied.
\end{proof}

\begin{corollary} \label{corrightlibgenerated}
Let $(A, B, \rho)$ be right-liberated in
$\AAA$ with respect to $\eE, \nu$.  Let
$\la A, B \ra$ denote the subalgebra of $\AAA$
generated by $A$ and $B$.  Then
\[
\eE \Big[ \la A, B \ra \Big] = \eE \big[ B \big].
\]
\end{corollary}

The obvious generalizations of Theorem
(\ref{thmrightlibmoments}) and Corollary (\ref{corrightlibgenerated}) to right-liberating representations is true as well, and are verified
inductively in the same manner.  We record
them here without proof.

\begin{theorem} \label{thmrightlibrepmoments}
Let $(\AAA, f,g, \eE, \nu)$
be a right-liberating representation
of $(A, B, \rho)$.  For
$\vec{x} = (b_0, a_1, b_1, \dots, a_\ell, b_\ell)
\in \WW_I$, let $(f \times g)(\vec{x})$
denote the element \\
 ${g(b_0) f(a_1) g(b_1)
\dots f(a_\ell) g(b_\ell) \in \AAA}$.
Define the functions $\LM_r, \RM_r, \UM_r: \WW_I \to B$ as in equations (\ref{momentsfirstline})-(\ref{momentslastline}),
except with $\nu$ replaced by $\nu \circ g$.

Then for any $\vec{x} \in \WW_I$,
\[
\eE[ (f \times g)(\vec{x})] = \LM_r(\vec{x}) \eE[\one].
\]
\end{theorem}

\begin{corollary} \label{corrightlibrepgenerated}
Let $(\AAA, f,g, \eE, \nu)$
be a right-liberating representation
of $(A, B, \rho)$.   Let
$\la A, B \ra$ denote the subalgebra of $\AAA$
generated by $f(A)$ and $g(B)$.  Then
\[
\eE \Big[ \la A, B \ra \Big] = \eE \big[ g(B) \big].
\]
\end{corollary}

Later we shall be interested in the continuity
properties of joint moments.  We record here the following simple observation:

\begin{proposition}  \label{propnormalmoments}
Let $(A, B, \rho)$ be right-liberated in $\AAA$ with
respect to $(\eE, \nu)$, where $\AAA$ is a $W^*$-algebra and $\rho, \eE, \nu$ are all normal.  Then
\begin{enumerate}
    \item For any $\vec{x} \in \WW_I$, $\LM(\vec{x})$
    is normal in each entry of $x$.  That is,
    given $\ell \geq 0$, $\vec{x} \in \WW_\ell$,
    and $1 \leq j \leq 2\ell+1$, let $x_k$
    be fixed for all $1 \leq k \leq 2\ell+1$ with $k \neq j$;
    then $\LM(\vec{x})$, viewed as a function
    of $x_j$, is a normal linear map from $A$ or $B$
    (depending on the parity of $j$) to $\AAA$.

    \item If $\rho$ is strongly continuous
    on the unit ball $A_1$, then
    $\LM(\vec{x})$ is jointly strongly continuous in the
    entries of $\vec{x}$ on bounded subsets.  That is,
    the corresponding map
    $A_1 \times B_1 \times \dots \times A_1
    \to \AAA$ is strongly continuous.
\end{enumerate}
\end{proposition}

The proof is a straightforward induction
on $\ell$.  For the second part we use the fact
that $\eE$ and $\nu$, being completely positive and normal, are therefore also strongly continuous (\cite{Blackadar} III.2.2.2),
and that multiplication is jointly strongly continuous on the unit ball.

Later we will need to consider moments with respect to several
maps.  When need arises, we use $\LM(\vec{x}; \rho)$ in place of $\LM(\vec{x})$ for specificity.
\begin{proposition} \label{propmomentwrtsemigroup}
Let $\AAA$ be a $W^*$-algebra, $\eE$ and $\nu$ normal,
$A \subset \AAA$ a $W^*$-subalgebra, and $\{\phi_t\}_{t \geq 0}$ a CP-semigroup on $A$.  Then for each fixed
$\vec{x} \in \WW_I$, $\LM(\vec{x}; \phi_t)$
is strongly continuous in $t$.
\end{proposition}

\begin{proof}
As discussed in section \ref{secCPsemigroupcontinuity},
$t \mapsto \phi_t(a)$ is strongly continuous for
fixed $a \in A$.  The result is now a straightforward
induction, using this fact plus the
joint strong continuity of multiplication on the
unit ball of $A$.
\end{proof}

\section{Left Liberation and Joint Moments}
Unlike right liberation and strong right
liberation, left liberation is a condition
involving an unspecified scalar;
as this scalar can propagate through the recursion,
calculation of moments is impossible.
However, we can still draw some important conclusions.

We use a similar definition of words of the first and second types, but with the roles of $A$ and $B$ interchanged;
that is, our new type I words are
of the form $a_0 b_1 a_1 \dots b_\ell a_\ell$,
and our new type II words of the form
$a_0 b_1 \mathring{a}_1 b_2 \mathring{a}_2
\dots b_{\ell-1} \mathring{a}_{\ell-1} b_\ell
a_\ell$.  A word of either type is in
\textbf{standard form} if $\nu(b_1) =
\dots = \nu(b_\ell) = 0$.

The new center-expand-simplify strategy is as
follows:
\begin{itemize}
    \item Given a word of type I, center the
    $b_i$, expand, and collapse.
    The result is a sum of type I words
    in standard form
    of length at most $\ell$.

    \item Center the $a_i$ for $0 < i < \ell$,
    expand, and collapse.
    The result is a sum of type II words of length
    at most $\ell$, such that the only
    word of length $\ell$ is in standard form.

    \item For remaining words not in standard form,
    un-center the $a_i$, then apply the same procedure
    to the type I words that result.
\end{itemize}

We formalize this strategy in the following definitions:
\begin{itemize}
    \item For each $\ell \geq 0$ let
\[
\widetilde{\WW}_\ell = \{(a_0, b_1, a_1,
\dots, b_\ell, a_\ell) \mid
a_0, \dots, a_\ell \in A; \
b_1, \dots, b_\ell \in B\}
\]
and let $\widetilde{\WW}_I =\bigcup_{\ell=0}^\infty
\widetilde{\WW}_\ell$.  Define the product
function $\widetilde{\Pi}: \widetilde{\WW}_I \to \AAA$
by\\ $\widetilde{\Pi}(a_0, b_1, a_1, \dots, b_\ell,a_\ell)
= a_0 b_1 a_1 \dots b_\ell a_\ell$.
Note that the span of $\widetilde{\Pi}(\widetilde{\WW}_I)$
is $\la A, B \ra$.

    \item Given $\ell \geq 1$, $\vec{x} \in \widetilde{\WW}_\ell$, and a subset $\vec{\iota} \subset [\ell]$, suppose that $\vec{\iota} = (i_1, \dots, i_m)$.
        For ecah $k = 0, \dots, m$ define
        $\alpha_k =\displaystyle\prod_{j=i_k}^{i_{k+1}-1} a_j$, with the conventions $i_0 = 0$ and $i_{m+1}=\ell+1$.  Then we define
        \[
        \LC'(\vec{x}; \vec{\iota}) =
        (\alpha_0, b_{i_1} - \nu(b_{i_1}) \one, \alpha_1,
        \dots, b_{i_m} - \nu(b_{i_m}) \one, \alpha_m).
        \]
        This is the vector that results by starting
        with $\vec{x}$, replacing those $b_j$ with
        $j \in \vec{\iota}$ by $b_j - \nu(b_j) \one$,
        and the others by $\one$,
        and multiplying together adjacent terms
        from $A$.  It is an element of $\widetilde{\WW}_{|\vec{\iota}|}$.

    \item Given $\vec{x} \in \widetilde{\WW}_\ell$ and a subset
    $\vec{\iota} \subset [\ell-1]$, suppose
    $\vec{\iota} = (i_1, \dots, i_m)$.  For each
    $k = 1, \dots, m+1$ let
    $\beta_k = b_{i_k} \displaystyle\prod_{j=i_k}^{i_{k+1}-1} \rho(a_j) b_{j+1}$, and define
    \[
    \RC'(\vec{x}; \vec{\iota}) =
    (a_0, \beta_1, a_{i_1}, \dots, a_{i_m}, \beta_m, a_\ell).
    \]
    This is the vector that results from starting with
    $\vec{x}$, replacing those $a_j$ with $j \notin \vec{\iota}$
    by $\rho(a_j)$, and then multiplying together adjacent terms from $B$.  It is an element of $\widetilde{\WW}_{|\vec{\iota}|+1}$.

    \item As before, the un-collapse is defined in the same way except with a minus sign; that is,
        redefine $\beta_k = b_{i_k} \displaystyle\prod_{j=i_k}^{i_{k+1}-1} (-\rho(a_j)) b_{j+1}$ and then let
         \[
    \UC'(\vec{x}; \vec{\iota}) =
    (a_0, \beta_1, a_{i_1}, \dots, a_{i_m}, \beta_m, a_\ell).
    \]

    \item Define $\LM', \RM', \UM': \widetilde{\WW}_I \to A$
    by
    \begin{align}
    \LM'(a_0) &= \RM'(a_0) = \UM'(a_0) = a_0, \label{leftmomentsfirstline}\\
    \LM'(\vec{x}) &= \sum_{\vec{\iota} \subset [\ell+1]}
    \RM'(\LC'(\vec{x}; \iota)) \prod_{j \in [\ell+1] \setminus \vec{\iota}} \nu(x_{2j}),\\
    \RM'(\vec{x}) &= \sum_{\vec{\iota} \subsetneq [\ell]}
    \UM'(\RC'(\vec{x}; \vec{\iota})),\\
    \UM'(\vec{x}) &= \sum_{\vec{\iota} \subset [\ell]}
    \LM'(\UC'(\vec{x}; \vec{\iota})) \label{leftmomentslastline}
    \end{align}
    for $\vec{x} \in \widetilde{\WW}_{\ell+1}$.
\end{itemize}

We then have the following analogues of our prior results.

\begin{theorem} \label{thmleftlibmoments}
Let $(A, B, \rho)$ be left-liberated in $\AAA$
with respect to $\eE, \nu$.  Then
for any $\vec{x} \in \widetilde{\WW}_I$,
\[
\eE\left[ \widetilde{\Pi}(\vec{x}) \right] = \LM'(\vec{x}) \eE[\one].
\]
\end{theorem}

\begin{proof}
The center-expand-simplify procedure
shows that $\eE[ \widetilde{\Pi}( \vec{x})]$ satisfies
the same recursion as $\LM'(\vec{x}) \eE[\one]$.
\end{proof}

\begin{corollary} \label{corleftlibgenerated}
Let $(A, B, \rho)$ be left-liberated in
$\AAA$ with respect to $\eE, \nu$.  Let
$\la A, B \ra$ denote the subalgebra of $\AAA$
generated by $A$ and $B$.  Then
\[
\eE \Big[ \la A, B \ra \Big] = \eE \big[ A \big].
\]
\end{corollary}

\begin{theorem} \label{thmleftlibrepmoments}
Let $(\AAA, f,g, \eE, \nu)$
be a left-liberating representation
of $(A, B, \rho)$.  For
$\vec{x} = (a_0, b_1, a_1, \dots, b_\ell, a_\ell)
\in \WW_I$, let $(f \times g)(\vec{x})$
denote the element \\
 ${f(a_0) g(b_1) f(a_1)
\dots g(b_\ell) f(a_\ell) \in \AAA}$.
Define the functions $\LM'_r, \RM'_r, \UM'_r: \WW_I \to B$ as in equations (\ref{leftmomentsfirstline})-(\ref{leftmomentslastline}),
except with $\nu$ replaced by $\nu \circ g$.

Then for any $\vec{x} \in \widetilde{\WW}_I$,
\[
\eE[ (f \times g)(\vec{x})] = \LM'_r(\vec{x}) \eE[\one].
\]
\end{theorem}

\begin{corollary} \label{corleftlibrepgenerated}
Let $(\AAA, f,g, \eE, \nu)$
be a left-liberating representation
of $(A, B, \rho)$.   Let
$\la A, B \ra$ denote the subalgebra of $\AAA$
generated by $f(A)$ and $g(B)$.  Then
\[
\eE \Big[ \la A, B \ra \Big] = \eE \big[ f(A) \big].
\]
\end{corollary}


\chapter{The Sauvageot Product} \label{chapsauvageotproduct}

\section{Introduction}
In this chapter we develop a modification of the unital free product of C$^*$-algebras, adapted for use in dilation theory.  As mentioned in example (\ref{exmarkovdilation}), the classical Daniell-Kolmogorov construction can be reduced to the construction of maps $\theta_t: C(S) \otimes C(S) \to C(S)$
given on simple tensors by $\theta_t(f \otimes g) = (P_t f) g$.  We shall return to the details of this reduction in chapter \ref{chapiteratedproducts}; at present we only describe enough of its features to see what we shall need for the appropriate noncommutative analogue.

Among the many embeddings of $C(S)$ into $C(S) \otimes C(S)$ we distinguish
two, the ``left'' embedding $f \mapsto f \otimes \one$ and the ``right''
embedding $f \mapsto \one \otimes f$.  The map $\theta_t$ is a retraction with respect to the right embedding, and its composition with the left embedding is $P_t$.  That is, by constructing $\theta_t$ we factor $P_t$ into an embedding followed by a retraction (with respect to a different embedding), as
depicted in the following diagram:
\[ \xymatrix{
    C(S) \otimes C(S) \ar@/^1pc/[rd]^{\theta_t}\\
    C(S) \ar[u] \ar[r]_{P_t} & C(S) \ar@{.>}[lu]
    }  \quad \theta_t(f \otimes g) = P_t(f) g
    \]
More generally, the inductive process will work
with tensor powers $C(S)^{\otimes \gamma}$ for finite
sets $\gamma \subset [0,\infty)$, building for each one
a retraction $\epsilon_\gamma: C(S)^{\otimes \gamma} \to C(S)$.
Given $\gamma' = \gamma \cup \{t_k\}$, where $\tau = t_k - \underset{t \in \gamma}{\min} \, t > 0$, we will seek to define $\epsilon_{\gamma'}$ such that
\[ \xymatrix{
   C(S)^{\otimes \gamma'} \ar@/^1pc/[rd]^{\epsilon_{\gamma'}}\\
    C(S)^{\otimes \gamma} \ar[u] \ar[r]_{P_\tau} & C(S) \ar@{.>}[lu]
    }  \quad \epsilon_{\gamma'}(f \otimes g) = P_\tau (\epsilon_\gamma(f)) g
    \]
We note in passing that Stinespring dilation produces a very similar
diagram: Given a unital completely map $\phi: A \to B(H)$ with
minimal Stinespring triple $(K, V, \pi)$, we obtain
\[ \xymatrix{
    B(K) \ar@/^1pc/[rd]^{\theta}\\
    A \ar[u] \ar[r]_{\phi} & B(H) \ar@{.>}[lu]
    }  \quad \theta(T) = V^* T V
    \]
Crucially, however, the right embedding in this case is the non-unital
map $X \mapsto V X V^*$, in contrast to the unital embedding in the
commutative example.  Hence, we take the tensor product as our model in what follows.

In addition to constructing tensor products $C(X) \otimes C(Y) \simeq C(X \times Y)$ of commutative unital C$^*$-algebras, one can also form tensor products of maps between them, and the resulting maps satisfy certain functorial properties. We summarize the properties of the tensor product which we shall seek to replicate in this chapter:
\begin{enumerate}
    \item Given unital C$^*$-algebras $A,B$ and a unital completely
    positive map $A \sa{\phi} B$, we construct a unital C$^*$-algebra
    $A \star B$ with unital embeddings of $A$ and $B$, the images
    of which generate $A\star B$.
    \item We also construct a retraction $A \star B \to B$ which
    factors $\phi$ in the sense of the above diagrams.
    \item Given unital completely positive maps $A \sa{\phi} B$
    and $C \sa{\psi} D$, and given unital *-homomorphisms
    $A \sa{f} C$ and $B \sa{g} D$ such that the square
    \[ \xymatrix{
    C \ar[r]^\psi & D \\
    A \ar[r]_\phi \ar[u]_f & B \ar[u]_g
    }\]
    commutes, we construct a (necessarily unique) unital *-homomorphism\\
    $f \star g: A \star B \to C \star D$ such that the squares
    \[ \xymatrix{
    A \star B \ar[r]^{f \star g} &C \star D \\
    A \ar[u] \ar[r]_f & C \ar[u]
    } \qquad \qquad \xymatrix{
    A \star B \ar[r]^{f \star g} &C \star D \\
    B \ar[u] \ar[r]_g & D \ar[u]
    }\]
    commute.
\end{enumerate}

We now begin our development of a construction satisfying these
requirements.

\section{Sauvageot Products of Hilbert Spaces and Bounded Operators}
Just as the free product of unital C$^*$-algebras can be constructed
from a free product of Hilbert spaces (\cite{VoiculescuSymmetries}),
our product construction on C$^*$-algebras will rely on an underlying
construction on Hilbert space.  Some notational preliminaries:
For a Hilbert space $H$, $H^+$ denotes $H \oplus \com$,
and if a unit vector has been distinguished, $H^-$ denotes the
complement of its span.  A distinguished unit vector (such as
$1 \in \com$ as an element of the direct sum $H\oplus \com$)
is generally denoted by $\Omega$.  We also follow the convention
(most common in physics and in Hilbert C$^*$-modules) that inner
products are linear in the second variable.

\begin{definition} \label{defsauvageotproducthilbertspaces}
Let $\HH$ and $\LL$ be Hilbert spaces.  The \textbf{Sauvageot
product} $\HH \star \LL$ is the space
\[
\HH \star \LL = \HH^+ \oplus \bigoplus_{n=0}^\infty \Big[
\big(\LL^{+ \otimes n} \otimes \LL \big) \oplus
\big(\HH \otimes \LL^{+ \otimes n} \otimes \LL \big) \Big]
\]
with the convention $\LL^{+ \otimes 0} \otimes \LL = \LL$.
\end{definition}

Though defined as a direct sum, the Sauvageot product of Hilbert
spaces may also be viewed as an infinite tensor product, as
expressed in the following proposition.

\begin{proposition} \label{prop3UE}
Let $\HH$ and $\LL$ be Hilbert spaces, and $\KK = \HH^+ \oplus \LL$.  Denote by $\LL^{+ \otimes \N}$ the infinite tensor power of $\LL^+$ with respect to $\Omega$.  Then there are unitary
equivalences between $\HH \star \LL$
and both $\HH^+ \otimes \LL^{+ \otimes \N}$ and
$\KK \otimes \LL^{+ \otimes \N}$, under which
\begin{itemize}
    \item the subspace $\HH \otimes \LL^{+ \otimes \N}$ of $\HH^+ \otimes \LL^{+ \otimes \N}$
    is identified with the subspace\\ $\displaystyle{\bigoplus_{n=0}^\infty \HH \otimes \LL^{+ \otimes n} \otimes \LL}$
    of $\HH \star \LL$
    \item the subspace $\com \otimes \LL^{+ \otimes \N}$ of $\HH^+ \otimes \LL^{+ \otimes \N}$
    is identified with the subspace\\ $\displaystyle\bigoplus_{n=0}^\infty \LL^{+ \otimes n} \otimes \LL$
    of $\HH \star \LL$
    \item the subspace $\HH \otimes \LL^{+ \otimes \N}$ of $\KK \otimes \LL^{+ \otimes \N}$
    is identified with the subspace\\ $\HH \
    \oplus \displaystyle\bigoplus_{n=0}^\infty \HH \otimes \LL^{+ \otimes n} \otimes \LL$
    of $\HH \star \LL$
    \item the subspace $\com \otimes \LL^{+ \otimes \N}$ of $\KK \otimes \LL^{+ \otimes \N}$
    is identified with the subspace\\ $\LL \oplus \displaystyle\bigoplus_{n=1}^\infty \com \otimes \LL^{+ \otimes (n-1)}
    \otimes \LL$ of $\HH \star \LL$
    \item the subspace $\LL \otimes \LL^{+ \otimes \N}$ of $\KK \otimes \LL^{+ \otimes \N}$
    is identified with the subspace\\ $\displaystyle\bigoplus_{n=1}^\infty \LL \otimes \LL^{+ \otimes (n-1)} \otimes \LL$
    of $\HH \star \LL$
\end{itemize}
\end{proposition}

\begin{proof}
We will use the unitary equivalences $\LL^{+ \otimes \N} \simeq \LL^+ \otimes \LL^{+ \otimes \N}$, which is evident, and $\LL^{+ \otimes \N} \simeq \com \oplus \bigoplus_{n=0}^\infty \LL^{+ \otimes n} \otimes \LL$,
which follows from the construction of the infinite tensor power: If
$V_{\infty, n}: \LL^{+ \otimes n} \to \LL^{+ \otimes \N}$ denotes the limit map, then $\LL^{+ \otimes \N} = \overline{\bigcup_{n=0}^\infty V_{\infty, n}
\LL^{+ \otimes n}}$,
and in order to get a direct sum we orthogonalize using \\ $\left(V_{\infty,n+1} \LL^{+ \otimes(n+1)} \right)
\ominus \left(V_{\infty, n} \LL^{+ \otimes n} \right)
\simeq \LL^{+ \otimes n} \otimes \LL$.  Repeated application of these equivalences plus
the associative, commutative, and distributive laws
\[
H \otimes (K_1 \otimes K_2) \simeq (H \otimes K_1) \otimes K_2, \quad
H \otimes K \simeq K \otimes H, \quad
H \otimes (K_1 \oplus K_2) \simeq (H \otimes K_1) \oplus (H \otimes K_2)
\]
and the identity $H \otimes \com \simeq H$
yield
\begin{align*}
\KK \otimes \LL^{+ \otimes \N} &= (\HH^+ \oplus \LL) \otimes \LL^{+ \otimes \N} \simeq (\HH \oplus \LL^+)\otimes \LL^{+ \otimes \N}
\simeq (\HH \otimes \LL^{+ \otimes \N}) \oplus (\LL^+ \otimes
\LL^{+ \otimes \N})\\
&\simeq (\HH \otimes \LL^{+ \otimes \N}) \oplus
\LL^{+ \otimes \N}
\simeq (\HH \otimes \LL^{+ \otimes \N}) \oplus
(\com \otimes \LL^{+ \otimes \N}) \simeq \HH^+ \otimes \LL^{+ \otimes \N}
\end{align*}
and
\begin{align*}
\HH \star \LL &= \HH^+ \oplus \bigoplus_{n=0}^\infty
\left(\HH \otimes \LL^{+ \otimes n} \otimes \LL
\oplus \LL^{+ \otimes n} \otimes \LL \right)\\
&\simeq (\HH^+ \otimes \com) \oplus \bigoplus_{n=0}^\infty
\left(\HH \otimes \LL^{+ \otimes n} \otimes \LL
\oplus \com \otimes \LL^{+ \otimes n} \otimes \LL \right)\\
&\simeq (\HH^+ \otimes \com) \oplus \bigoplus_{n=0}^\infty
(\HH^+ \otimes \LL^{+ \otimes n} \otimes \LL)
\simeq \HH^+ \otimes \left(\com \oplus \bigoplus_{n=0}^\infty
\LL^{+ \otimes n} \otimes \LL\right) \simeq
\HH^+ \otimes \LL^{+ \otimes \N}.
\end{align*}
The specific identifications arise by following subspaces through these equivalences.
\end{proof}

As a simple corollary, we obtain the folllowing identifications:
\begin{proposition} \label{propsauvageotproductsHS}
Let $\HH, \HH_1, \HH_2, \LL$ be Hilbert spaces.
\begin{enumerate}
    \item $(\HH_1 \oplus \HH_2)^+ \star
    \LL \simeq (\HH_1 \star \LL) \oplus
    (\HH_2 \star \LL)$
    \item $\HH \star \{0\} \simeq \HH^+$
    \item $\{0\} \star \LL \simeq \LL^{+ \otimes \N}$
\end{enumerate}
\end{proposition}

Our next goal is to define the product of maps between Hilbert
spaces.

\begin{definition} \label{defSPops}
Let $\HH_1, \HH_2, \LL_1, \LL_2$ be Hilbert spaces,
$\KK_1 = \HH_1^+ \oplus \LL_1$ and $\KK_2 = \HH_2^+
\oplus \LL_2$, $S: \HH_1^+ \to \HH_2^+$ and $T: \KK_1 \to \KK_2$ bounded maps, and $V: \LL_1 \to \LL_2$ a contraction.
Define the bounded maps $S \star V: \HH_1 \star \LL_1
\to \HH_2 \star \LL_2$ and $T \hstar V:
\HH_1\star \LL_1 \to \HH_2 \star \LL_2$ as follows: Let $V^+
= V \oplus \dss{\text{id}}{\com}: \LL_1^+ \to \LL_2^+$
and let $V^{+ \otimes \N}: \LL_1^{+ \otimes \N} \to \LL_2^{+ \otimes \N}$
be the limit
of the contractions $V^{+\otimes n}: \LL_1^{+ \otimes n}
\to \LL_2^{+ \otimes n}$.  Then
$S \star V: \HH_1 \star \LL_1 \to \HH_2 \star \LL_2$ is the
operator $S \otimes V^{+ \otimes \N}:
\HH_1^+ \otimes \LL_1^{+ \otimes \N}
\to \HH_2^+ \otimes \LL_2^{+ \otimes \N}$ composed
with the unitary equivalences of Proposition \ref{prop3UE};
similarly, $T \hstar V$ is the operator $T \otimes V^{+ \otimes \N}:
\KK_1 \otimes \LL_1^{+ \otimes \N} \to
\KK_2 \otimes \LL_2^{+ \otimes \N}$ composed with the appropriate
unitary equivalences.
\end{definition}

By following the sequence of maps in the proof
of Proposition \ref{prop3UE}, we can calculate
how product maps act on the various summands
of $\HH_1 \star \LL_1$.

\begin{proposition} \label{propSPopsummands}
Let $\HH_1, \HH_2, \LL_1, \LL_2$ be Hilbert spaces,
and for $i = 1,2$ let $\KK_i = \HH_i^+ \oplus \LL_i$.
Let $\HH_1^+ \sa{S} \HH_2^+$ and $\KK_1 \sa{T} \KK_2$ be
bounded operators and $\LL_1 \sa{V} \LL_2$ a contraction.
For each $n \geq 0$ let $V^{(n)}$ denote
$V^{+\otimes n} \otimes V: \LL_1^{+ \otimes n}
\otimes \LL_1 \to \LL_2^{+ \otimes n} \otimes \LL_2$.

Let $h \in \HH_1^+$, $h_0 \in \HH_1$,
$k \in \KK_1$, $\ell \in \LL_1^+$, and $\xi \in \LL_1^{+ \otimes n} \otimes \LL_1$ for some $n \geq 0$, and suppose that
\begin{align*}
S \Omega_1 &= \alpha \Omega_2 + y, \qquad \qquad \alpha \in \com, y \in \HH_2\\
S h_0 &= \beta \Omega_2 + z, \qquad \qquad \beta \in \com, z \in \HH_2\\
T h_0 &= \eta + w, \qquad \qquad \eta \in \HH_2, w \in \LL_2^+\\
T\ell &= \zeta+u, \qquad \qquad \zeta \in \HH_2, u \in \LL_2^+.
\end{align*}
Then
\begin{align*}
(S \star V) h &= Sh\\
(S \star V)\xi &= \alpha V^{(n)} \xi + (y \otimes V^{(n)} \xi)\\
(S \star V)(h_0 \otimes \xi) &= \beta V^{(n)}\xi
+ (z \otimes V^{(n)}\xi)\\
(T \hstar V) k &= Tk \\
(T \hstar V)(h_0 \otimes \xi) &= (\eta \otimes V^{(n)}\xi) + (w \otimes V^{(n)}\xi)\\
(T \hstar V)(\ell \otimes \xi) &= (\zeta \otimes V^{(n)}\xi) + (u \otimes V^{(n)}\xi).
\end{align*}
\end{proposition}

Next, we develop some of the essential properties of
this construction.

\begin{proposition} \label{propSPopproperties}
Let $\HH_1^+ \sa{S} \HH_2^+ \sa{S'} \HH_3^+$
and $\KK_1 \sa{T} \KK_2 \sa{T'} \KK_3$ be bounded
maps, and $\LL_1 \sa{V} \LL_2 \sa{V'} \LL_3$ contractions.
\begin{enumerate}
    \item $(S' \star V') \circ (S \star V) =
    (S' \circ S) \star (V' \circ V)$ and
    $(T' \hstar V') \circ (T \hstar V)
    = (T' \circ T) \hstar (V' \circ V)$.
    \item If $S$ is the identity map on $\HH_1 = \HH_2$, and
    $V$ the identity map on $\LL_1 = \LL_2$,
    then $S \star V$ is the identity map on
    $\HH_1 \star \LL_1 = \HH_2 \star \LL_2$.  Similarly, if $T$
    and $V$ are the appropriate identity maps, then so is
    $T \hstar V$.
    \item $\|S \star V\| \leq \|S\| \|V\|$ and
    $\|T \hstar V\| \leq \|T\| \|V\|$.
    \item If $S$ and $V$ are isometries (resp.\ unitaries),
    so is $S \star V$; if $T$ and $V$ are isometries (resp.\ unitaries),    so is $T \hstar V$.
    \item $(S \star V)^* = S^* \star V^*$ and
    $(T \hstar V)^* = T^* \hstar V^*$.
    \item If $S$ decomposes as a direct sum $\dss{S}{L}
    \oplus \dss{S}{R}: \HH_1 \oplus \com \to \HH_2 \oplus \com$,
    then $S \star V$ maps the summands
    of $\HH_1 \star \LL_1$ into the corresponding summands
    of $\HH_2 \star \LL_2$.  That is, if $P_1$ is the projection
    from $\HH_1 \star \LL_1$ onto any of
    $\HH_1^+$, $\LL_1^{+ \otimes n} \otimes \LL$,
    or $\HH_1 \otimes \LL_1^{+ \otimes n} \otimes \LL$,
    and $P_2$ the projection from $\HH_2 \star \LL_2$ onto
    its corresponding subspace, then
    \begin{equation} \label{eqnSPopsintertwines}
    (S \star V) P_1 = P_2 (S \star V).
    \end{equation}
    Similarly, if $T$ decomposes as a direct sum
    $\dss{T}{L} \oplus \dss{T}{R}:\HH_1 \oplus \LL_1
    \to \HH_2 \oplus\LL_2$ and $P_1, P_2$ are as before, then
    \begin{equation} \label{eqnSPopsintertwines2}
    (T \hstar V) P_1 = P_2 (T \hstar V).
    \end{equation}

    \item If $S$ decomposes as the direct sum $\dss{S}{L}
    \oplus \dss{\text{id}}{\com}: \HH_1 \oplus \com \to \HH_2 \oplus \com$
    and $T = S\oplus V$, then
    $S \star V = T \hstar V$.
\end{enumerate}
\end{proposition}

\begin{proof}
The first five assertions are simple
consequences of the corresponding properties of
tensor products of operators.  The sixth follows
as a special case of Proposition \ref{propSPopsummands}
with $y = \beta = 0$ or $w = \zeta = 0$, and the
seventh is also an easy corollary of Proposition \ref{propSPopsummands}.
\end{proof}

\begin{remark} \label{rembifunctor}
The first two properties say that $- \star -$
and $- \hstar -$ are bifunctors from (Hilbert spaces, bounded maps)
$\times$(Hilbert spaces, contractions) to (Hilbert spaces, bounded maps), and the third and fourth imply that they restrict to bifunctors
from (Hilbert spaces, contractions)$^2$ to (Hilbert spaces, contractions),
from (Hilbert spaces, isometries)$^2$ to (Hilbert spaces, isometries), and from (Hilbert spaces, unitaries)$^2$ to (Hilbert spaces, unitaries).
\end{remark}

\begin{remark} \label{remSPsubspaces}
Together, the fourth and seventh parts of Proposition
\ref{propSPopproperties} imply that, given isometries
$W: \HH_1 \to \HH_2$ and $V: \LL_1 \to \LL_2$, one obtains
an isometry $\HH_1 \star \LL_1 \to \HH_2 \star \LL_2$ which
may be constructed either as $(W \oplus \dss{\text{id}}{\com})
\star V$ or as $(W \oplus \dss{\text{id}}{\com} \oplus V)
\hstar V$.  The sixth part then implies that this
induced isometry maps each summand of $\HH_1 \star \LL_1$
into the corresponding summand of $\HH_2 \star \LL_2$.
\end{remark}

An important special case of Proposition
\ref{propSPopsummands} occurs when $\HH_1 = \HH_2$,
$\LL_1 = \LL_2$, and $V$ is the identity map.  In
 this case we obtain unital representations
of both $B(\HH^+)$ and $B(\KK)$ on $\HH \star \LL$,
given by $S \mapsto S \star I$ and $T \mapsto T \hstar I$.

\begin{proposition} \label{prop3UEreps}
Let $\HH$ and $\LL$ be Hilbert spaces and $\KK = \HH^+ \oplus \LL$.
Let $\Phi: B(\HH^+) \to B(\HH \star \LL)$ and
$\Psi: B(\KK) \to B(\HH \star \LL)$ be the representations induced
by the unitary equivalences of Proposition \ref{prop3UE}.  Let
$b \in B(\HH^+)$, $a \in B(\KK)$, $h \in \HH^+$, $h_0 \in \HH$,
$k \in \KK$, $\ell \in \LL^+$, and $\xi \in \LL^{+ \otimes n} \otimes \LL$ for some $n \geq 0$, and suppose that
\begin{align*}
b \Omega &= \alpha \Omega + y, \qquad \qquad \alpha \in \com, y \in \HH\\
b h_0 &= \beta \Omega + z, \qquad \qquad \beta \in \com, z \in \HH\\
a h_0 &= \eta + w, \qquad \qquad \eta \in \HH, w \in \LL^+\\
a \ell &= \zeta+u, \qquad \qquad \zeta \in \HH, u \in \LL^+.
\end{align*}
Then
\begin{align*}
\Phi(b) h &= h\\
\Phi(b) \xi &= \alpha \xi + (y \otimes \xi)\\
\Phi(b)(h_0 \otimes \xi) &= \beta \xi + (z \otimes \xi)\\
\Psi(a) k &= a k \\
\Psi(a)(h_0 \otimes \xi) &= (\eta \otimes \xi) + (w \otimes \xi)\\
\Psi(a)(\eta \otimes \xi) &= (\zeta \otimes \xi) + (u \otimes \xi).
\end{align*}
\end{proposition}

The following is an immediate consequence:

\begin{corollary} \label{corinvariant}
In the situation of Proposition \ref{prop3UEreps},
the subspaces $\HH^+$ and\\
$(\LL^{+ \otimes n} \otimes \LL)
\oplus (\HH \otimes \LL^{+ \otimes n} \otimes \LL)$ of
$\HH \star \LL$ are $\Phi$-invariant, while the subspaces\\
$\HH^+ \oplus \LL$ and $(\HH \otimes \LL^{+ \otimes n} \otimes \LL)
\oplus (\LL^{+ \otimes (n+1)} \otimes \LL)$ are $\Psi$-invariant.
\end{corollary}

We visualize this corollary using a stairstep diagram:
\[
\begin{matrix}
\HH^+ \\
\LL & (\HH \otimes \LL) \\
& (\LL^+ \otimes \LL) & (\HH \otimes \LL^+ \otimes \LL) \\
& & (\LL^{+ \otimes 2} \otimes \LL) & (\HH \otimes \LL^{+ \otimes 2}
\otimes \LL) \\
& & & \ddots
\end{matrix}
\]
The rows here are $\Phi$-invariant, while the columns are $\Psi$-invariant.  Equivalently, $\Phi$ and $\Psi$
have staggered block-diagonal decompositions:
\[
\Phi = \begin{bmatrix} * & 0  & 0 & 0 & 0 & \dots \\
0 & * & * & 0 & 0 & \dots\\
0 & * & * & 0 & 0 & \dots\\
0 & 0 & 0 & * & * & \dots\\
0 & 0 & 0 & * & * & \dots\\
\vdots & \vdots & \vdots & \vdots & \vdots & \ddots
\end{bmatrix}, \qquad \qquad \qquad
\Psi = \begin{bmatrix} * & *  & 0 & 0 & 0 & 0 & \dots \\
* & * & 0 & 0 & 0 & 0 & \dots\\
0 & 0 & * & * & 0 & 0 & \dots\\
0 & 0 & * & * & 0 & 0 & \dots\\
0 & 0 & 0 & 0 & * & * & \dots\\
0 & 0 & 0 & 0 & * & * & \dots\\
\vdots & \vdots & \vdots & \vdots & \vdots & \vdots & \ddots
\end{bmatrix}
\]
We see that a sufficiently long word $\Phi(b_0) \Psi(a_1)
\Phi(b_1) \Psi(a_2) \cdots$ applied to a vector in one of
these subspaces could have a nonzero component in any other subspace.  Keeping track of such components will become important later on.

\begin{remark} \label{remsauvageotvsfree}
For simplicity of definition, we have thus far begun with Hilbert spaces $\HH$ and $\LL$, and defined the space $\KK = \HH^+ \oplus \LL$ in terms of them.  In application, however, we will begin with an inclusion
$H \subset K$ (obtained from Stinespring dilation), select
 a unit vector $\Omega \in H$, and form the Sauvageot product
 $H^- \star (K \ominus H)$.
As noted above, $B(H^- \star (K \ominus H))$ contains unital
copies of both $B(H)$ and $B(K)$.  In this alone, however, it
is no different from $B(H \otimes K)$ or $B(H * K)$, where
$H * K$ denotes the free product of Hilbert spaces in the
sense of \cite{VoiculescuSymmetries}.  The crucial difference
is that, when both are represented on $H^- \star (K \ominus H)$,
the copy of $B(H)$ is a corner of the copy of $B(K)$; if
$H \subset K$ is a Stinespring dilation, the compression will
implement a given unital completely positive map.
\end{remark}

\section{Sauvageot Products of C$^*$-Algebras and W$^*$-Algebras}
We now begin our construction of the Sauvageot product of
unital C$^*$-algebras with respect to a unital completely
positive map; the construction requires the choice of
a state on one of the C$^*$-algebras, prompting the following definition.

\begin{definition} \label{defCPtuple}
A \textbf{CPC$^*$-tuple} (resp.\ \textbf{CPW$^*$-tuple})
is a quadruple $(A, B, \phi, \omega)$ where
$A$ and $B$ are unital C$^*$-algebras (resp.\ W$^*$-algebras),
$\phi: A \to B$ a unital (normal) completely positive map, and $\omega$ a (normal) state on $B$.
The term \textbf{CP-tuple} will refer to CPC$^*$- and
CPW$^*$-tuples together.  A CP-tuple
is said to be \textbf{faithful} if $\omega$ is a faithful state.
\end{definition}

\begin{definition} \label{defrepCPtuple}
A \textbf{representation} of a CPC$^*$-tuple $(A, B, \phi, \omega)$
is a sextuple $(H, \Omega, \dss{\pi}{R}, K, V, \dss{\pi}{L})$
where
\begin{enumerate}
    \item $H$ is a Hilbert space
    \item $\Omega \in H$ is a unit vector
    \item $\dss{\pi}{R}: B \to B(H)$ is a unital *-homomorphism
    such that $\la \Omega, \dss{\pi}{R}(\cdot) \Omega
    \ra = \omega(\cdot)$
    \item $K$ is a Hilbert space
    \item $V: H \to K$ is an isometry
    \item $\dss{\pi}{L}: A \to B(K)$ is a unital *-homomorphism
    such that $V^* \dss{\pi}{L}(\cdot) V =
    \dss{\pi}{R}(\phi(\cdot))$.
\end{enumerate}
For a CPW$^*$-tuple, we also require that $\dss{\pi}{R}$
and $\dss{\pi}{L}$ be normal.  A representation is
\textbf{right-faithful} if $\dss{\pi}{R}$ is injective (which is
automatically the case for a representation of a faithful
CP-tuple), and \textbf{left-faithful} if $\dss{\pi}{L}$ is
injective.  We also refer to $(H, \Omega, \dss{\pi}{R})$
satisfying the first three criteria as a \textbf{representation} of $(A, \omega)$.
\end{definition}

From now until Definition \ref{defSPCPtuple}
we fix a CP-tuple $(A, B, \phi, \omega)$
and a right-faithful representation
$(H, \Omega, \dss{\pi}{R}, K, V, \dss{\pi}{L})$.
We introduce the following additional notation:
\begin{itemize}
    \item $L=K \ominus VH$
    \item $\Hh = H^- \star L$
    \item $\dss{\psi}{L}: A \to B(\Hh)$ and $\dss{\psi}{R}: B \to B(\Hh)$ are the compositions of $\dss{\pi}{L}$ and $\dss{\pi}{R}$ with the representations of Proposition \ref{prop3UEreps}
    \item $A \star B$ is the C$^*$-subalgebra (or von Neumann
    subalgebra, in case of a CPW$^*$-tuple) of $B(\Hh)$
    generated by the images of $\dss{\psi}{L}$
    and $\dss{\psi}{R}$
    \item $H' = H \ominus \overline{\dss{\pi}{R}(B) \Omega}$, which
    could be zero
    \item $q_n$ for $n \geq 0$ is the projection from
    $\Hh$ onto the subspace $H' \otimes L^{+ \otimes n}
    \otimes L$ of $H \otimes L^{+ \otimes n} \otimes L$
    \item $\Cc: B(\Hh) \to B(\Hh)$ is the non-unital conditional
    expectation
    \[
    \Cc(T) = \dss{P}{H} T \dss{P}{H} + \sum_{n=0}^\infty q_n T q_n
    \]
    \item $\varpi$ is the vector state on $B(\Hh)$
    corresponding to $\Omega$.
\end{itemize}

\begin{proposition} \label{propcovariantcorner}
\[
\Cc \circ \dss{\psi}{L} = \Cc \circ \dss{\psi}{R}
\circ \phi.
\]
\end{proposition}

\begin{proof}
For $a \in A$ and $h \in H$, let
$\dss{\pi}{L}(a) h = x + \ell$ with $x \in h$ and $\ell
\in L$; then $x = \dss{P}{H} \dss{\pi}{L}(a) \dss{P}{H} h
= \dss{\pi}{R}(\phi(a)) h$.  It follows from
Proposition \ref{prop3UEreps} that $\dss{\psi}{L}(a) h
= x +\ell$, so that
\[
\dss{P}{H} \dss{\psi}{L}(a) \dss{P}{H}
= x = \dss{P}{H} \dss{\pi}{R}(\phi(a)) \dss{P}{H}.
\]
Similarly, $q_n \dss{\psi}{L}(a) q_n
= q_n \dss{\pi}{R}(\phi(a)) q_n$ for all $n \geq 0$.
Summing over $n$ yields the result.
\end{proof}

For the next lemma and proposition we use
$E_n$ to denote the subspace $L^{+ \otimes n} \otimes L$
of $\Hh$.

\begin{lemma} \label{lemhorriblyunmotivated}
Let $\zeta \in E_n$.
\begin{enumerate}
        \item Let $a \in A$ and $b \in B$ with $\omega(b) = 0$.  Then
        \[
        \big[ \dss{\psi}{L}(a) - \dss{\psi}{R}(\phi(a))  \big] \dss{\psi}{R}(b)  \zeta
        = [ \dss{P}{L^+} \dss{\pi}{L}(a) \dss{\pi}{R}(b) \Omega] \otimes \zeta
        - \omega(\phi(a) b) \zeta \in E_n \oplus E_{n+1}.
        \]

        \item More generally, given $a_1, \dots, a_k \in A$
        and $b_1, \dots, b_k \in B$ such that $\omega(b_i) = 0$ for
        $i = 1, \dots, k$, if we define
        \[
        \zeta_k =  \left( \prod_{i=1}^k \big[ \dss{\psi}{L}(a_i) - \dss{\psi}{R}(\phi(a_i)) \big] \dss{\psi}{R}(b_i)
        \right) \zeta,
        \]
        then
        \[
        \zeta_k \in \bigoplus_{i=0}^k E_{n+i} \text{ with }
        P_{E_n} (\zeta_k)  = (-1)^k \prod_{i=1}^k \omega(\phi(a_i) b_i) \zeta.
        \]
    \end{enumerate}
\end{lemma}

\begin{proof}
The stipulation that $\omega(b) = 0$ implies that $\dss{\pi}{R}(b) \Omega \in H^-$,
so that
\[
\xi := \dss{\psi}{R}(b) \zeta = (\dss{\pi}{R}(b) \Omega) \otimes \zeta
\in H^- \otimes E_n
\]
where we have used the calculations in Proposition (\ref{prop3UEreps}). We have now to apply two different operators to $\xi$ and subtract the results.  First, when we apply $\dss{\psi}{R}(\phi(a))$ we get
    \[
    \dss{\psi}{R}(\phi(a)) \xi = (\dss{\pi}{R}(\phi(a) b) \Omega) \otimes \zeta
    = [\omega(\phi(a) b) \Omega] \otimes \eta + [\dss{P}{H^-} \dss{\pi}{R}(\phi(a) b) \Omega] \otimes \zeta
    \]
    by virtue of the fact that $\dss{\pi}{R}(\beta) \Omega = \omega(\beta) \Omega + \dss{P}{H^-} (\dss{\pi}{R}(\beta) \Omega)$
    for all $b \in B$.    Secondly, we apply $\dss{\psi}{L}(a)$ as follows:
    \begin{align*}
    \dss{\psi}{L}(a) \xi &= (\dss{\pi}{L}(a) \dss{\pi}{R}(b) \Omega) \otimes \zeta \\
    &= \Big[ \left(\dss{P}{H^-} \dss{\pi}{L}(a) \dss{\pi}{R}(b) \Omega \right)
    \oplus \left(\dss{P}{L^+} \dss{\pi}{L}(a) \dss{\pi}{R}(b) \Omega \right) \Big] \otimes \zeta\\
    &= \Big[ \left(\dss{P}{H^-} P_H \dss{\pi}{L}(a) \dss{\pi}{R}(b) \Omega \right)
    \oplus \left(\dss{P}{L^+} \dss{\pi}{L}(a) \dss{\pi}{R}(b) \Omega \right) \Big] \otimes \zeta\\
    &= \Big[ \left(\dss{P}{H^-} \dss{\pi}{R}(\phi(a)b) \Omega \right) \otimes \zeta\Big]
    \oplus \Big[\left(\dss{P}{L^+} \dss{\pi}{L}(a) \dss{\pi}{R}(b) \Omega \right) \otimes \zeta\Big].
    \end{align*}
    Subtracting  yields the desired result.
    The second assertion of the lemma follows by induction.
\end{proof}

We now connect the current material to chapter \ref{chapliberation}.

\begin{proposition} \label{propcornerliberation}
$(B(\Hh), \dss{\psi}{L}, \dss{\psi}{R}, \Cc, \varpi)$ is a right-liberating
representation of $(A, B, \phi)$.
\end{proposition}

\begin{proof}
The first two properties of a right-liberating representation
are easy to verify:
\begin{enumerate}
    \item This was Proposition \ref{propcovariantcorner}.

    \item Since $H$ and each $H' \otimes E_n$ are
    $\dss{\psi}{R}$-invariant subspaces by Proposition
    \ref{prop3UEreps}, their projections all commute
    with $\dss{\psi}{R}$, so that $\Cc$ is a
    $\dss{\psi}{R}(B)$-bimodule map.
\end{enumerate}
For the last, let $\xi \in H$ and let $a_1, \dots, a_n \in A$,
    $b_1, \dots, b_n \in B$ such that $\omega(b_2) = \dots = \omega(b_n) = 0$.
    Define the operators
    \[
    T_k = \Big(\dss{\psi}{L}(a_k) - \dss{\psi}{R}(\phi(a_k)) \Big)\dss{\psi}{R}(b_k)
    \]
    on $\Hh$, and the vectors
    \[
    \zeta_k = T_k \dots T_1 \xi \in \Hh.
    \]
    We will show by induction that $\zeta_k \in \displaystyle\bigoplus_{j=0}^{k-1} E_j$,
    which is contained in the kernel of $\dss{P}{H}$; it will follow that
    $\dss{P}{H} T_k \dots T_1 \dss{P}{H} = 0$.  For the base case $k=1$, we have
    $\dss{\psi}{R}(b_1) \xi = \dss{\pi}{R}(b_1) \xi,$ so that
    $\dss{\psi}{R}(\phi(a_1)) \dss{\psi}{R}(b_1) \xi = \dss{\pi}{R}(\phi(a_1) b_1) \xi$.
    We also have
    \begin{align*}
    \dss{\psi}{L}(a_1) \dss{\psi}{R}(b_1) \xi &= \dss{\pi}{L}(a_1) \dss{\pi}{R}(b_1) \xi \\
    &= \Big( \dss{P}{L} \dss{\pi}{L}(a_1) \dss{\pi}{R}(b_1) \xi \oplus
    \dss{P}{H} \dss{\pi}{L}(a_1) \dss{\pi}{R}(b_1) \xi \Big) \\
    &= \Big( \dss{P}{L} \dss{\pi}{L}(a_1) \dss{\pi}{R}(b_1) \xi \oplus
    \dss{P}{H} \dss{\pi}{L}(a_1) \dss{P}{H} \dss{\pi}{R}(b_1) \xi \Big) \\
    &= \Big( \dss{P}{L} \dss{\pi}{L}(a_1) \dss{\pi}{R}(b_1) \xi \oplus
    \dss{\pi}{R}(\phi(a_1)) \dss{\pi}{R}(b_1) \xi \Big)
    \end{align*}
    and subtracting yields
    \[
    \zeta_1 = \dss{P}{L} \dss{\pi}{L}(a_1) \dss{\pi}{R}(b_1) \xi \in E_0
    \]
    as desired.  The inductive step is immediate from Lemma
    \ref{lemhorriblyunmotivated}.

    Similarly, for $\xi \in H'$ and $\eta \in E_n$,
    $\dss{\psi}{R}(b_1)(\xi \otimes \eta) = (\dss{\pi}{R}(b_1) \xi) \otimes \eta$ so that
    \[
    \dss{\psi}{R}(\phi(a_1)) \dss{\psi}{R}(b_1) (\xi \otimes \eta)
    = [\dss{\pi}{R}(\phi(a_1) b_1) \xi] \otimes \eta.
    \]
    Then
    \begin{align*}
    \dss{\psi}{L}(a_1) \dss{\psi}{R}(b_1) (\xi \otimes \eta)
    &= \dss{\psi}{L}(a_1)[(\dss{\pi}{R}(b_1) \xi) \otimes \eta] \\
    &= [\dss{\pi}{L}(a_1) \dss{\pi}{R}(b_1) \xi] \otimes \eta \\
    &= \Big[ \dss{P}{L} \dss{\pi}{L}(a_1) \dss{\pi}{R}(b_1) \xi \oplus
    \dss{P}{H} \dss{\pi}{L}(a_1) \dss{\pi}{R}(b_1) \xi \Big] \otimes \eta \\
    &= \Big[ \dss{P}{L} \dss{\pi}{L}(a_1) \dss{\pi}{R}(b_1) \xi \oplus
    \dss{P}{H} \dss{\pi}{L}(a_1) \dss{P}{H} \dss{\pi}{R}(b_1) \xi \Big] \otimes \eta \\
    &= \Big[ \dss{P}{L} \dss{\pi}{L}(a_1) \dss{\pi}{R}(b_1) \xi \oplus
    \dss{\pi}{R}(\phi(a_1)) \dss{\pi}{R}(b_1) \xi \Big] \otimes \eta
    \end{align*}
    and subtracting yields
    \[
    \zeta_1 = [\dss{P}{L} \dss{\pi}{L}(a_1) \dss{\pi}{R}(b_1) \xi] \otimes \eta \in E_{n+1}.
    \]
    It follows by induction that $\zeta_k \in \displaystyle\bigoplus_{j=0}^{k-1} E_{j+n}$, so that $q_n \zeta_k = 0$; hence
    $q_n T_k \dots T_1 q_n = 0$.  Summing over $n$,
    we have $\Cc(T_k \dots T_1) = 0$.
\end{proof}

\begin{corollary} \label{corcornermapsinto}
\[
\Cc(A \star B) = \Cc(\dss{\psi}{R}(B)).
\]
\end{corollary}

\begin{proof}
This is an immediate consequence of Proposition
\ref{propcornerliberation} together with
Corollary \ref{corrightlibrepgenerated} and the
norm continuity and normality of $\Cc$.
\end{proof}

Before making our next definition, we note that the
right-faithfulness of our representation guarantees that
$b \mapsto \dss{P}{H} \dss{\psi}{R}(b) \dss{P}{H}$ is
injective, so that $\Cc \circ \dss{\psi}{R}$ is
injective as well.

\begin{definition} \label{defSPretraction}
The \textbf{Sauvageot retraction} for the given tuple
and representation is the map $\theta: A \star B \to B$
given by
\[
\theta = (\Cc \circ \dss{\psi}{R})\inv \circ \Cc.
\]
\end{definition}

The Sauvageot retraction is well-defined by Corollary
 \ref{corcornermapsinto}, and is
 a retraction with respect to $\dss{\psi}{R}$; furthermore,
 as a consequence of Proposition \ref{propcovariantcorner},
 it factors $\phi$ in the sense that
 \begin{equation} \label{eqnthetafactorsphi}
 \theta \circ \dss{\psi}{L} = \phi.
 \end{equation}

 Furthermore, the following is an immediate consequence
 of Proposition \ref{propcornerliberation}:

\begin{corollary} \label{corthetaliberates}
$(A \star B, \dss{\psi}{L}, \dss{\psi}{R},
\dss{\psi}{R} \circ \theta, \varpi)$ is a right-liberating
representation of $(A, B, \phi)$.
\end{corollary}

\begin{definition} \label{defSPCPtuple}
Given a CP-tuple $(A, B, \phi, \omega)$ and a right-liberating
representation $(H, \Omega, \dss{\pi}{R}, K, V, \dss{\pi}{L})$,
the \textbf{Sauvageot product of the tuple realized by
the representation} is the tuple $(A \star B, \dss{\psi}{L},
\dss{\psi}{R}, \theta)$ of objects constructed as above.
\end{definition}

\section{Induced Morphisms and Uniqueness}
We pause now to consider an analogy with other product constructions.  In building either the (minimal) tensor product or the free product of C$^*$-algebras $A$ and $B$, one can proceed as follows:
\begin{enumerate}
    \item Represent $A$ and $B$ on Hilbert spaces $H$ and $K$
    \item Form the product Hilbert space $H \otimes K$ or $H * K$
    \item Lift the representations of $A$ and $B$ to
    representations of each on this product Hilbert space
    \item Take the C$^*$-subalgebra generated by the images
    of these representations.
\end{enumerate}
In both cases, one can show that the resulting C$^*$-algebra
$A \otimes B$ or $A * B$ is, up to isomorphism, independent of
the choice of the representations of $A$ and $B$ provided
both are faithful.

We have followed the same outline in constructing $A \star B$,
and come now to the question of the independence of this object
from the representations used to produce it.  It turns out that
we need some more complicated hypotheses on the representation,
resulting from the fact that a representation of a CPC$^*$-tuple
is more complicated than a representation of a C$^*$-algebra,
as well as the fact that the product $A \star B$ comes with
the additional information of a retraction onto $B$.

\begin{definition} \label{defdecompfaithful}
Let $(A, B, \phi, \omega)$ be a CP-tuple,
$(H, \Omega, \dss{\pi}{R}, K, V, \dss{\pi}{L})$ a
representation, and $L = K \ominus VH$.
\begin{itemize}
    \item A \textbf{decomposition} of the representation is
    a pair $(L', L'')$ of $\dss{\pi}{L}$-invariant
    subspaces $L' \subset L$ and $L'' \subset L^+$, with the properties
    \begin{align*}
    L &\subset L' + \overline{\dss{\pi}{L}(A) VH}\\
    L^+ &\subset L'' + \overline{\dss{\pi}{L}(A) V H^-}.
    \end{align*}
    \item A decomposition is \textbf{faithful} if the subrepresentation
    $\dss{\pi}{L} \big|_{L'}$ is faithful.
    \item A representation for which there exists a
    faithful decomposition is \textbf{faithfully
    decomposable}.
    \item A representation is \textbf{right-faithful} if
    $\dss{\pi}{R}$ is faithful; note that any representation
    of a faithful CPC$^*$-tuple or CPW$^*$-tuple is
    automatically right-faithful.
    \item A representation is \textbf{left-faithful} if
    $\dss{\pi}{L}$ is faithful; note that every faithfully
    decomposable representation is left-faithful.
    \item A representation is \textbf{faithful} if it
    is right-faithful and faithfully decomposable.
\end{itemize}
\end{definition}

\begin{proposition} \label{propexistsfaithfulrep}
Every CP-tuple has a faithful representation.
\end{proposition}

\begin{proof}
We begin by letting $(H, \Omega, \dss{\pi}{R})$ be the GNS
construction for $(B, \omega)$; if $\omega$ is not
faithful, we replace $(H, \dss{\pi}{R})$ by its direct
sum with some faithful representation of $B$.  This guarantees
that our representation will be right-faithful.

Next, let $(K, V, \dss{\pi}{L})$ be the minimal
Stinespring dilation of $\dss{\pi}{R} \circ \phi$, and
let $L' = K \ominus \overline{\dss{\pi}{L}(A) VH}$
and $L'' = K \ominus \overline{\dss{\pi}{L}(A) H^-}$.
If $\dss{\pi}{L} \Big|_{L'}$ is not faithful,
we replace $(K, \dss{\pi}{L})$ by its direct sum with
some faithful representation of $A$, thereby
guaranteeing faithful decomposability.
\end{proof}

When some of these additional hypotheses are satisfied,
we can define a retraction from $A \star B$ to $A$ which
has properties analogous to the retraction $A \star B
\sa{\theta} B$ already discussed.  We continue our standard
notation for a CP-tuple, a right-faithful representation,
and the corresponding realization of the Sauvageot product,
and now fix a decomposition $(L', L'')$ as well (not assumed faithful unless specified).  We introduce additional notation:
\begin{itemize}
    \item $E_0'$ is the subspace $L' \subset L$
    \item For $n \geq 1$, $E_n'$ is the subspace
    $L'' \otimes L^{+ \otimes (n-1)} \otimes L
    \subset E_n$
    \item For all $n \geq 0$, $p_n$ is the projection
    from $\Hh$ onto $E_n'$
    \item For all $n \geq 0$, $F_n = H \otimes E_n$
    and $F_n' = H' \otimes E_n$; recall
    that $q_n$ is the projection onto $F_n'$
    \item $E_{-1} = \com \Omega$ and $F_{-1} = H^-$.
\end{itemize}

\begin{definition} \label{defleftcornermap}
The \textbf{left corner map}
for the given realization and decomposition is
the non-unital conditional expectation
$\Cc'$ on $B(\Hh)$ defined by
\[
\Cc'(T) = \sum_{n=0}^\infty p_n T p_n.
\]
\end{definition}

\begin{lemma} \label{lemintersectkernels}
Let $J \subset A \star B$ be an ideal contained
in the intersection of the kernels of $\Cc$ and $\Cc'$.
Then $J = \{0\}$.
\end{lemma}

\begin{proof}
Let $J^0$ denote the annihilator of $J$ in $\Hh$, i.e. the largest (necessarily closed)
subspace such that $J J^0 = \{0\}$.
\begin{enumerate}
    \item For any $\alpha \in J$, note that $\alpha^* \alpha \in J$; then
    \[
    (\alpha \dss{P}{H})^* (\alpha \dss{P}{H}) = \dss{P}{H} \alpha^* \alpha \dss{P}{H} \leq \Cc[\alpha^* \alpha] = 0,
    \]
    so that $\alpha \dss{P}{H} = 0$.  Similarly, $\alpha q_n = \alpha p_n = 0$ for all $n$.  Hence $H \subseteq J^0,$ and for each
    $n \geq 0$, $E_n'  \subseteq J^0$ and $H' \otimes E_n\subseteq J^0$.

    \item We will prove by induction
    that $E_n \subseteq J^0$ and $F_n \subseteq J^0$;
    the base case $n = -1$ was just established,
    since $E_{-1} \oplus F_{-1} = H \subset J^0$.

    \item Suppose $E_n \subseteq J^0$ and $F_n \subseteq J^0$.
    \begin{enumerate}
        \item Since $J \psi_L(A) F_n \subseteq J F_n = \{0\}$, we have $\psi_L(A) F_n \subseteq J^0$.
        \item By Proposition (\ref{prop3UEreps}), $[\psi_L(A) H^-] \otimes E_n \subseteq \psi_L(A) F_n$,
        so that $[\psi_L(A) H^-] \otimes E_n \subseteq J^0$, for the case $n \geq 0$; for $n = -1$,
        we have $\psi_L(A) H \subseteq J^0$.
        \item Since we already know $L' \subseteq J^0$ (in the case $n = -1$) or
        $L'' \otimes E_n \subseteq J^0$ (in the case $n \geq 0$), it follows that
        $E_{n+1} = L \subset \overline{\psi_L(A) H} + L' \subseteq J^0$ for the case $n = -1$,
        or $E_{n+1} = L^+ \otimes E_n \subset [\overline{\pi_L(A) H^-} + L''] \otimes E_n \subseteq J_0$
        for the case $n \geq 0$.
        \item Since $J \psi_R(B) E_{n+1} \subseteq J E_{n+1} = \{0\}$, we have $\psi_R(B) E_{n+1} \subseteq J^0$.
        \item By Proposition (\ref{prop3UEreps}), this
        implies
        $\overline{\pi_R(B) \Omega} \otimes E_{n+1} \subseteq \overline{\psi_R(B) E_{n+1}} \subseteq J^0$.
        \item Since we also have $F_n' \subseteq J^0$,
        \[
        F_{n+1} = H^- \otimes E_{n+1}
        \subseteq H \otimes E_{n+1} = (\overline{\pi_R(B) \Omega} + H') \otimes E_{n+1} \subseteq \overline{\psi_R(B) E_{n+1}} + F_{n+1}' \subseteq J^0.
        \]
    \end{enumerate}
\end{enumerate}
\end{proof}

\begin{remark} \label{remideals}
We note that if $L'$ and $L''$ are both zero (for instance,
when $K$ is given by a minimal Stinespring dilation),
then $\Cc'$ is the zero map; in this case, the lemma
says that $\Cc$ has no nontrivial ideals in its kernel.
This corresponds to the fact that $A$ and $B$ together move $H$ around to all the other components of $\Hh$, in the sense that
  $\overline{(A \star B) H} = \Hh$.  On the other extreme,
if $(L', L'')$ is a faithful decomposition, one has
instead that $\overline{(A \star B) H
+ (A \star B) L' + (A \star B) L''} = \Hh$, but none
of $H, L', L''$ by itself is enough to reach all of
$\Hh$.  As a result, $\Cc$ and $\Cc'$ may each contain
ideals in their kernel, but these ideals are ``orthogonal''
in the sense of the lemma.
\end{remark}

\begin{proposition} \label{propleftcornercovariant}
Let $\hat{\omega}: B \to A$ be the unital completely
positive map $\hat{\omega}(b) = \omega(b) \one$.
Then $\Cc' \circ \dss{\psi}{L} \circ \hat{\omega}
= \Cc' \circ \dss{\psi}{R}$.
\end{proposition}

\begin{proof}
Given any $b \in B$, we have $\dss{\pi}{R}(b) \Omega
= \omega(b) \Omega + h_0$ for some $h_0 \in H^-$.
By Proposition \ref{prop3UEreps}, it follows that, for
any $\xi \in E_n'$, $\dss{\psi}{R}(b) \xi =
\omega(b) \xi + h_0 \otimes \xi$ and therefore that
$p_n \dss{\psi}{R}(b) \xi = \omega(b) \xi$; hence,
\[
p_n \dss{\psi}{R}(b) p_n
= \omega(b) \dss{\one}{E_n}
= p_n \omega(b) \one p_n = p_n \dss{\psi}{L}(\hat{\omega}(b)) p_n.
\]
Summing over $n$ yields the result.
\end{proof}

\begin{proposition} \label{propleftcornerliberated}
$(B(\Hh), \dss{\psi}{L}, \dss{\psi}{R}, \Cc', \varpi)$ is
a left-liberating representation of $(A, B, \phi)$.
\end{proposition}

\begin{proof}
The first two properties of a left-liberating proposition are
easy to check:
\begin{enumerate}
    \item This was Proposition \ref{propleftcornercovariant}.
    \item By Proposition \ref{prop3UEreps}, all of the
    $E_n'$ are $\dss{\psi}{L}$-invariant, so their
    projections commute with $\dss{\psi}{L}$; hence
    $\Cc'$ is a $\dss{\psi}{L}(A)$-bimodule map.
\end{enumerate}
For the last, let $a_1, \dots, a_n \in A$
and $b_0, \dots, b_n \in B$ with $\omega(b_i) = 0$.  Let
$m \geq 0$ and $\xi \in E_m'$.  By Lemma \ref{lemhorriblyunmotivated},
\[
\prod_{k=1}^n \left[ \dss{\psi}{L}(a_k) - \dss{\psi}{R}(\phi(a_k))
\right] \xi \in \bigoplus_{k=0}^m E_{n+k}.
\]
Since $\omega(b_0) = 0$, it follows that $\dss{\pi}{R}
(b_0) \Omega \in H^-$, so that by Proposition \ref{prop3UEreps}
we obtain
\[
\dss{\psi}{R}(b_0) \prod_{k=1}^n \left[ \dss{\psi}{L}(a_k) - \dss{\psi}{R}(\phi(a_k))
\right] \xi \in \bigoplus_{k=0}^m H^- \otimes E_{n+k}.
\]
From this we see that
\[
p_n \dss{\psi}{R}(b_0) \prod_{k=1}^n \left[ \dss{\psi}{L}(a_k) - \dss{\psi}{R}(\phi(a_k))
\right] p_n = 0,
\]
and summing over $n$ finishes the proof.
\end{proof}

\begin{corollary} \label{corleftcornermapsinto}
$\Cc'(A \star B) = \Cc'(A)$.
\end{corollary}

\begin{proof}
This is an immediate consequence of Proposition \ref{propleftcornerliberated}, Corollary \ref{corleftlibgenerated},
and the contractivity (and, in case $(A, B, \phi, \omega)$
is a CPW$^*$-tuple, the normality) of $\Cc'$.
\end{proof}

In the case of a faithful decomposition, $\Cc' \circ \dss{\psi}{L}$ is injective, which allows us to make the following definition:

\begin{definition} \label{defleftretraction}
Given a CP-tuple, a faithful representation, and a choice
of faithful decomposition, the \textbf{left retraction} for the given tuple
and representation is the map $\theta: A \star B \to A$
given by
\[
\theta = (\Cc' \circ \dss{\psi}{L})\inv \circ \Cc'.
\]
This is well-defined by Corollary
 \ref{corleftcornermapsinto}, and is
 a retraction with respect to $\dss{\psi}{L}$.
\end{definition}

We come now to the main result of this section.

\begin{theorem} \label{thmSPhoms}
Let $(A_1, B_1, \phi_1, \omega_1)$ and $(A_2, B_2, \phi_2, \omega_2)$ be CPC$^*$-tuples (resp. CPW$^*$-tuples),
$(H_1, \Omega_1, \dss{\pi}{R}^{(1)},
K_1, V_1, \dss{\pi}{L}^{(1)})$ a faithful representation of the former, \\
$(H_2, \Omega_2, \dss{\pi}{R}^{(2)}, K_2, V_2, \dss{\pi}{L}^{(2)})$ a right-faithful representation of the latter, and\\
$(A_1 \star B_1, \dss{\psi}{L}^{(1)}, \dss{\psi}{R}^{(1)},
\varpi_1, \theta_1)$ and $(A_2 \star B_2,
\dss{\psi}{L}^{(2)}, \dss{\psi}{R}^{(2)}, \varpi_2,
\theta_2)$ the Sauvageot products realized by these
representations.  Let $f: A_1 \to A_2$ and $g: B_1 \to B_2$
be unital (normal) *-homomorphisms
satisfying $\phi_2 \circ f = g \circ \phi_1$ and
$\omega_2 \circ g = \omega_1$.  Then there is a unique (normal) unital *-homomorphism
$f \star g: A_1 \star B_1 \to A_2 \star B_2$
with the properties that
\begin{enumerate}
    \item $(f \star g) \circ \dss{\psi}{L}^{(1)}
    = \dss{\psi}{L}^{(2)} \circ f$
    \item $(f \star g) \circ \dss{\psi}{R}^{(1)} =
    \dss{\psi}{R}^{(2)} \circ g$
    \item $\varpi_2 \circ (f \star g) = \varpi_1$
    \item $\theta_2 \circ (f \star g) = g \circ \theta_1$
\end{enumerate}
If $f$ and $g$ are both injective and
$(H_2, \Omega_2, \dss{\pi}{R}^{(2)},K_2, V_2, \dss{\pi}{L}^{(2)})$
is faithful, then $f \star g$ is injective.
\end{theorem}

\begin{proof}
Let $H = H_1 \oplus H_2$, $\Omega = \Omega_1$, $\dss{\pi}{R} = \dss{\pi}{R}^{(1)} \oplus (\dss{\pi}{R}^{(2)} \circ g)$,
$K = K_1 \oplus K_2$,
$V = V_1 \oplus V_2$, $\dss{\pi}{L} =
\dss{\pi}{L}^{(1)} \oplus (\dss{\pi}{L}^{(2)} \circ f)$.  Then
$(H, \Omega, \dss{\pi}{R}, K, V, \dss{\pi}{L})$ is another right-faithful representation of $(A_1, B_1, \phi_1, \omega_1)$.  Moreover, if $(L_1', L_1'')$ is a decomposition for $(H_1, \dots, \dss{\pi}{L}^{(1)})$, then
$L' = L_1' \oplus L_2$, $L'' = L_1'' \oplus L_2$
defines a decomposition $(L', L'')$
of $(H, \dots, \dss{\pi}{L})$, and the faithfulness
of $\dss{\pi}{L}^{(1)}$ on $L_1'$ implies the faithfulness
of $\dss{\pi}{L}$ on $L'$.

Let $(A_1 \tstar B_1, \dss{\psi}{L}, \dss{\psi}{R}, \varpi,
\theta)$ be the Sauvageot product realized
by $(H, \dots, \dss{\pi}{L})$, on the Hilbert space
$\Mm = H^- \star L$.

The inclusions of $H_1$ into $H$ and of $L_1$ into $L$ induce
an isometry $W: \Hh_1 \to \Hh$ as in Remark (\ref{remSPsubspaces}).  Moreover, by Equation (\ref{eqnSPopsintertwines}), this
isometry satisfies
\begin{equation} \label{eqnWintertwines}
W \circ \dss{\psi}{L}^{(1)}(\cdot) = \psi_L(\cdot) \circ W, \qquad
W \circ \dss{\psi}{R}^{(1)}(\cdot) = \psi_R(\cdot) \circ W.
\end{equation}
Let $\Psi$ be the restriction to $A_1 \tstar B_1$
of the (normal) unital CP map $T \mapsto W^* T W$, which
maps $B(\Hh)$ to $B(\Hh_1)$.  It follows from
Equation (\ref{eqnWintertwines}) that the image
of $W$ is invariant under $\dss{\psi}{L}$
and $\dss{\psi}{R}$, so that $W W^*$ commutes
with $A_1 \tstar B_1$.  Then
    \[
    \Psi(XY) = W^* XY W = W^*
    W W^* XY W =
    W^* X W W^* Y W
    = \Psi(X) \Psi(Y)
    \]
    for $X,Y \in A_1 \tstar B_1$, so that
    $\Psi$ is a *-homomorphism.

    Next, we show that
    $\Psi$ intertwines the representations,
    states, and retractions:
\begin{itemize}
    \item $\Psi \circ \dss{\psi}{L} = \dss{\psi}{L}^{(1)}$
    \item $\Psi \circ \dss{\psi}{R} = \dss{\psi}{R}^{(1)}$
    \item $\varpi_1 \circ \Psi = \varpi$
    \item $\theta_1 \circ \Psi = \theta$
    \item $\theta_1' \circ \Psi = \theta'$
\end{itemize}

The first three are immediate consequences of Equation (\ref{eqnWintertwines}).  For the fourth and fifth, we have
by Equation (\ref{eqnSPopsintertwines}) that
\[
\Psi \circ \Cc(T) = W^*\dss{P}{H} T \dss{P}{H} W^* + \sum_n W^* p_n T p_n W = \dss{P}{H_1} W^* T W \dss{P}{H_1} + \sum_n p_{n,1} W^* T W p_{n,1} = \Cc_1 \circ \Psi(T)
\]
so that $\Psi \circ \Cc = \Cc_1 \circ \Psi$,
and similarly for $\Cc'$ and $\Cc_1'$.  Now
\[
\Psi \circ \Cc \circ \dss{\psi}{R} =
\Cc_1 \circ \Psi \circ \dss{\psi}{R} = \Cc_1 \circ \dss{\psi}{R}^{(1)},
\]
which is invertible; then
\[
(\Cc \circ \dss{\psi}{R})\inv =
(\Psi \circ \Cc \circ \dss{\psi}{R})\inv\circ \Psi
\circ \Cc \circ \dss{\psi}{R} \circ (\Cc \circ \dss{\psi}{R})\inv
= (\Psi \circ \Cc \circ \dss{\psi}{R})\inv \circ \Psi
\]
from which it follows that
\begin{align} \label{eqnintertwineretractions}
\theta &= (\Cc \circ \dss{\psi}{R})\inv \circ \Cc \nonumber\\
&= (\Psi \circ \Cc \circ \dss{\psi}{R})\inv \circ \Psi \circ \Cc\nonumber\\
&= (\Cc_1 \circ \dss{\psi}{R}^{(1)})\inv \circ \Cc_1 \circ \Psi \nonumber\\
&= \theta_1 \circ \Psi
\end{align}
and similarly
\begin{equation} \label{eqnintertwinesretractions2}
\theta' = \theta_1' \circ \Psi.
\end{equation}

Note that $\Psi$ maps into $A_1 \star B_1$,
as it maps both $\dss{\psi}{L}(A_1)$ and $\dss{\psi}{R}(B_1)$
into $A_1 \star B_1$, hence also the C$^*$-algebra (resp.\ von Neumann algebra) that they generate; moreover, it is onto $A_1 \star B_1$, since its range is a C$^*$-algebra (resp.\ von Neumann algebra)
which includes
both $\dss{\psi}{L}^{(1)}(A_1)$ and $\dss{\psi}{R}^{(1)}(B_1)$.

Next, $\Psi$ is injective, because
its kernel is an ideal in $A_1  \tstar B_1$ which, by equations
\ref{eqnintertwineretractions} and \ref{eqnintertwinesretractions2},
is contained in the kernels of both $\theta$ and $\theta'$, therefore also in the kernels of both $\Cc$ and $\Cc'$, and
hence is the zero ideal by Lemma \ref{lemintersectkernels}.
So $\Psi$ is an isomorphism from $A_1 \tstar B_1$
to $A_1 \star B_1$.

We repeat the above analysis for the inclusions of $H_2$ into $H$
and $K_2$ into $K$ to obtain a unital *-homomorphism
$\Xi: A_1 \tstar B_1 \to A_2 \star B_2$ such that
\begin{itemize}
    \item $\Xi \circ \dss{\psi}{L} = \dss{\psi}{L}^{(2)} \circ f$
    \item $\Xi \circ \dss{\psi}{R} = \dss{\psi}{R}^{(2)} \circ g$
    \item $\varpi_2 \circ \Xi = \varpi$
    \item $\theta_2 \circ \Xi = g \circ \theta$
\end{itemize}

We can now define $f \star g = \Xi \circ \Psi\inv:
A_1 \star B_1 \to A_2 \star B_2$.  Then we
obtain the enumerated properties of $f \star g$
by combining the lists of properties for $\Psi$
and $\Xi$, as
\[
(f \star g) \circ \dss{\psi}{L}^{(1)} = \Xi \circ \Psi\inv
\circ \dss{\psi}{L}^{(1)} = \Xi \circ \dss{\psi}{L} = \dss{\psi}{L}^{(2)} \circ f
\]
and similarly.

The uniqueness of $f \star g$ follows from the fact that it is contractive (resp.\ normal)
and is determined on the dense subalgebra of $A_1 \star B_1$
generated by $\dss{\psi}{L}^{(1)}(A_1)$ and $\dss{\psi}{R}^{(1)}(B_1)$.

Finally, if $f$ and $g$ are both injective and
$(H_2, \dots, \dss{\pi}{L}^{(2)})$ is faithful,
we can prove the additional property
$\theta_2' \circ \Xi = \Xi \circ f$,
 after which we prove $\Xi$ to be injective exactly as we did with
 $\Psi$ was.  Hence $f \star g$ is a composition
of injective maps.
\end{proof}

\begin{corollary} \label{corunique}
Let $(A, B, \phi, \omega)$ be a CP-tuple.  Then the realizations
of the Sauvageot product by any two faithful representations
are isomorphic.
\end{corollary}

For clarity, ``isomorphic'' here refers to an isomorphism which
intertwines the appropriate maps.  The proof is
simply to take the map $\dss{\text{id}}{A} \star
\dss{\text{id}}{B}$ constructed in the theorem.  Based on this corollary, we may now speak of \emph{the}
Sauvageot product of a CP-tuple.

Another special case of interest occurs when one, but
not both, of the initial maps is the identity.  The
results are summarized as follows.

\begin{corollary} \label{corcommutingsquare}
Let $A, B$ be unital C$^*$-algebras (resp.\ W$^*$-algebras),
$A \sa{f} B$ a unital (normal) *-homomorphism
and $C$ another unital C$^*$-algebra (resp.\ W$^*$-algebra).

\begin{enumerate}
    \item Let $B \sa{\phi} C$ be a (normal) unital CP map
    and $\omega$ a (normal) state on $C$.  Then,
    for the CP-tuples $(A, C, \phi \circ f, \omega)$
    and $(B, C, \phi, \omega)$ with Sauvageot retractions
    $A \star C \sa{\theta} C$ and $B \star C \sa{\eta} C$, the diagrams
    \[ \xymatrix{
    A \ar[r]^f \ar[d] & B \ar[d]\\
    A \star C \ar[r]_{f \star \text{id}} & B \star C
    } \qquad \qquad \qquad \xymatrix{
    A \star C \ar[r]^{f \star \text{id}} \ar[rd]_{\theta}
    & B \star C \ar[d]^{\eta}\\
    & C
    }\]
    commute.

    \item Let $C \sa{\phi} A$ be a (normal) unital CP
    map and $\omega$ a (normal) state on $B$.  Then, for
    the CP-tuples $(C, A, \phi, \omega \circ f)$ and
    $(C, B, f \circ \phi, \omega)$, the square
    \[ \xymatrix{
    A \ar[r]^f \ar[d] & B \ar[d]\\
    C \star A \ar[r]_{\text{id} \star f} & C \star B
    } \]
    commutes.
\end{enumerate}
\end{corollary}

The composition of Sauvageot products of maps obeys
the obvious functorial property:

\begin{proposition} \label{propfunctorial}
For $i = 1,2,3$ let $(A_i, B_i, \phi_i, \omega_i)$ be
CP-tuples, and for $i = 1,2$ let $A_i \sa{f_i} A_{i+1}$
and $B_i \sa{g_i} B_{i+1}$ be (normal) unital *-homomorphisms,
such that the diagram
\[ \xymatrix{
A_1 \ar[r]^{f_1} \ar[d]_{\phi_1} &A_2 \ar[r]^{f_2} \ar[d]_{\phi_2} &A_3 \ar[d]_{\phi_3} \\
B_1 \ar[r]^{g_1} \ar[rd]_{\omega_1} &B_2 \ar[r]^{g_2} \ar[d]_{\omega_2} &B_3 \ar[ld]^{\omega_3} \\
& \com &
} \]
commutes.  Then
\[
(f_2 \circ f_1) \star (g_2 \circ g_1)
= (f_2 \star g_2) \circ (f_1 \star g_1).
\]
\end{proposition}

Next, we note that the Sauvageot retraction possesses a certain
universal property.

\begin{proposition} \label{propuniqueretraction}
Let $(A, B, \phi, \omega)$ be a CP-tuple, with Sauvageot product
$A \star B$ and retraction $\theta: A \star B \to B$.
Suppose $\hat{\theta}: A \star B \to B$ is another (normal) retraction
with respect to $\dss{\psi}{R}$ such that
$(A \star B, \dss{\psi}{L}, \dss{\psi}{R},
\hat{\theta} \circ \dss{\psi}{R}, \varpi)$ is a right-liberating representation for $(A, B, \phi, \omega)$.  Then
$\hat{\theta} = \theta$.
\end{proposition}

\begin{proof}
Applying Theorem \ref{thmrightlibmoments} to
the conditional expectations $\dss{\psi}{R} \circ \theta$
and $\dss{\psi}{R} \circ \hat{\theta}$, we see that
they agree on a dense *-subalgebra of $A \star B$, hence
on the whole by continuity.  Since $\dss{\psi}{R}$ is
injective, this implies $\theta = \hat{\theta}$.
\end{proof}
\section{Trivial Cases of the Sauvageot Product}
Tensor products have the property that $A \otimes \com
\simeq A \simeq \com \otimes A$ for any commutative unital
C$^*$-algebra $A$; similarly, unital free products have
the property that $A * \com \simeq A \simeq \com * A$
for any unital C$^*$-algebra $A$.  Moreover, amalgamated
free products satisfy $A *_A A \simeq A$.  We
now consider analogues of these properties for
the Sauvageot product.  These are of interest not only
for their own sake, but also as the base cases in the inductive
system of the next chapter.

\begin{proposition}[$\com \star \AAA \simeq \AAA$] \label{propcomstarA}
Let $\AAA$ be any unital C$^*$-algebra (resp.\ W$^*$-algebra),
$\upsilon: \com \to \AAA$ the embedding of $\com$,
and $\omega$ any state (resp. normal state) on $\AAA$.
Then the Sauvageot product $\com \star \AAA$ of the CP-tuple
$(\com, \AAA, \upsilon, \omega)$ is isomorphic
to $\AAA$; modulo this identification,
the embedding $\dss{\psi}{L}:
\com \to \com \star \AAA$ is $\upsilon$,
and $\dss{\psi}{R}: \AAA \to \com \star \AAA$
and $\E: \com \star \AAA \to \AAA$ are both the
identity map.
\end{proposition}

\begin{proof}
One can prove this by constructing a representation of
this CP-tuple; on the space $\Hh$, one
has $\dss{\psi}{L}$ mapping into $\dss{\psi}{R}(\AAA)$,
so that the algebra generated by both the images together
is isomorphic to $\AAA$.  Alternatively, right-liberation becomes trivial when one of the algebras involved is $\com$, so that $\E$ is multiplicative and hence is a *-homomorphic inverse for
$\dss{\psi}{R}$.
\end{proof}

\begin{remark} \label{remproductwithsubalgebra}
One might conjecture that, more generally, the Sauvageot product with respect to an embedding is trivial; that is, if
$A \sa{\iota} B$ is an embedding, or equivalently
if $A \subset B$ is an embedding, that $A \star B \simeq B$.

This turns out not to be the case.  We are interested in
whether $\dss{\psi}{L} = \dss{\psi}{R} \circ \iota$;
but on the subspace $L'$ in a faithful decomposition,
$\dss{\psi}{L}$ acts faithfully, whereas $\dss{\psi}{R}
\circ \iota$ acts in a trivial fashion (in particular,
the component in $L'$ of $\dss{\psi}{R}(\iota(a)) \xi$
for $\xi \in L'$ must be a scalar multiple of $\xi$).

This illustrates an important feature of the Sauvageot
product.  If we were to start by representing
$B$ on some $H$ through the GNS construction, then
use Stinespring dilation to obtain a representation of
$A$ on $K$, then
in the special case that the map from $A$ to $B$ is an embedding
(indeed, any homomorphism) one would have $K = H$ and therefore
$L = \{0\}$, from which it would follow that $\Hh \simeq H$
as well, and $A \star B \simeq B$.  But the Sauvageot product
is defined with respect to a \emph{faithful} representation, which involves taking direct sums at various points in the process so as
to avoid collapsing into triviality.
\end{remark}

\begin{proposition}[$\AAA \star \com \simeq \com$]
Let $\AAA$ be any unital C$^*$-algebra (resp.\ W$^*$-algebra),
and $\omega$ any state (resp.\ normal state) on $\AAA$.
Then the Sauvageot product $\AAA \star \com$ of the
CP-tuple $(\AAA, \com, \omega, \text{id}{\com})$ is
isomorphic to $\AAA$; modulo this identification,
the left embedding $\dss{\psi}{L}: \AAA \to \AAA \star \com$
is the identity map, the right embedding
$\dss{\psi}{R}: \com \to \AAA \star \com$ is $\upsilon$,
and the retraction $\E: \AAA \star \com \to \com$
is $\omega$.
\end{proposition}

\begin{proof}
As with the previous proposition.
\end{proof}

\begin{remark} \label{remstarwithcom}
Now given a CP-tuple $(A, B, \phi, \omega)$, one
can identify $A$ with $\com \star A$ (resp.\ $A \star \com$)
and $B$ with $\com \star B$ (resp.\ $B \star \com$); it
is then natural to ask whether $\phi$ is thereby identified
with $\dss{\text{id}}{\com} \star \phi$ (resp.\
$\phi \star \dss{\text{id}}{\com}$).  The answer is yes;
indeed, this is a special case of Corollary \ref{corcommutingsquare}.
\end{remark}


\chapter{Algebraic C$^*$-Dilations through Iterated Products} \label{chapiteratedproducts}

\section{Introduction}
Having shown how to construct the Sauvageot product of a CP-tuple,
we now broach the question of how to iterate this product in order to construct dilations.  For motivation, we return again to the Daniell-Kolmogorov construction as viewed through the lens of the tensor product (Example \ref{exmarkovdilation}).

Recall that we begin with a compact Hausdorff space $S$ (the state
space of a Markov process), with corresponding path space
$\SSSS = S^{[0,\infty)}$; we use $\AAA$ to denote $C(S)$ and
$\Aa$ to denote $C(\SSSS)$, though we seek here to
construct $\Aa$ only through C$^*$-algebraic means, without
reference to $\SSSS$ except as a guide to understanding.  For
each finite subset $\gamma \subset [0,\infty)$, we let
$\AAA_\gamma$ denote a tensor product of $|\gamma|$ copies
of $C(S)$ with itself.  When we have constructed $\Aa$,
we will embed such an $\AAA_\gamma$ into it,
representing those functions on the path space which only depend on times in $\gamma$.

For $\beta \leq \gamma$ we can embed $\AAA_\beta$ into
$\AAA_\gamma$ by tensoring with $\one$'s in all the missing coordinates. It is difficult to find notation which makes this more precise while maintaining the basic simplicity of the concept, but here are two attempts.  First, an example.  If $\gamma = \{t_1, \dots, t_7\}$ with
the times listed in increasing order, and $\beta = \{t_2, t_5, t_6\}$, then one embeds $\AAA_\beta$ into $\AAA_\gamma$ via
\[
f \otimes g \otimes h \longmapsto \one \otimes f \otimes \one
\otimes \one \otimes g \otimes h \otimes \one.
\]
Second, a general observation: Such an embedding can be built from repeated embeddings corresponding to adding a single time, so we
may reduce to the case $\beta = \{t_1, \dots, t_n\}$ and
$\gamma = \{t_1, \dots, t_k, \tau, t_{k+1}, \dots, t_n\}$ where
again we assume the times are in increasing order.  In this
case the embedding is
\[
f_1 \otimes \dots \otimes f_n
\longmapsto f_1 \otimes \dots \otimes f_k
\otimes \one \otimes f_{k+1} \otimes \dots \otimes f_n.
\]
It is easy to see that the family of embeddings under consideration form an inductive system, so that we may take the limit to obtain a
C$^*$-algebra $\Aa$ generated by copies of each $\AAA_\gamma$.

We note in passing that the limit construction becomes even
simpler when viewed through the lens of the Gelfand functor.
Since $\AAA_\gamma$ may be identified with $C(S^\gamma)$, one
can consider the inverse system of compact Hausdorff spaces
$\{S^\gamma\}$ equipped with the canonical projections.  The
projective (aka inverse) limit is the path space $\SSSS$,
so applying the
contravariant equivalence of categories, the inductive
(aka direct) limit
of the corresponding embeddings is isomorphic to $C(\SSSS)$.  While elegant, however, this point of view will be of little use in our noncommutative generalizations, since there is no underlying path space to work with.

Having constructed the limit algebra $\Aa$, with the embedding
$\AAA \hookrightarrow \Aa$ corresponding to the identification of $\AAA$ with $\AAA_{\{0\}}$, we are left with the task of constructing the retraction $\E: \Aa \to \AAA$.  We do this by first constructing a consistent family of retractions $\AAA_\gamma \to \AAA_\beta$
for $\beta \leq \gamma$, then showing how to use a limiting process
to induce the retraction $\Aa \to \AAA$.  First, we reduce as
before to the case where $\gamma$ contains one more point
than $\beta$, then retract
\[
f_1 \otimes \dots \otimes f_k \otimes g
\otimes f_{k+1} \otimes \dots \otimes f_n
\longmapsto f_1 \otimes \dots \otimes (f_k P_{\tau-t_k} g)
\otimes f_{k+1} \otimes \dots \otimes f_n.
\]
Note that in particular, when $\gamma$ contains 0 and
one identifies $\AAA$ with $\AAA_{\{0\}}$, repeated application
of this rule yields the retraction $\AAA_\gamma \to \AAA$ given
on simple tensors by
\[
f_1 \otimes \dots \otimes f_n
\longmapsto f_1 P_{t_2-t_1} \Big( f_2 P_{t_3-t_2}
\big(f_3 \cdots P_{t_n-t_{n-1}}(f_n)\big) \cdots \Big).
\]
Again, one can check that this family of retractions is consistent with the inductive system, so that it yields a well-defined and contractive map onto $\AAA$ from the dense subalgebra of $\Aa$ generated by the images of all the $\AAA_\gamma$; as this map is contractive, it
extends to a retraction on all of $\Aa$.

When seeking to carry this method across to the Sauvageot product,
one runs into several hurdles.  First, one does not form the Sauvageot product merely of two C$^*$-algebras, but rather of a CP-tuple; hence, one cannot begin by defining $\AAA_\gamma = \AAA \star \dots \star \AAA$
without specifying what maps are used between the various copies of $\AAA$.  A related and deeper problem is the failure of associativity; even when the relevant maps have been selected to make the notation well-defined, in general one does not have $(A \star A) \star (A \star A)$ isomorphic to $((A \star A) \star A) \star A$.
Hence, we are led to adopt a more laborious inductive construction, though we follow the same high-level strategy as in the commutative case.

For the remainder of the chapter, we fix a unital C$^*$-algebra (resp. W$^*$-algebra) $\AAA$, a faithful state (resp. faithful normal state)
$\omega$ on $\AAA$, and a cp$_0$-semigroup $\{\phi_t\}$ on $\AAA$.
We use $\FF$ to denote the set of finite subsets of $[0,\infty)$.
Throughout, we assume unless otherwise indicated that times within
such sets are listed in increasing order; hence, writing
$\gamma = \{t_1, \dots, t_n\}$ implies $t_1 < \dots < t_n$.

\section{Construction of the Inductive System and Limit}
\subsection{Objects and Immediate-Tail Morphisms}
\begin{definition} \label{deftail}
Let $\beta, \gamma \in \FF$ with
$\gamma = \{t_1,  \dots, t_n\}$.  We call $\beta$ an
\textbf{initial segment} of $\gamma$ if
$\beta = \{t_1, \dots, t_m\}$ for some $1 \leq m \leq n$,
and a \textbf{tail}
of $\gamma$ if $\beta = \{t_\ell, \dots t_n\}$
for some $1 \leq \ell \leq n$.
If $\ell = 2$ we call $\beta$ an \textbf{immediate tail}
with \textbf{distance} $t_2 - t_1$.
\end{definition}

We are now able to define the objects of our inductive system,
as well as some of the morphisms.

\begin{definition} \label{definductiveobjects}
For nonempty $\gamma \in \FF$ we define inductively
\begin{enumerate}
    \item a unital C$^*$-algebra (resp. W$^*$-algebra)
    $\AAA_\gamma$
    \item a unital embedding $\iota_\gamma: \AAA \to \AAA_\gamma$
    \item a retraction $\epsilon_\gamma: \AAA_\gamma \to \AAA$
\end{enumerate}
as follows:
\begin{itemize}
    \item If $\gamma$ is a singleton,
    then $\AAA_\gamma = \AAA$ and both $\iota_\gamma$
     and $\epsilon_\gamma$ are the identity.
    \item If $\beta$ is an immediate tail of $\gamma$
    with distance $\tau$, let $\Phi
    = \phi_\tau \circ \epsilon_\beta: \AAA_\beta \to \AAA$, and
    form the CP-tuple $(\AAA_\beta, \AAA, \Phi, \omega)$.
    Then $\AAA_\gamma$ is the Sauvageot product
    $\AAA_\beta \star \AAA$, $\iota_\gamma$ is the
    embedding of $\AAA$ into $\AAA_\beta \star \AAA$ (denoted
    $\dss{\psi}{R}$ in the previous chapter), and
    $\epsilon_\gamma$ is the Sauvageot retraction from
    $\AAA_\beta \star \AAA$ onto $\AAA$.
\end{itemize}
We also define $\AAA_\emptyset = \com$.
\end{definition}

Note that this definition also implicitly gives us
embeddings $\AAA_\beta \hookrightarrow \AAA_\gamma$ in the special
case where $\beta$ is an immediate tail of $\gamma$; this is
just the canonical embedding of $\AAA_\beta$ into $\AAA_\beta
\star \AAA$, the map denoted in the previous chapter
by $\dss{\psi}{L}$.

We turn
next to the question of how to embed $\AAA_\beta$ into
$\AAA_\gamma$ when $\beta \leq \gamma$ more generally.

\subsection{General Morphisms}
Consider now any inclusion $\beta \leq \gamma$
of nonempty elements of $\FF$.
Let $\gamma = \{t_1, \dots, t_n\}$ and for each $\ell \in \{1,
\dots, n\}$ define subsets $\gamma(\ell) \leq \gamma$
and $\beta(\ell) \leq \beta$ by
\[
\gamma(\ell) = \gamma \cap \{t_\ell, \dots, t_n\},
\quad \beta(\ell) = \beta \cap \{t_\ell, \dots, t_n\}.
\]
Then each $\gamma(\ell)$ is a tail of $\gamma$, with
$\gamma(1) = \gamma$, and similarly for $\beta$.
(Note that some of the $\beta(\ell)$ may be empty, if
$t_n \notin \beta$.)

\begin{definition} \label{definductivemorphisms}
For $\beta, \gamma$ as above, we define an embedding
$\AAA_\beta \sa{f} \AAA_\gamma$ by recursively
defining embeddings $\AAA_{\beta(\ell)} \sa{f_\ell} \AAA_{\gamma(\ell)}$
and letting $f = f_1$.  The embeddings are as follows:
\begin{itemize}
    \item In the base case $\ell = n$, the embedding $f_n$ is
    the identity map in case $t_n \in \beta$, or the canonical
    embedding $\com \hookrightarrow \AAA$ otherwise.
    \item Given $f_{\ell+1}$,
    let $\BB$ denote either $\AAA$ in the case that
    $t_\ell \in \beta$, or $\com$ otherwise;
    more succinctly, $\BB = \AAA_{\beta \cap \{t_\ell\}}$.  Let
    $\BB \sa{\psi} \AAA$ be either the identity map
    or the embedding of $\com$, accordingly.  Then
    \[
    f_{\ell} = f_{\ell+1} \star \psi.
    \]
\end{itemize}
\end{definition}

\begin{proposition} \label{propinductivesystem}
The family of embeddings $\AAA_\beta \hookrightarrow \AAA_\gamma$
in Definition \ref{definductivemorphisms} is an inductive system.
\end{proposition}

\begin{proof}
Let $\beta \leq \gamma \leq \delta$ be nonempty sets in $\FF$.
Write $\delta = \{t_1, \dots, t_n\}$.  We first prove
 that the embedding $\AAA_\delta \hookrightarrow \AAA_\delta$
 is the identity map.  We prove this for the embeddings
 $\AAA_{\delta(\ell)} \hookrightarrow \AAA_{\delta(\ell)}$
 by reverse induction; the base case $\ell=n$ is trivial,
 and the inductive step is just Corollary \ref{corunique}.

Now for each $\ell = 1, \dots, n$ let
\begin{align*}
\AAA_{\beta(\ell)} &\longsa{g_\ell} \AAA_{\gamma(\ell)}\\
\AAA_{\gamma(\ell)} &\longsa{f_\ell} \AAA_{\delta(\ell)}\\
\AAA_{\beta(\ell)} &\longsa{h_\ell} \AAA_{\delta(\ell)}
\end{align*}
be the embeddings from Definition \ref{definductivemorphisms}.
We will prove by reverse induction for $\ell = n, \dots, 1$
that $f_\ell \circ g_\ell = h_\ell$.  The base case $\ell=n$
is trivial, as each of the three maps
in question is either the identity map or the embedding
$\com \hookrightarrow \AAA$.  Supposing now the result
to be established for $\ell+1$, let $\BB = \AAA_{\beta \cap \{t_\ell\}}$
and $\CC = \AAA_{\gamma \cap \{t_\ell\}}$, and let
$\BB \sa{\psi} \CC \sa{\eta} \AAA$ be the corresponding
embeddings.  Then  by Proposition \ref{propfunctorial},
\[
    h_\ell = h_{\ell+1} \star (\eta \circ \psi)
    = (f_{\ell+1} \circ g_{\ell+1}) \star(\eta \circ \psi)
    = (f_{\ell+1} \star \eta)
    \circ (g_{\ell+1} \star \psi)
    = f_\ell \circ g_\ell.
    \]
\end{proof}

\section{Endomorphisms of the Limit Algebra}
We have now constructed unital C$^*$-algebras (resp.\
W$^*$-algebras) $\AAA_\gamma$ for each $\gamma \in \FF$, together
with (normal) embeddings $\AAA_\beta \hookrightarrow \AAA_\gamma$
for $\beta \leq \gamma$, which we now denote $\dss{f}{\gamma, \beta}$,
satisfying the inductive properties
\begin{align*}
\dss{f}{\gamma, \gamma} &= \text{id}_{\AAA_\gamma}\\
\dss{f}{\delta, \beta} &= \dss{f}{\delta,\gamma} \circ
\dss{f}{\gamma, \beta} \qquad \text{ for }
\beta \leq \gamma \leq \delta.
\end{align*}
By a standard construction (see for instance section 1.23 of \cite{Sakai}, Proposition 11.4.1 of \cite{KadisonRingrose2},
or section II.8.2 of \cite{Blackadar}) we obtain an inductive
limit, that is, a unital C$^*$-algebra $\Aa$ and
embeddings $\dss{f}{\infty,\gamma}: \AAA_\gamma \to \Aa$
such that $\dss{f}{\infty,\gamma} \circ \dss{f}{\gamma, \beta}
= \dss{f}{\infty,\beta}$ for all $\beta \leq \gamma$, and
with the universal property that, given any other unital C$^*$-algebra
$\Bb$ and *-homomorphisms (not necessarily embeddings)  $\dss{g}{\infty,\gamma}: \AAA_\gamma \to \Bb$ satisfying $\dss{g}{\infty,\gamma}
\circ \dss{f}{\gamma, \beta} = \dss{g}{\infty,\beta}$, there
is a unique unital *-homomorphism $\Phi: \Aa \to \Bb$
satisfying $\dss{g}{\infty,\gamma} = \Phi \circ \dss{f}{\infty,\gamma}$
for all $\gamma$.

We note that inductive limits do not always exist in the category
of W$^*$-algebras and normal *-homomorphisms; hence, $\Aa$
will not in general be a W$^*$-algebra even when $\AAA$ is.  There are
a couple of standard ``fixes'' we could apply to replace $\Aa$
by a suitable W$^*$-algebra: First, we could represent $\Aa$ on
some Hilbert space (for instance, by taking an inductive limit
of appropriately intertwined representations of the
$\AAA_\gamma$ as in \cite{KadisonRingrose2} Exercise 11.5.28) and take the weak closure.  However, the limit
representation will in general be on a non-separable Hilbert space,
which creates problems further on.  Second, we could
apply the double dual functor (\cite{Blackadar} III.5.2,
\cite{Sakai} 1.17) to obtain a W$^*$-algebra $\Aa^{**}$
and normal embeddings $\dss{f}{\gamma}^{**} \circ \dss{\iota}{\gamma}: \AAA_\gamma
\to \Aa^{**}$, where $\AAA_\gamma \sa{\iota_\gamma}
\AAA_\gamma^{**}$ is the canonical embedding.
However, this also creates problems,
as $\Aa^{**}$ will in general have non-separable predual
even if each of the $\AAA_\gamma$ has separable predual.
We postpone until the next chapter the question of how to
adapt our construction to the W$^*$-category, and continue
for the time being with a purely C$^*$-construction.

Our next task is to define a semigroup of unital *-endomorphisms
of $\Aa$.  For this we note that for any $\gamma \in \FF$
and any $\tau \geq 0$, if we let $\gamma + \tau$ denote
the set $\{t + \tau \mid t \in \gamma\}$, then
$\AAA_{\gamma+\tau} = \AAA_\gamma$.  (Note that this
is an \emph{equality}, not just an isomorphism.)  This
is immediate from Definition \ref{definductiveobjects} by induction
on the size of $\gamma$.  Similarly,
$\dss{f}{\gamma+t, \beta+t} = \dss{f}{\gamma, \beta}$.  But
this latter equation implies that $\dss{f}{\infty, \gamma+t}
\circ \dss{f}{\gamma, \beta} = \dss{f}{\infty, \beta+t}$ for
any $\beta \leq \gamma$, allowing us to make the following definition.

\begin{definition} \label{defsigmat}
For each $t \geq 0$ let $\sigma_t: \Aa \to \Aa$ denote the unique
unital *-endomorphism obtained through the inductive limit
as the unique map for which all the diagrams
\[ \xymatrix{
\AAA_\gamma \ar[rd]_{f_{\infty,\gamma+t}} \ar[r]^{f_{\infty,\gamma}} &\Aa \ar[d]^{\sigma_t}\\
& \Aa
} \]
commute.
\end{definition}

The universal property of the inductive limit then immediately
implies the following.
\begin{proposition} \label{prope0semigroup}
The maps $\{\sigma_t\}_{t \geq 0}$ form an e$_0$-semigroup on $\Aa$.
That is, $\sigma_0 = \dss{\text{id}}{\Aa}$, and for all
$s,t \geq 0$,
\[
\sigma_t \circ \sigma_s = \sigma_{s+t}.
\]
\end{proposition}

\section{The Limit Retraction}
We now turn to the construction of our retraction.
In the commutative analogue, for a set $\gamma$ with
minimum time $\tau$, the retraction $\epsilon_\gamma$
would (when composed with the embedding $\AAA \hookrightarrow
\Aa$) correspond to a conditional expectation onto the
subalgebra of $\Aa$ consisting of functions which depend only
on the location of a path at time $\tau$.  This does not
form a consistent system with respect to the embeddings
$\dss{f}{\gamma, \beta}$, because for $\beta \leq \gamma$
one could have times in $\gamma$ earlier than any in $\beta$.
However, the restriction to time sets which contain 0 is
consistent, which we now show in the noncommutative case.
We first consider how to relate the retraction for a given
set to the retractions for its tails.

\begin{lemma} \label{lemtailretractions}
Let $\gamma = \{t_1, \dots, t_n\} \in \FF$ and
$1 \leq \ell \leq n$.  Then
\[
\epsilon_\gamma \circ \dss{f}{\gamma, \gamma(\ell)}
= \phi_{t_\ell-t_1} \circ \epsilon_{\gamma(\ell)}.
\]
\end{lemma}

\begin{proof}

We proceed by (forward!) induction on $\ell$.
The base case $\ell=1$ is trivial.  Now supposing
the result is true for $\ell$, recall that $\AAA_{\gamma(\ell)}$
is the product $\AAA_{\gamma(\ell+1)} \star \AAA$
with respect to the map $\phi_{t_{\ell+1}-t_\ell}
\circ \epsilon_{\gamma(\ell+1)}: \AAA_{\gamma(\ell+1)} \to \AAA$, that $\dss{f}{\gamma(\ell),\gamma(\ell+1)}$ is the
embedding of $\AAA_{\gamma(\ell+1)}$ into this product, and that $\epsilon_{\gamma(\ell)}$ is the Sauvageot retraction.  By Equation \ref{eqnthetafactorsphi} we therefore have
\[
\epsilon_{\gamma(\ell)} \circ \dss{f}{\gamma(\ell),\gamma(\ell+1)}
= \phi_{t_{\ell+1}-t_\ell} \circ \epsilon_{\gamma(\ell+1)}
\]
so that
\begin{align*}
\epsilon_\gamma \circ \dss{f}{\gamma,\gamma(\ell+1)}
&= \epsilon_\gamma \circ \dss{f}{\gamma,\gamma(\ell)}
\circ \dss{f}{\gamma(\ell),\gamma(\ell+1)}\\
&= \phi_{t_\ell-t_1} \circ \epsilon_{\gamma(\ell)}
\circ \dss{f}{\gamma(\ell),\gamma(\ell+1)}\\
&= \phi_{t_\ell-t_1} \circ \phi_{t_{\ell+1}-t_\ell} \circ
\epsilon_{\gamma(\ell+1)}\\
&= \phi_{t_{\ell+1}-t_1} \circ \epsilon_{\gamma(\ell+1)}.
\end{align*}
\end{proof}

\begin{proposition} \label{propconsistentwith0}
Let $\beta \leq \gamma \in \FF$ such that the
minimum time in $\gamma$ is also in $\beta$.  Then
\[
\epsilon_\gamma \circ \dss{f}{\gamma,\beta} = \epsilon_\beta.
\]
\end{proposition}

\begin{proof}
Let $\gamma = \{t_1, \dots, t_n\}$.
We will prove
that
\[
\epsilon_{\gamma(\ell)}
\circ \dss{f}{\gamma(\ell),\beta(\ell)} = \epsilon_{\beta(\ell)}
\]
for all $\ell$ such that $t_\ell \in \beta$.
For the base case with the maximal such $\ell$,
 $\dss{f}{\gamma(\ell), \beta(\ell)}$ is equal
to $\iota_{\gamma(\ell)}$, and since $\epsilon_{\gamma(\ell)}$
is a corresponding retraction, their composition is
$\dss{\text{id}}{\AAA} = \epsilon_{\beta(\ell)}$.
Inductively, suppose $t_\ell \in \beta$ and $t_{\ell+k}$ is
the next time in $\beta$, so that $\beta(\ell+1) =
\beta(\ell+k)$; then
\begin{align*}
\phi_{t_{\ell+1}-t_\ell} \circ \epsilon_{\gamma(\ell+1)}
\circ \dss{f}{\gamma(\ell+1),\beta(\ell+1)} &=
\phi_{t_{\ell+1}-t_\ell} \circ \epsilon_{\gamma(\ell+1)}
\circ \dss{f}{\gamma(\ell+1),\beta(\ell+k)} && \quad (\beta(\ell+1) = \beta(\ell+k)) \\ &=\phi_{t_{\ell+1}-t_\ell} \circ \epsilon_{\gamma(\ell+1)}
\circ \dss{f}{\gamma(\ell+1),\gamma(\ell+k)} \circ
\dss{f}{\gamma(\ell+k),\beta(\ell+k)} && \quad (\text{$f$'s consistent})\\
&= \phi_{t_{\ell+1}-t_\ell} \circ \phi_{t_{\ell+k}-t_{\ell+1}}
\circ \epsilon_{\gamma(\ell+k)} \circ \dss{f}{\gamma(\ell+k),\beta(\ell+k)} && \quad (\text{Lemma } \ref{lemtailretractions})\\
&= \phi_{t_{\ell+k}-t_\ell} \circ \epsilon_{\beta(\ell+k)} && \quad (\text{induction}).
\end{align*}
It then follows from Corollary \ref{corcommutingsquare} that
\[
\epsilon_{\gamma(\ell)} \circ \dss{f}{\gamma(\ell),
\beta(\ell)} = \epsilon_{\gamma(\ell)} \circ
(\dss{f}{\gamma(\ell+1),\beta(\ell+1)} \star \text{id})
= \epsilon_{\beta(\ell)}
\]
as desired.  The case $\ell=1$ gives us the result.
\end{proof}

\begin{corollary} \label{corconsistentretractions}
The restriction of the family of retractions $\{\epsilon_\gamma\}$ to the subset $\FF_0 \subset \FF$ of sets containing 0
is consistent.
\end{corollary}

Since $\FF_0$ is a tail of $\FF$, the limit $\Aa$ is generated
by images of $\AAA_\gamma$ with $\gamma \in \FF_0$.
Hence, Corollary \ref{corconsistentretractions} implies
the existence of a retraction $\E: \Aa \to \AAA$
with the property that $\E \circ \dss{f}{\infty,\gamma}
= \epsilon_\gamma$ for all $\gamma \in \FF_0$.

\begin{definition} \label{defdilationretraction}
The \textbf{Sauvageot dilation retraction} for
$(\AAA, \{\phi_t\}, \omega)$ is the map
$\E: \Aa \to \AAA$ characterized by
\[
\E \circ \dss{f}{\infty, \gamma} = \epsilon_\gamma
\qquad \text{ for all } 0 \in \gamma \in \FF.
\]
\end{definition}

We now prove that $(\E, \{\sigma_t\})$ provides a strong dilation
of the semigroup $\{\phi_t\}$.

\begin{theorem} \label{thmstrongdilation}
For all $t \geq 0$,
\[
\E \circ \sigma_t = \phi_t \circ \E.
\]
\end{theorem}

\begin{proof}
The case $t = 0$ is trivial.  Now
let $\gamma \in \FF$ be nonempty and $t > 0$.  Let
$\delta = (\gamma+t) \cup \{0\}$; then $\AAA_\delta$ is
the Sauvageot product $\AAA_{\gamma+t} \star \AAA$ with
respect to the map $\phi_t \circ \epsilon_\gamma$.  By
Equation \ref{eqnthetafactorsphi}, it follows
that
\[
\epsilon_\delta \circ \dss{f}{\delta, \gamma+t}
= \phi_t \circ \epsilon_\gamma.
\]
Then
\begin{align*}
\E \circ \sigma_t \circ \dss{f}{\infty,\gamma}
&= \E \circ \dss{f}{\infty,\gamma+t}\\
&= \E \circ \dss{f}{\infty,\delta} \circ \dss{f}{\delta,\gamma+t}\\
&= \epsilon_\delta \circ \dss{f}{\delta,\gamma+t}\\
&= \phi_t \circ \epsilon_\gamma\\
&= \phi_t \circ \E \circ \dss{f}{\infty,\gamma}.
\end{align*}
So $\E \circ \sigma_t$ and $\phi_t \circ \E$ agree on
the dense subalgebra of $\Aa$ consisting of the images of
all the $\dss{f}{\infty,\gamma}$; as both are
contractive, they are equal.
\end{proof}

This concludes our construction of unital e$_0$-dilations
for cp$_0$-semigroups on C$^*$-algebras.  We summarize the
result in the following theorem.

\begin{theorem}
Let $\AAA$ be a unital C$^*$-algebra on which there exists
a faithful state.  Then every cp$_0$-semigroup on $\AAA$
has a strong unital e$_0$-dilation.
\end{theorem}

\chapter{Continuous W$^*$-Dilations}\label{chapcontinuous}
In the previous chapter we saw how to construct a unital e$_0$-dilation of a cp$_0$-semigroup.  It remains to investigate whether such a construction dilates a continuous semigroup to a continuous semigroup (that is, whether it produces a unital E$_0$-dilation of a CP$_0$-semigroup), or, failing that, whether the construction can be modified to achieve this result.  Additionally, we have not yet resolved the question of how to adapt our C$^*$ construction to the W$^*$ setting.  To these issues we now turn our attention.

\section{Introduction: The Problem of Continuity}
\subsection{Establishing the Problem: Discontinuity of the Existing Dilation}
The first question to consider is whether the existing dilation
may already be continuous.  It turns out that this is
\emph{never} the case unless $\AAA = \com$.  Consider a
nontrivial $\AAA$ with faithful state $\omega$, and
let $a$ be any nonzero element of $\ker \omega$.
Fixing some faithful representation $(H, \Omega, \dss{\pi}{R})$
of $(\AAA, \omega)$, let $h = \dss{\pi}{R}(a) \Omega$, which
is orthogonal to $\Omega$.  For each $t > 0$ there
is a faithful representation $(H, \Omega, \dss{\pi}{R},
K^{(t)}, V^{(t)}, \dss{\pi}{L}^{(t)})$ of $(A, A, \phi_t, \omega)$.
Form the Sauvageot product $\Hh^{(t)} = H^-
\star L^{(t)}$, and let $\xi$ be any unit vector in $L^{(t)}$.  By
Proposition \ref{prop3UEreps} we see that $\dss{\psi}{L}^{(t)}(a) \xi$ is a vector in $L^{(t)}$, whereas
$\dss{\psi}{R}^{(t)}(a) \xi = h \otimes \xi$ is
in $H^- \otimes L^{(t)}$.  Since these are orthogonal subspaces
of $\Hh^{(t)}$,
\[
\|\dss{\psi}{L}^{(t)}(a) \xi - \dss{\psi}{R}^{(t)}(a) \xi\|
\geq \|h \otimes \xi\| = \|h\| \|\xi\|
\]
which implies
\[
\|\dss{\psi}{L}^{(t)}(a) - \dss{\psi}{R}^{(t)}(a) \| \geq
\|h\|.
\]
Now letting $\gamma = \{0,t\}$, we have $\AAA_\gamma$ as the
Sauvageot product $\AAA \star \AAA$ with respect to $\phi_t$,
so that $\dss{\psi}{L}^{(t)}(a) - \dss{\psi}{R}^{(t)}(a)$
is the element $\dss{f}{\gamma, \{t\}}(a) -
\dss{f}{\gamma, \{0\}}(a)$ of $\AAA_\gamma$.  By
the above, this element has norm at least $\|h\|$.
Now because $\dss{f}{\infty, \gamma}$ is isometric,
\begin{align*}
\| \sigma_t(\iota(a)) - \iota(a) \|
&= \|\dss{f}{\infty, \{t\}}(a) - \dss{f}{\infty, \{0\}}(a)\|\\
&= \left\|\dss{f}{\infty,\gamma} \Big(
\dss{f}{\gamma, \{t\}}(a) - \dss{f}{\gamma, \{0\}}(a)\Big) \right\|\\
&= \|\dss{f}{\gamma, \{t\}}(a) - \dss{f}{\gamma, \{0\}}(a)\|
\geq \|h\|.
\end{align*}
It follows that $\|\sigma_t(\iota(a)) - \iota(a)\| \not \to 0$
as $t \to 0^+$.

Upon further reflection, the discontinuity of $\{\sigma_t\}$ is
not surprising, because it appears in the commutative dilation that the Sauvageot construction mimics.  Considering again the case
$\AAA = C(S)$, $\Aa = C(\mathscr{S})$ of Example
\ref{exmarkovdilation}. Given a regular Borel
probability measure $\mu_0$ on $S$, we obtain via Riesz representation
a regular Borel probability measure $\mu$ on $\SSSS$ characterized by
\[
\forall f \in \Aa: \qquad \int_\mathscr{S} f \, d\mu = \int_S (\E f) \, d\mu_0.
\]
Now consider the strong, aka point-norm, continuity of the shift semigroup.  For a function $f \in \Aa$, we want to know whether $\lim_{t \to 0^+} \|\sigma_t f - f\| = 0$.  We will show that in fact a less stringent form of continuity,
fails, viz.\ the ``point-pointwise'' continuity defined
by the
property that for any fixed path $\pP \in \SSSS$ and any
$f \in \Aa$,
$(\sigma_t f - f)(\pP) \to 0$.  The failure of point-pointwise
continuity certainly implies the failure of point-norm continuity.  Now
let $\pP$ be any path not continuous at time 0, let
$\phi: S \to [0,1]$ be a continuous function such that
$\phi(\pP(t)) \not \to \phi(\pP(0))$ as $t \to 0^+$ (which
exists by Urysohn's lemma),
and let $f \in \Aa$ be defined by $f(p) = \phi(p(0))$.
Then
\[
\lim_{t \to 0^+} (\sigma_t f - f)(\pP)
= \lim_{t \to 0^+} \phi(\lambda_t \pP) - \phi(\pP)
= \lim_{t \to 0^+} \phi(\pP(t)) - \phi(\pP(0)) \neq 0.
\]

\subsection{Fixing the Problem: Skorohod Space?}
This section is not used in the rest of the thesis, but
is mentioned merely for the sake of interest.

The preceding considerations show that, in the commutative setting,  continuity of $\{\sigma_t\}$ breaks down because the path space
$\SSSS$ contains discontinuous paths.  One could try to ``repair'' the
construction by working instead with something like $C(\Cc)$,
where $\Cc \subset \SSSS$ is the subspace of continuous paths.
This runs into problems, however, because $\Cc$ is not
closed in $\SSSS$ (the pointwise limit of continuous functions
need not be continuous), hence not compact, so that one cannot
form the commutative unital C$^*$-algebra $C(\Cc)$.  One
could instead endow $\Cc$ with the topology of uniform convergence;
although it is complete in the corresponding metric, however,
it is not compact nor even locally compact, so that attempts
to form commutative C$^*$-algebras such as $BC(\Cc)$ run
into trouble as well.

One way out is to consider instead the Skorohod space
$D$ of \cad paths, that is, paths which are continuous from
the right and have limits from the left.  Convergence
 in $D$ can be defined as follows: For convenience we consider paths parametrized by
$[0,1)$ rather than $[0,\infty)$.  Let $\Lambda$ denote the set of all continuous strictly increasing self-maps of $[0,1]$, and
define $\gamma_n \sa{D} \gamma$ if there exists a sequence
$\{\lambda_n\} \subset \Lambda$ such that $\lambda_n \to \text{id}$ uniformly and $\gamma_n\circ \lambda_n \to \gamma$ uniformly.  Heuristically, this may be contrasted with uniform convergence as follows: Identifying paths with their graphs in $[0,1) \times S$, two paths are uniformly close if one may be obtained from the other by a small perturbation of the $S$ coordinates, whereas two paths are Skorohod-close if
one may be obtained from the other by a simultaneous small perturbation of both the $S$ \emph{and} the $[0,1)$ coordinates.

It turns out that there is a metric on $D$ which induces the aforementioned convergence, and that $D$ is separable and complete with respect to this metric  (\cite{Billingsley} chapter 14, cf.\ \cite{Kallenberg} Theorem A2.2).
Furthermore, it is easy to check that
translation is Skorohod continuous, that is, $\lambda_t \pP\sa{D} \pP$
as $t \to 0^+$ for any $\pP \in D$.  It follows that
one may define the semigroup $\{\widehat{\sigma}_t\}$ on the commutative
unital C$^*$-algebra $C(D)$ by $(\widehat{\sigma}_t f)(\pP)
= f(\lambda_t(\pP))$, and that this semigroup will be point-pointwise continuous.  The embedding $\widehat{i}: C(S) \to C(D)$ is defined
as always by $(\, \widehat{i} f)(\pP) = f(\pP(0))$.  It remains to
define a retraction $\widehat{\E}: C(D) \to C(S)$.  For this we
invoke the theorem that every Feller process has a \cad version (\cite{Kallenberg} Theorem 17.15);
that is, given a probability measure $\nu$ on $S$, one obtains a measure
$\mu_\nu$ on $D$ such that the coordinate-projection process
has $\{P_t\}$ as its transition semigroup.  The assignment
$\nu \mapsto \mu_\nu$ is a positive linear map from $M(S)$
to $M(D)$ which, one can verify, is weak-* continuous,
and
has the property that $\int_D \widehat{i}(f) \, d\mu_\nu = \int_S f \, d\nu$
for all $\nu \in M(S)$ and $f \in C(S)$;
it is therefore the adjoint of a positive linear
map $\E: C(D) \to C(S)$ with the property
$\widehat{\E}(\, \widehat{i}(f)) = f$.
Furthermore, $\widehat{\E} \circ \widehat{\sigma}_t \circ \widehat{i} = P_t$,
so that we have obtained a point-pointwise continuous dilation.

In attempting to adapt this fix to the noncommutative setting,
several obstacles present themselves:
\begin{enumerate}
    \item What is the right notion of a ``noncommutative Skorohod
    space''?  That is, given a noncommutative C$^*$-algebra $\AAA$,
    what C$^*$-algebra bears the same relation to $\AAA$ that
    $C(D)$ does to $C(S)$?
    \item What is the analogue of point-pointwise continuity in
    the noncommutative setting, when the elements of $\AAA$ (and
    of whatever algebra we dilate to) may not be functions on
    some state space or path space?
    \item How might we obtain a theorem corresponding to
    the existence of \cad versions of Feller processes?
\end{enumerate}
I do not know how to address these questions; fortunately,
another approach proved successful, so that the answers to
these questions are not needed in the rest of this thesis.

\subsection{Fixing the Problem: Dilations on $L^\infty$}
Another way to resolve the issue of continuity is to move
from the C$^*$ to the W$^*$ category, by considering maps on $L^\infty(S)$ and $L^\infty(\SSSS)$ instead of $C(S)$ and $C(\SSSS)$.  Given $\AAA = C(S)$ and $\Aa = C(\SSSS)$
as before, we now select some regular Borel probability
measure $\mu_0$ on $S$, and let $\widehat{\AAA}
= L^\infty(S, \mu_0)$.  Let $\mu$ be the corresponding
measure on $\SSSS$ as in Example \ref{exmarkovdilation},
and $\widehat{\Aa} = L^\infty(\SSSS, \mu)$.
We define the completely positive semigroup
$\{\widehat{P}_t\}$ on $\AAA$ by
\[
(\widehat{P}_t f)(x) = \int_S f(y) \, dp_{t,x}(y)
\]
which has the additional property that each map $\widehat{P}_t$
is normal: If $f_\nu \uparrow f$ are positive elements of $\AAA$, then
\[
(\widehat{P}_t f_\nu)(x) = \int_S f_\nu(y) dp_{t,x}(y)
\to \int_S f(y) \, dp_{t,x}(y) = (\widehat{P}_t f)(x)
\]
by Dominated Convergence.

 We define the semigroup $\{\widehat{\sigma}_t\}$
of normal endomorphisms of $\widehat{\Aa}$ by
$(\widehat{\sigma}_t f)(\pP)
= f (\lambda_t(\pP))$,
and the normal embedding $\hat{i}:
\AAA \to \Aa$ by $(\hat{i} f)(\gamma) = f(\gamma(0))$.  The image
of $\hat{i}$ corresponds to $L^\infty(\SSSS, \Ff_0, \mu)$, where
$\Ff$ is the Borel $\sigma$-algebra on $\SSSS$ and
$\Ff_0 \subset \Ff$ is the $\sigma$-subalgebra of sets
of the form $E \times [0,\infty)$ for a Borel subset $E \subset S$.
We therefore obtain a normal retraction $\widehat{E}: \widehat{\Aa}
\to \widehat{\AAA}$ through the Radon-Nikodym derivative, and since
$\widehat{\E} \circ \widehat{\sigma}_t \circ \widehat{i}$
and $\widehat{P}_t$ are both normal and agree on
the weak-* dense subspace $C(S) \subset L^\infty(S, \mu_0)$,
where they equal $\E \circ \sigma_t \circ i$ and $P_t$
respectively, they are equal.

We have not yet shown that $t \mapsto \widehat{\sigma}_t$
is weak-* continuous on $\widehat{\Aa}$.  Since we are interested
in the commutative case purely for heuristic purposes at this
point, we shall set aside the question of what conditions
on $S$, $\mu_0$, and $\{P_t\}$ are necessary for some
of our subsequent assumptions to hold.  Let
$\pi$ be the multiplication representation of $\AAA$
on $H = L^2(S, \mu_0)$, and assume that the weak closure
$\pi(\AAA)''$ is isomorphic to $\widehat{\AAA}$.
By Stinespring dilation of $\pi \circ \E$ we obtain
a representation $\psi$ of $\Aa$ on some $K$ (which we
could imagine to be the multiplication representation
on $L^2(\SSSS, \mu)$, but we won't use that hypothesis).
We assume that the weak closure $\psi(\AAA)''$ is isomorphic
to $\widehat{\AAA}$.  The semigroup $\widehat{\sigma}_t$
is related to $\{\sigma_t\}$ through the covariance
relation $\widehat{\sigma}_t \circ \psi = \psi \circ \sigma_t$.
Since vectors of the form $\psi(y)V\xi$ for $y \in \Aa$, $\xi \in H$ are dense in $K$, the question of WOT-continuity of $\widehat{\sigma}_t(a)$ reduces to the continuity of
the expression
\[
\la \widehat{\sigma}_t(a) \psi(y) V\xi, \psi(z) V \eta \ra
= \la V^* \psi(z)^* \widehat{\sigma}_t(a) \psi(y) V \xi, \eta \ra.
\]
When we restrict to the weakly dense subset of
$a \in \widehat{\Aa}$ of
the form $\psi(x)$ for $x \in \Aa$, this expression becomes (using the covariance relation)
\[
\la V^* \psi(z^* \sigma_t(x) y) V \xi, \eta \ra
= \la \pi \circ \E[z^* \sigma_t(x) y] \xi, \eta \ra.
\]
We are thus led to consider the WOT-continuity properties
of how the retraction $\E$ interactions with the translation
semigroup $\sigma_t$.  In particular, if we could find a way
to reduce expressions of the form $\E[z^* \sigma_t(x) y]$
to expressions involving the semigroup $\widehat{P}_t$,
we could use
the assumed continuity properties of the latter.
The search for such a reduction leads to the
concept of \textbf{moment polynomials}, which we now take up.

\section{Moment Polynomials}
In the Sauvageot C$^*$-dilation of chapter \ref{chapiteratedproducts}, the inductive limit algebra $\Aa$
is norm-generated as an algebra by elements $\sigma_t(i(a))$
for $t \geq 0$ and $a \in \AAA$.  In studying the retraction
$\E$, therefore, one is naturally led to
consider expressions of the form
\begin{equation} \label{eqnkeyexpression}
\E \Big[ \sigma_{t_1}i((a_1)) \sigma_{t_2}(i(a_2))
\dots \sigma_{t_n}(i(a_n)) \Big], \quad
t_1, \dots, t_n \geq 0; \quad a_1, \dots, a_n \in \AAA.
\end{equation}
In particular, it would be desirable to have a formula
for the value of (\ref{eqnkeyexpression}) in terms of the
original semigroup $\{\phi_t\}$ and the state $\omega$
chosen for the dilation procedure.  From the construction
of $\E$ in previous chapters, we
see that (\ref{eqnkeyexpression}) can be evaluated as
follows:
\begin{itemize}
    \item If all the $t_i$ are strictly positive, let
    $\tau$ denote the minimum; then, by Theorem
    \ref{thmstrongdilation},
    \[
    \E \Big[ \sigma_{t_1}i((a_1)) \sigma_{t_2}(i(a_2))
    \cdots \sigma_{t_n}(i(a_n)) \Big]
    = \phi_\tau \bigg( \E \big[ \sigma_{t_1-\tau}i((a_1)) \sigma_{t_2-\tau}(i(a_2))
\cdots \sigma_{t_n-\tau}(i(a_n)) \big] \bigg).
    \]

    \item If some of the $t_i$ are zero, let $\tau$ be the
    minimum of the nonzero values, and let $\gamma+\tau$
    be the set of nonzero values, where $\gamma$ is
    a finite subset of $[0,\infty)$.  Then
    $\sigma_{t_1}(i(a_1)) \cdots \sigma_{t_n}(i(a_n))$
    may be viewed as an element of $\AAA_{\{0\} \cup (\gamma+\tau)}$, which we further break down
    as a word in $\AAA$ and $\AAA_\gamma$, related
     through the map $\AAA_\gamma \sa{\phi_\tau \circ \epsilon_\gamma}
    \AAA$;, we apply the right-liberation
    property to calculate the value of $\E$ on this word.
\end{itemize}

As in the second chapter, we now formalize this strategy
in terms of recursively defined functions.

\begin{notation}
\
\begin{itemize}
    \item For $n \geq 1$ let $[0,\infty)^n_0$
    denote the subset
    \[
    \{(t_1, \dots, t_n)
    \in [0,\infty)^n \mid \min (t_1, \dots, t_n) = 0 \}.
    \]

    \item For $n \geq 1$ let
    $\psi_n: [0,\infty)^n \to [0,\infty)
    \times [0,\infty)^n_0$ denote the homeomorphism
    \[
    \psi_n(t_1, \dots, t_n) = (\min t_i,
    t_1 - \min t_i, \dots, t_n - \min t_i)
    \]
    with inverse
    \[
    \psi_n\inv(\tau, s_1, \dots, s_n)
    = (s_1+\tau, \dots, s_n+\tau).
    \]

    \item For $n \geq 1$, $\vec{s} \in [0,\infty)^n_0$,
    and $\vec{a} \in \AAA^n$, the \textbf{standard decomposition} of $(\vec{s}, \vec{a})$ is as
    follows: Write $\vec{s} = \vec{\nN}_0 \vee \vec{\sS}_0
    \vee \vec{\nN}_1 \vee \dots \vee \vec{\sS}_{m-1}
    \vee \vec{\nN}_m$, where
    each entry in each $\vec{\nN}_k$ is zero, each entry
    in each $\vec{\sS}_k$ is nonzero, and
    some of $\vec{\nN}_0$, $\vec{\sS}_0$,
     and $\vec{\nN}_m$ may be empty; here $\vee$
     denotes concatenation.
    Write $\vec{a} = \vec{z}_0 \vee \vec{w}_0
    \vee \vec{z}_1 \vee \dots \vee \vec{w}_{m-1}
    \vee \vec{z}_m$, where each $\vec{z}_i$
    has the same length as $\vec{\nN}_i$
    and each $\vec{w}_i$ the same length as $\vec{\sS}_i$.
    We refer to $(\vec{\nN}_0,
    \vec{\sS}_0, \dots, \vec{\nN}_m)$
    as the standard decomposition of $(\vec{s})$,
    and $(\vec{z}_0, \vec{w}_0, \dots, \vec{z}_m)$
    as the standard decomposition of $\vec{a}$ with
    respect to $\vec{s}$.  The \textbf{alternation number}
    of $\vec{s}$, denoted $\text{alt}(\vec{s})$,
    is the number $m$ appearing in
    the standard decomposition; the alternation number of
    an element $\vec{t} \in [0,\infty)^n$ is the
    alternation number of $\psi_n(\vec{t} \,)$.

    \item By $([0,\infty)_0 \times \AAA)^\sharp_L$
    we denote the set of tuples $(\vec{s}, \vec{a},
    \vec{\iota})$ such that, for some $n \geq 1$,
    $\vec{s} \in [0,\infty)^n_0$, $\vec{a} \in \AAA^n$,
    and $\vec{\iota} \in [\text{alt}(\vec{s})-1]$, with
    the convention $[-1] = \emptyset$.  By
    $([0,\infty)_0 \times \AAA)^\sharp_R$ we denote the
    same set except with $[\text{alt}(\vec{s})]$ in
    place of $[\text{alt}(\vec{s})-1]$.
\end{itemize}
\end{notation}

We next introduce ``diachronic'' versions of the
collapse and moment functions from chapter 2.

\begin{definition} \label{defmomentpolynomials}
We recursively define functions
\begin{align*}
\Ss: &\big([0,\infty) \times \AAA \big)^\sharp \to \AAA\\
\DRM: &\big( [0,\infty) \times \AAA \big)^\sharp \to \AAA\\
\DUM: &\big( [0,\infty) \times \AAA \big)^\sharp \to \AAA\\
\DRC: &\big( [0,\infty) \times \AAA \big)^\sharp_R \to
\big([0,\infty) \times \AAA \big)^\sharp\\
\DLC: &\big( [0,\infty) \times\AAA \big)^\sharp_L
\to \big( [0,\infty) \times \AAA \big)^\sharp\\
\DUC: &\big( [0,\infty) \times \AAA \big)^\sharp_L \to
\big( [0,\infty) \times \AAA \big)^\sharp
\end{align*}
as follows:
\begin{enumerate}
    \item $\Ss(t; a) = \DRM(t; a) = \DUM(t; a) = \phi_t(a)$.
    \item For $n \geq 2$,
    $\vec{t} \in [0,\infty)^n$, and $\vec{a} \in \AAA^n$,
    let $m = \text{alt}(\vec{t} \,)$; then
    \begin{align*}
    \Ss(\vec{t}; \vec{a}) &= \sum_{\vec{\iota} \subset [m-1]} \DRM \big( \DLC (\vec{t}; \vec{a}; \vec{\iota})\big)
    \prod_{j \in [m-1] \setminus \vec{\iota}}
    \omega \left( \Pi(\vec{z}_{2j+1}) \right)\\
    \DRM(\vec{t}; \vec{a}) &= \sum_{\vec{\iota}
    \subsetneq [m]}
    \DUM(\DRC(\vec{t}; \vec{a}; \vec{\iota}))\\
    \DUM(\vec{t}; \vec{a}) &=
    \sum_{\vec{\iota}
    \subset [m]}
    \Ss(\DUC(\vec{t}; \vec{a}; \vec{\iota}))
    \end{align*}

    \item For $\vec{t}, \vec{a}$ as above, let
    $(\tau, \vec{s}) = \psi_n(\vec{t} \,)$,
    $(\vec{\nN}_0, \dots, \vec{\nN}_m)$ be the
    standard decomposition of $\vec{s}$,
    and $(\vec{z}_0, \dots, \vec{z}_m)$ the corresponding standard decomposition of $\vec{a}$.  Given also
    $\vec{\iota} \subset [m-1]$, let
    $\vec{\iota} = (i_1, \dots, i_\ell)$. For
    each $k = 1, \dots, \ell+1$ define $\vec{\alpha}_k
    = \vec{w}_{i_{k-1}+1} \vee
    \dots \vee \vec{w}_{i_k}$, with the conventions
    $i_0 = 0$ and $i_{\ell+1} = m$, and corresponding
    time vectors $\vec{u}_k =
    \vec{\sS}_{i_{k-1}+1} \vee \dots \vee \vec{\sS}_{i_k}$.  Let
    \[
    \vec{\uU} = \vec{\nN}_0
    \vee \vec{u}_1 \vee \{0\} \vee \vec{u}_2
    \vee \{0\} \vee \dots \vee \{0\} \vee \vec{u}_{\ell+1}
    \vee \vec{\nN}_{m+1}
    \]
    and
    \[
    \vec{b} = \vec{z}_0 \vee \vec{\alpha}_1
    \vee \left\{ \Pi(\vec{z}_{i_1})
    - \omega \left(\Pi(\vec{z}_{i_1})\right) \right\}
    \vee \dots \vee
    \left\{ \Pi(\vec{z}_{i_\ell})
    - \omega \left(\Pi(\vec{z}_{i_\ell})\right) \right\} \vee \vec{\alpha}_{\ell+1} \vee
    \vec{z}_{m+1}.
    \]
    Then $\DLC(\vec{t}; \vec{a}) = (\vec{\uU}; \vec{b})$.

    \item For $\vec{t}, \vec{a}$ as above
    and $\vec{\iota}= (i_1, \dots, i_\ell) \subset [m]$,
    with $i_0 = 0$ and $i_{\ell+1} = m+1$,
    define for each $k = 0, \dots, \ell$ the
    elements
    \begin{align*}
    \beta_k &= \Pi \left(\vec{z}_{i_k} \right)
    \prod_{j=i_k+1}^{i_{k+1}-1} \Ss(\vec{\sS}_j;
    \vec{w}_j) \Pi \left(\vec{z}_j \right)\\
    \gamma_k &= \Pi \left(\vec{z}_{i_k} \right)
    \prod_{j=i_k+1}^{i_{k+1}-1} -\Ss(\vec{\sS}_j;
    \vec{w}_j) \Pi \left(\vec{z}_j \right)
    \end{align*}
    Let
    \begin{align*}
    \vec{\uU} &= \{0\}
    \vee \vec{\sS}_{i_1} \vee \{0\} \vee
    \vec{\sS}_{i_2} \vee \dots \vee \vec{\sS}_{i_m}
    \vee \{0\}\\
    \vec{b} &= \{\beta_0\} \vee \vec{w}_{i_1}
    \vee \{\beta_1\} \vee \dots \vee \vec{w}_{i_\ell}
    \vee \{\beta_\ell\}\\
    \vec{c} &= \{\gamma_0\} \vee \vec{w}_{i_1}
    \vee \{\gamma_1\} \vee \dots \vee \vec{w}_{i_\ell}
    \vee \{\gamma_\ell\}
    \end{align*}
    Then
    \begin{align*}
    \DRC(\vec{t}; \vec{a}; \vec{\iota})
    &= (\vec{\uU}; \vec{b})\\
    \DUC(\vec{t}; \vec{a}; \vec{\iota}) &=
    (\vec{\uU}; \vec{c}).
    \end{align*}
\end{enumerate}
\end{definition}
We note that the $\DLC$, $\DRC$, and $\DUC$ functions
output vector pairs at least as short as the input
vector pairs; this together with the strict subset
inclusion in the definition of $\DRM$ yield a well-defined recursion from the above formulas.

The reason for defining these functions is the
following proposition:

\begin{proposition} \label{propmomentpolyasLM}
Let $\AAA$ be a unital C$^*$-algebra, $\{\phi_t\}$
a CP$_0$-semigroup on $\AAA$, $\omega$ a faithful
state on $\AAA$, and $(\Aa, i, \E, \{\sigma_t\})$
the Sauvageot dilation.  Then
for every $t_1, \dots, t_n \geq 0$
and $a_1, \dots, a_n \in \AAA$,
\[
\E \Big[ \sigma_{t_1}(i(a_1))
\cdots \sigma_{t_n}(i(a_n)) \Big]
= \Ss(\vec{t}; \vec{a}).
\]
\end{proposition}

\begin{proof}
Recall from Definition \ref{defdilationretraction}
that $\E \circ \dss{f}{\infty,\gamma}= \epsilon_\gamma(a)$
for any $\gamma \in \FF$.  Here, letting
$\gamma$ be the union of $\vec{t}$,
$\gamma(2)$ the immediate tail of $\gamma$ with
distance $\tau$,
and defining the elements $\aA_i \in \AAA_{\gamma(2)}$
corresponding to times $\vec{\sS}_i$ and tuples
$\vec{w}_i$,
the above expectation is, by Definition \ref{definductiveobjects},
Corollary \ref{corthetaliberates}, and
Theorem \ref{thmrightlibmoments}, equal to
\[
\LM(\Pi(\vec{z}_0),
\aA_1, \Pi(\vec{z}_1), \dots, \aA_m, \Pi(\vec{z}_{m+1});
\phi_\tau \circ \epsilon_{\gamma(2)}).
\]
Further consideration of what happens when
$\epsilon_{\gamma(2)}$ is applied to the elements $\aA_i$,
together with the recurrence that defines the $\LM$
function, shows that
$\E[\sigma_{t_1}(i(a_1)) \cdots \sigma_{t_n}(i(a_n))]$
and $\Ss(\vec{t}; \vec{a})$ satisfy the same
recurrence and initial conditions, so they are
equal.
\end{proof}

\section{Continuity Properties of Moment Polynomials}
The continuity properties of $\Ss(\vec{t}; \vec{a})$ in
the case where $\AAA$ is a W$^*$-algebra
will be important in what follows.
There are three types of continuity properties
to consider: continuity in $a_1, \dots, a_n$ with respect
to both the weak and the strong topologies, and
continuity in $t_1, \dots, t_n$.  It turns out
that weak continuity holds with respect
to $a_1, \dots, a_n$ separately (which is the best
we could hope for, as multiplication is not jointly
weakly continuous), whereas
strong continuity holds jointly in $a_1, \dots, a_n$,
and a restricted form of joint continuity
in $t_1, \dots, t_n$ holds as well.

\begin{proposition}
Let $\AAA$ be a W$^*$-algebra, $\{\phi_t\}$ a CP$_0$-semigroup on $\AAA$, $\omega$ a faithful normal state on $\AAA$.
Fix $n \geq 1$, $t_1, \dots, t_n \geq 0$, $j \in \{1, \dots,n\}$, and $a_k$ for $k \in \{1, \dots, n\} \setminus \{j\}$.  Then $\Ss(\vec{t}; \vec{a})$, viewed as a
function of $a_j$, is a normal linear map from $\AAA$
to itself.
\end{proposition}

\begin{proof}
This is a straightforward induction from Definition
(\ref{defmomentpolynomials}); we show simultaneously
that the six functions $\Ss, \DRM, \DUM, \DLC, \DRC, \DUC$
are normal functions of $a_j$ when the other $a_i$
and all the $t_i$ are fixed.  This follows from
the normality of the state $\omega$ and the maps
$\phi_t$, as well as the normality of multiplication by
a fixed element of $\AAA$.
\end{proof}

\begin{definition} \label{defnoncrossingly}
For $n \geq 1$ and elements $\{\vec{s}_k\}$ and
$\vec{t}$ of $[0,\infty)^n$, we say that $\vec{s}_k$
\textbf{converges non-crossingly to } $\vec{t}$ if
$\vec{s}_k \to \vec{t}$ and, for all $k$, the
order relations among the entries of $\vec{s}_k$ are the
same as those in $\vec{t}$; that is, if
\[
\forall k: \ \forall i,j = 1, \dots, n: \
(s_k)_i \leq (s_k)_j \LRA t_i \leq t_j.
\]
\end{definition}

\begin{proposition} \label{propjointcontinuitymoments}
Let $\AAA$ be a separable W$^*$-algebra, $\{\phi_t\}$
a CP$_0$-semigroup on $\AAA$, $\omega$ a faithful
normal state on $\AAA$.  Let $n \geq 1$.  Let
$\vec{t}_k \to \vec{t}$ converge non-crossingly,
and let $\vec{a}_k \to \vec{a}$ be a strongly convergent
sequence of tuples in $(\AAA_1)^n$.  Then
$\Ss(\vec{t}_k; \vec{a}_k) \to \Ss(\vec{t}; \vec{a})$
strongly.  That is, $\Ss(\vec{t}; \vec{a})$ is jointly
strongly continuous in $\vec{t}$ and $\vec{a}$, subject
to the non-crossing restriction on $\vec{t}$.
\end{proposition}

\begin{proof}
We induct on $n$, using definition (\ref{defmomentpolynomials})
and the following observations:
\begin{itemize}
    \item If $\vec{t}_k \to \vec{t}$ non-crossingly, then $\vec{s}_k \to \vec{s}$ non-crossingly, where $\vec{s}_k$, $\vec{s}$ are the sub-tuples of $\vec{t}_k, \vec{t}$ corresponding to a fixed subset of indices
        from $\{1, \dots, n\}$.
    \item If $\vec{t}_k \to \vec{t}$ non-crossingly
    and $\vec{\iota}$ is given, let $\vec{\uU}$ be
    defined from $\vec{t}$ and $\vec{\uU}_k$
    from $\vec{t}_k$ as in the $\DRC$ and $\DUC$ functions;
    then $\vec{\uU}_k \to \vec{\uU}$ non-crossingly.

    \item If $\vec{t}_k \to \vec{t}$ non-crossingly,
    then $\psi_n(\vec{t}_k)$ and $\psi_n(\vec{t} \,)$
    are zero at the same entries; moreover, the
    corresponding parts $\vec{\nN}_i, \vec{\sS}_i$ of the standard
    decompositions of $\vec{t}_k$ and of $\vec{t}$
    are all the same length.
\end{itemize}
These considerations plus the strong continuity
of the state $\omega$ and the joint strong continuity
of the semigroup $\phi$ (Theorem \ref{thmjointcontinuityCP})
imply the result.
\end{proof}

\section{The Continuous Theorem}
We now return to the question of how to obtain a
continuous W$^*$-dilation from an algebraic C$^*$-dilation.
The technique in this section is adapted
from the eighth chapter of \cite{ArvesonDynamics}.
Throughout, we let $\AAA$ denote a separable W$^*$-algebra,
$\{\phi_t\}$ a CP$_0$-semigroup
on $\AAA$, $(\Aa, i, \E, \{\sigma_t\})$ the Sauvageot
dilation from the previous chapter,
$\PP \subset \Aa$ the subset
\[
\PP = \{\sigma_{t_1}(i(a_1)) \dots \sigma_{t_k}(i(a_k))
\mid t_1, \dots, t_k \geq 0; \, a_1, \dots, a_k \in \AAA\},
\]
$\Aa_0 \subseteq \Aa$ the norm-dense linear span of $\PP$,
$(H, \pi)$ a faithful
normal representation of $\AAA$ on a separable Hilbert
space, $(\Hh, V, \psi)$ a minimal Stinespring
dilation of $\pi \circ \E$, $\widetilde{\Aa} = \psi(\Aa)''$,
and $\widetilde{\E}: \widetilde{\Aa} \to \AAA$ the
map $\widetilde{\E}[T] = \pi\inv(V^*T V)$, which
is well-defined because $T \mapsto V^* T V$ is normal and maps
the weakly dense subspace $\psi(\Aa) \subset \widetilde{\Aa}$ into the weakly closed set $\pi(\AAA)$, and
because $\pi$ is faithful; it satisfies
$\widetilde{\E} \circ \psi = \E$ and therefore is
a normal retraction with respect to $\psi \circ i$.

We begin with the observation that weak-operator continuity
of families of contractions can be checked
on a dense subset of Hilbert space.

\begin{lemma} \label{lemWOTcontinuitydense}
Let $\HH$ be a Hilbert space and $\{T_t\}_{t \geq 0}$
a family (not necessarily a semigroup) of contractions on $\HH$.  Let $\HH_0 \subseteq \HH$ be a dense linear
subspace such that for all $\xi, \eta \in \HH_0$,
the map $t \mapsto \la T_t \xi, \eta \ra$ is
continuous.  Then $t \mapsto T_t$ is WOT-continuous.
\end{lemma}

\begin{proof}
Let $\xi, \eta \in \HH$ and $t_0 \geq 0$.  Given
$\epsilon > 0$, choose $\xi_0, \eta_0 \in \HH_0$
with $\|\xi - \xi_0\| < \max(1,\epsilon)$ and
$\|\eta - \eta_0\| < \epsilon$.  Then
for any $t\geq 0$,
\begin{align*}
\la (T_t - T_{t_0}) \xi, \eta \ra &= \la (T_t - T_{t_0})
\xi_0, \eta_0 \ra \\
&+ \la (T_t - T_{t_0}) \xi_0, \eta - \eta_0 \ra\\
&+ \la (T_t - T_{t_0}) (\xi - \xi_0), \eta \ra.
\end{align*}
The first term tends to zero as $t \to t_0$ by hypothesis, so
that in particular it is less than $\epsilon$
for $t$ sufficiently close to $t_0$.  The second term
is at most $2 \|\xi_0\| \epsilon \leq
2 (\|\xi\|+1) \epsilon$ by Cauchy-Schwarz, and
the third term at most $2 \|\eta\| \epsilon$.
Hence
\[
|\la (T_t - T_{t_0}) \xi, \eta \ra|
\leq  \Big(3 + 2\|\eta\| + 2 \|\xi\| \Big) \epsilon
\]
for $t$ sufficiently near $t_0$.
\end{proof}

The next lemma is rather technical, but it
advances our study of how $\E$ interacts with time
translations, and in particular with translation
of the middle term of a threefold product.

\begin{lemma} \label{lemexistsQ}
Let $y,z \in \PP$ and $t \geq 0$.
There exist $y_0,z_0 \in \PP$ and a normal
linear map $Q: \AAA \to \AAA$ such
that, for every $x \in \Aa$,
\begin{equation} \label{eqnQmap}
\E \big[ y \sigma_t(x) z \big] = Q \Big(
\E \big[ y_0 x z_0 \big] \Big).
\end{equation}
\end{lemma}

\begin{proof}
Let $y = \sigma_{s_1}(a_1) \dots \sigma_{s_m}(a_m)$
and $z = \sigma_{t_1}(b_1) \dots \sigma_{t_n}(b_n)$.  We
proceed by strong induction on $m+n$.  In the base case $m+n=0$
(meaning that $y = z = \one$) the requisite map is
$Q = \phi_t$, by Theorem \ref{thmstrongdilation}.
Inductively, letting
$\tau = \min(s_1, \dots, s_m, t_1, \dots, t_n)$, the result
is again trivial in case $t \leq \tau$, as then
one can use $Q = \phi_t$, $y_0 = \sigma_{s_1-t}(a_1)
\cdots \sigma_{s_m-t}(a_m)$, and $z_0 = \sigma_{t_1-t}(b_1) \cdots \sigma_{t_n-t}(b_n)$.  Hence we assume $t > \tau$.
We further assume $\tau = 0$, as the case $\tau > 0$ reduces to
this by Theorem \ref{thmstrongdilation} again.

Let $(s_1', \dots, s_q')$ be the (possibly empty) final segment of nonzero entries from $(s_1, \dots, s_m)$,
and $(a_1', \dots, a_q')$ the corresponding entries
from $(a_1, \dots ,a_m)$.  Similarly, let $(t_1', \dots,
t_p')$ be the initial segment of nonzero entries
from $(t_1, \dots, t_n)$, and $(b_1', \dots, b_p')$
the corresponding entries from $(b_1, \dots, b_n)$.
Let $y_0 = \sigma_{s_1'}(a_1') \dots \sigma_{s_q'}(a_q')$
and $z_0 = \sigma_{t_1'}(b_1') \dots \sigma_{t_p'}(b_p')$.

For any $x \in \PP$, write $x = \sigma_{u_1}(c_1) \cdots \sigma_{u_\ell}(c_\ell)$, so that $\sigma_t(x)
= \sigma_{u_1+t}(c_1) \cdots \sigma_{u_\ell+t}(c_\ell)$.  Now $\E[y \sigma_t(x) z]
= \Ss(\vec{s} \vee (\vec{u}+t) \vee \vec{t};
\vec{a} \vee \vec{c} \vee \vec{b})$ by Proposition \ref{propmomentpolyasLM}.
In the standard decomposition $\vec{s} \vee (\vec{u}+t)
\vee \vec{t}
= \vec{\nN}_1 \vee \vec{\sS}_1
\vee \dots \vee \vec{\nN}_{m+1}$, we must have $\vec{u}+t$
contained in a single one of the $\vec{\sS}_i$;
more specifically, for some $i$ we have
$\vec{\sS}_i = (s_1', \dots, s_q')
\vee (\vec{u}+t) \vee (t_1', \dots, t_p')$ and
$\vec{w}_i = (a_1', \dots, a_q') \vee
\vec{c} \vee (b_1', \dots, b_p')$.
Then Propositions (\ref{propmomentpolyasLM}) and (\ref{propnormalmoments}) imply that $\E[y \sigma_t(x) z]$
is the composition of $\E[y_0 x z_0]$ with some normal map $Q$, which is independent of $x$.  This gives us equation (\ref{eqnQmap})
for all $x \in \PP$, and since both sides are linear and
norm-continuous in $x$, it follows that (\ref{eqnQmap})
holds for all $x \in \Aa$.
\end{proof}

\begin{theorem} \label{thmexistsnormalextension}
There exists a (necessarily unique) semigroup of normal unital *-endomorphisms $\{\widetilde{\sigma}_t\}_{t \geq 0}$ of $\widetilde{\Aa}$
such that
\begin{equation} \label{eqncovarianceofextension}
\forall t \geq 0: \qquad \widetilde{\sigma}_t \circ \psi = \psi \circ \sigma_t.
\end{equation}
\end{theorem}

\begin{proof}
We construct $\{\widetilde{\sigma}_t\}$ and verify
its properties in the following sequence of steps.
\begin{enumerate}
    \item For each $t \geq 0$ and $\xi, \eta \in \psi(\PP) V H$,
    we construct a normal linear functional $\rho_{t, \xi, \eta}$
    on $\widetilde{\Aa}$ as  follows.
    Let $\xi = \psi(y) V \xi'$ and $\eta = \psi(z) V \eta'$
    for $y,z \in \PP$ and $\xi', \eta' \in H$.
    By Lemma (\ref{lemexistsQ}), there exists a normal
    linear map $Q: \AAA \to \AAA$ and elements $y_0, z_0 \in \PP$
    such that $\E[z^* \sigma_t(x) y]
    = Q(\E[z_0^* x y_0])$ for all $x \in \Aa$.  We thus have
\[
\forall x \in \Aa: \qquad \la \psi(\sigma_t(x)) \xi, \eta \ra =
\la \pi \circ Q \circ \E[z_0^* x y_0] \xi', \eta' \ra.
\]
We now define $\rho_{t, \xi, \eta}$ by
\[
\rho_{t, \xi, \eta}(T) = \la \pi \circ Q \circ \widetilde{\E}
[\psi(z_0)^* T \psi(y_0)] \xi', \eta'\ra, \qquad T \in \widetilde{\Aa}.
\]
Then that the restriction to $\psi(\Aa)$ satisfies
\begin{align}
\forall x \in \Aa: \qquad \rho_{t, \xi, \eta}(\psi(x))
&= \la \pi \circ Q \circ \widetilde{\E} \circ
\psi(z_0^* x y_0) \xi', \eta' \ra \nonumber\\
&= \la \pi \circ Q \circ \E[z_0^* x y_0] \xi', \eta' \ra \nonumber\\
&= \la \pi \circ \E[z^* \sigma_t(x) y] \xi', \eta' \ra \nonumber\\
&= \la V^* \psi(z^* \sigma_t(x) y) V \xi', \eta' \ra \nonumber\\
&= \la \psi(\sigma_t(x)) \xi, \eta \ra. \label{eqnrhofunctional}
\end{align}

    \item We extend the definition to $\xi, \eta$
    in the linear span of $\psi(\PP) V H$ in the natural way; for
$\xi = \sum_i c_i \xi_i$ and $\eta = \sum_j d_j \eta_j$ with
$\xi_i, \eta_j \in \psi(\PP) VH$, we define
$\rho_{t, \xi, \eta} = \sum_{i,j} c_i \overline{d}_j \rho_{t, \xi_i, \eta_j}$.  This is well-defined because, if
$\sum_i c_i \xi_i = \sum_k \tilde{c}_k \tilde{\xi}_k$
and $\sum_j d_j \eta_j = \sum_\ell \tilde{d}_\ell \tilde{\eta}_\ell$
then equation (\ref{eqnrhofunctional}) implies that, for $x$ in
the ultraweakly dense subspace $\psi(\Aa)$ of $\widetilde{\Aa}$,
\begin{align*}
\dss{\rho}{t, \sum c_i \xi_i, \sum d_j \eta_j}
(\psi(x)) &= \left\la \psi(\sigma_t(x))\sum c_i \xi_i,
\sum d_j \eta_j \right\ra\\
&= \left\la \psi(\sigma_t(x)) \sum \tilde{c}_k \tilde{\xi}_k,
\sum \tilde{d}_\ell \tilde{\eta}_\ell \right\ra
= \dss{\rho}{t, \sum \tilde{c}_k \tilde{\xi}_k,
\sum \tilde{d}_\ell \tilde{\eta}_\ell}(\psi(x)).
\end{align*}

    \item Next, we note that equation (\ref{eqnrhofunctional}) also
    implies that $\|\rho_{t, \xi, \eta}\| \leq \|\xi\| \|\eta\|$.
    This allows us to extend the definition to $\xi, \eta$ in
    the norm closure of the linear span of $\psi(\PP) VH$,
    which is all of $\Hh$.

    \item Having defined the family of functionals $\{\rho_{t, \xi, \eta}\}$, we now use them to define the family
        of endomorphisms $\{\widetilde{\sigma}_t\}$.
        Equation (\ref{eqnrhofunctional})
      implies that, for fixed $t \geq 0$ and $x \in \Aa$, $\rho_{\xi,\eta}(\psi(x))$
        is a bounded sesquilinear function of $\xi$ and $\eta$, so that it
corresponds to a unique operator in $B(\Hh)$, which we call
$S_t(\psi(x))$, characterized by the property
\begin{equation} \label{eqndefSt}
\forall \xi, \eta \in \Hh: \quad
\rho_{t,\xi, \eta}(\psi(x)) = \la S_t(\psi(x))
\xi, \eta \ra.
\end{equation}

    \item Equations (\ref{eqnrhofunctional}) and (\ref{eqndefSt})
together imply that
\begin{equation} \label{eqnStcovariance}
\forall x \in \Aa: \quad S_t(\psi(x)) = \psi(\sigma_t(x)).
\end{equation}

    \item Because $\psi$ and $\sigma_t$
    are unital *-homomorphisms, equation (\ref{eqnStcovariance})
    implies that $S_t$ is as well.

    \item Because $S_t$ is a unital *-homomorphism of a C$^*$-algebra, it is contractive.  This implies
    \begin{equation} \label{eqncontractivity}
    \forall x \in \Aa: \quad \|\psi(\sigma_t(x))\|
    \leq \|\psi(x)\|.
    \end{equation}

    \item Given any $z \in \widetilde{\Aa}$, we can now
    show that $\rho_{t,\xi, \eta}(z)$ is a bounded sesquilinear
    function of $\xi$ and $\eta$.  For boundedness, we
    will show more precisely that
    \begin{equation} \label{eqnboundedforfixedz}
    |\rho_{t,\xi, \eta}(z)| \leq \|z\| \|\rho\|
    \|\eta\|.
    \end{equation}
    Let $z, \xi, \eta$ be given, and choose
    $\epsilon > 0$.  By the Kaplansky density theorem and
    the normality of $\rho_{t,\xi, \eta}$, there exists
    $x \in \Aa$ such that $\|\psi(x)\| \leq \|z\|$ and
    $|\rho_{t,\xi, \eta}(z - \psi(x))|  < \epsilon$.  Then
    \begin{align*}
    |\rho_{t,\xi, \eta}(z)| &\leq
    |\rho_{t,\xi, \eta}(\psi(x))| + |\rho_{\xi, \eta}(z - \psi(x))|\\
    &\leq \epsilon + |\la \psi(\sigma_t(x)) \xi, \eta \ra|\\
    &\leq \epsilon + \|\psi(\sigma_t(x))\| \|\xi\| \|\eta\|\\
    &\leq \epsilon + \|\psi(x)\| \|\xi\| \|\eta\| \\
    &\leq \epsilon + \|z\| \|\xi\| \|\eta\|.
    \end{align*}
    Letting $\epsilon \to 0$,
    we have \ref{eqnboundedforfixedz}.

    To show linearity in $\xi$, let $c_1, c_2 \in \com$
    and $\xi_1, \xi_2, \eta \in \Hh$ be given, and choose
    $\epsilon > 0$.  By Kaplansky density and the normality
    of $\rho_{t, \xi_1, \eta}$, $\rho_{t, \xi_2, \eta}$,
    and $\rho_{t, c_1 \xi_1 + c_2 \xi_2, \eta}$, there
    exists $x \in \Aa$ such that $\|\psi(x)\| \leq \|z\|$
    and the three inequalities
    \begin{align*}
    \left| \rho_{t,c_1 \xi_1 + c_2 \xi_2,\eta}(z - \psi(x)) \right|
    &< \epsilon \\
    |c_1| \left|\rho_{t,\xi_1, \eta}(z - \psi(x)) \right| &<
    \epsilon \\
    |c_2| \left| \rho_{t, \xi_2, \eta}(z - \psi(x))\right|
    &< \epsilon
    \end{align*}
    all hold.  Then
    \begin{align*}
    &\left|\rho_{t, c_1 \xi_1 + c_2 \xi_2, \eta}(z)
    - c_1 \rho_{t, \xi_1, \eta}(z) - c_2 \rho_{t, \xi_2,
    \eta}(z) \right| \\
    \leq &\left|\rho_{t, c_1 \xi_1 + c_2 \xi_2, \eta}(z - \psi(x))\right|
    + |c_1| \left|\rho_{t, \xi_1, \eta}(z - \psi(x))\right|
    + |c_2| \left|\rho_{t, \xi_2, \eta}(z - \psi(x))\right|\\
    &+ \left|\rho_{t, c_1 \xi_1 + c_2 \xi_2, \eta}(\psi(x))
    - c_1 \rho_{t, \xi_1, \eta}(\psi(x)) - c_2 \rho_{t, \xi_2,
    \eta}(\psi(x)) \right| \\
    \leq &3\epsilon +
    \left| \la \psi(\sigma_t(x)) (c_1 \xi_1 + c_2 \xi_2),
    \eta \ra - c_1 \la \psi(\sigma_t(x)) \xi_1, \eta \ra
    - c_2 \la \psi(\sigma_t(x)) \xi_2, \eta \ra \right|
    = &3\epsilon
    \end{align*}
    and as this is true for all $\epsilon > 0$,
    we conclude that
    \[
    \rho_{t, c_1 \xi_1 + c_2 \xi_2, \eta}(z)
    = c_1 \rho_{t, \xi_1, \eta}(z) +
    c_2 \rho_{t, \xi_2, \eta}(z).
    \]
    Conjugate-linearity in $\eta$ is, of course, established
    in the same way.

    \item We therefore obtain an operator in $B(\Hh)$,
    which we call $\widetilde{\sigma}_t(z)$, characterized by
    the property
    \begin{equation} \label{eqndefsigma}
    \forall \xi, \eta \in \Hh: \quad
    \rho_{t,\xi, \eta}(z) = \la \widetilde{\sigma}_t(z) \xi, \eta \ra.
    \end{equation}
    We now have a function (not yet known to be linear,
    continuous, multiplicative, or self-adjoint)
    $\widetilde{\sigma}_t: \widetilde{\Aa}
    \to B(\Hh)$ which extends the unital *-endomorphism
    $S_t: \psi(\Aa) \to \psi(\Aa)$.

    \item The function $\widetilde{\sigma}_t$ is contractive, because
    \[
    \|\widetilde{\sigma}_t(z)\| = \sup_{\xi, \eta \in \Hh_1}
    |\la \widetilde{\sigma}_t(z) \xi, \eta \ra |
    = \sup_{\xi, \eta \in \Hh_1} |\rho_{t, \xi, \eta}(z)|
    \leq \sup_{\xi, \eta \in \Hh_1} \|z\| \|\xi\| \|\eta\|
    = \|z\|
    \]
    by equation (\ref{eqnboundedforfixedz}).

    \item Weak and strong-* continuity of
    $\widetilde{\sigma}_t$ are
     straightforward consequence of the normality of the
    $\rho_{t, \xi, \eta}$.  Indeed, if $z_\nu \to z$
    weakly in the unit ball $\widetilde{\Aa}_1$, then for all
    $\xi, \eta$ it follows that
    \[
    \la \widetilde{\sigma}_t(z_\nu) \xi, \eta \ra
    = \rho_{t, \xi, \eta}(z_\nu) \to \rho_{t, \xi, \eta}(z)
    = \la \widetilde{\sigma}_t(z) \xi, \eta \ra
    \]
    so that $\widetilde{\sigma}_t(z_\nu) \to \widetilde{\sigma}_t(z)$
    in the weak operator topology, which agrees
    with the weak topology of $\widetilde{\Aa}$
    on bounded subsets.  If $z_\nu \to z$ strong-*,
    then $(z_\nu - z)^* (z_\nu - z) \to 0$ weakly
    and $(z_\nu - z)(z_\nu-z)^* \to 0$ weakly, and
    we repeat the analysis.

    \item Since $\widetilde{\sigma}_t$ maps the unit ball
    of $\psi(\Aa)$ into $\psi(\Aa)$, it follows
    from the previous step and the Kaplansky density
    theorem that it maps the unit ball of $\widetilde{\Aa}$
    into $\widetilde{\Aa}$.  Hence $\widetilde{\sigma}_t$,
    initially defined as a map from $\widetilde{\Aa}$
    into $B(\Hh)$, is actually a self-map of $\widetilde{\Aa}$.

    \item Next, we prove that $\widetilde{\sigma}_t$ is
    a *-endomorphism of $\widetilde{\Aa}$.  For multiplicativity,
    let $z,w \in \widetilde{\Aa}$.  Using Kaplansky density,
    choose nets $\{x_\nu\}, \{y_\nu\} \subset \Aa$ with
    $\|\psi(x_\nu)\| \leq \|z\|$ and
    $\|\psi(y_\nu)\| \leq \|w\|$ for all $\nu$, and
    $\psi(x_\nu) \to z$ and $\psi(y_\nu) \to w$ strongly.
    Then $\|\psi(x_\nu) \psi(y_\nu)\| \leq \|z\| \|w\|$
    for all $\nu$ and since multiplication is jointly strongly
    continuous, we have
    $\psi(x_\nu) \psi(y_\nu) \to z w$ strongly.  Then
    since $\widetilde{\sigma}_t$ is strongly continuous and is multiplicative on $\psi(\Aa)$,
    \begin{align*}
    \widetilde{\sigma}_t(z w) &= \widetilde{\sigma}_t(\lim_{\mu, \nu}
    \psi(x_\nu) \psi(y_\mu))\\
    &= \lim_{\mu, \nu} \widetilde{\sigma}_t(\psi(x_\nu) \psi(y_\mu))\\
    &= \lim_{\mu, \nu} \widetilde{\sigma}_t(\psi(x_\nu))
    \widetilde{\sigma}_t(\psi(y_\mu)) \\
    &= \widetilde{\sigma}_t(z) \widetilde{\sigma}_t(w).
    \end{align*}

    For linearity, let $c_1, c_2 \in \com$ and $z, w
    \in \widetilde{\Aa}$ be given.  Choose $\{x_\nu\}$
    and $\{y_\mu\}$ as before; then for all $\mu, \nu$ we
    have $\|c_1 x_\nu + c_2 y_\mu\| \leq |c_1| \|z\|
    + |c_2| \|w\|$, so that $\{c_1 x_\nu + c_2 y_\mu\}$
    is contained in a bounded subset of $\widetilde{\Aa}$.
    The calculation then proceeds as for multiplicativity.
    Self-adjointness is proved similarly.

    \item Finally, it is clear that
    $\widetilde{\sigma}_0 = id$, and for all $s,t \geq 0$
    and all $x \in \Aa$,
    \[
    \widetilde{\sigma}_{s+t}(\psi(x))
    = \psi(\sigma_{s+t}(x))
    = \psi(\sigma_s (\sigma_t(x)))
    = \widetilde{\sigma}_s(\psi(\sigma_t(x)))
    = \widetilde{\sigma}_s(\widetilde{\sigma}_t(\psi(x)))
    \]
    so that $\widetilde{\sigma}_{s+t}$ and $\widetilde{\sigma}_s
    \circ \widetilde{\sigma}_t$ agree on the ultraweakly
    dense subset $\psi(\Aa) \subset \widetilde{\Aa}$;
    as both are normal, they are equal.
\end{enumerate}
\end{proof}

As one corollary, we can now find many dense
subspaces of $\Hh$.  Recall that $\psi(\PP) VH$ is
dense by the standard properties of the minimal Stinespring
dilation plus the fact that $\PP$ is norm-dense
in $\Aa$.

\begin{lemma} \label{lemfamilydense}
For any finite set $F \subset [0,\infty)$ let
$\PP^{(F)}$ denote those elements of $\PP$ which
do not use any time indices from $F$.  Then for
all finite $F \subset [0,\infty)$, $\psi(P^{(F)} VH$
is dense in $\Hh$.
\end{lemma}

\begin{proof}
Consider a general vector of the form
$\sigma_{t_1}(i(a_1)) \cdots \sigma_{t_n}(i(a_n)) V h$,
which we already know to be total in $\Hh$.  We proceed
by induction on $n$.  In the case $n = 1$ we have for
any $\tilde{t}$ that
\[
\|(\sigma_t(i(a)) - \sigma_{\tilde{t}}(i(a)) ) Vh\|^2
= \Ss(t; a^*a) - \Ss(\tilde{t}, t; a^*, a)
- \Ss(t, \tilde{t}; a^*, a) + \Ss(\tilde{t}; a^* a).
\]
As $\tilde{t} \to t$, this approaches zero by
the continuity properties of $\Ss$.  Inductively,
we can approximate $\sigma_{t_2}(i(a_2))
\cdots \sigma_{t_n}(i(a_n)) Vh$ by a vector in
$\psi(\PP^{(F)}) VH$, which we then use as our
$h$ and proceed as before.
\end{proof}

Before establishing our main continuity result, one more preliminary is needed.

\begin{proposition} \label{propseparabilityofHh}
The Hilbert space $\Hh$ is separable.
\end{proposition}

\begin{proof}
Let $H_0$ be a countable dense subset of $H$,
and $\AAA_0$ a countable ultraweakly dense subset
of $\AAA$.  We may assume WLOG that $\AAA_0$ is
a self-adjoint $\Q$-subalgebra, so that its unit ball
is strongly dense in the unit ball of $\AAA$
by Kaplansky's theorem.

We will show that the countable set
\[
\Big\{\psi\big(\sigma_{t_1}(i(x_1)) \dots \sigma_{t_n}(i(x_n))\big)Vh \mid 0 \leq t_1, \dots, t_n \in \Q;
\, x_1, \dots, x_n \in \AAA_0; \, h \in H_0 \Big\}
\]
spans a dense subset of $\Hh$.  We already know that $\psi(\PP) VH$ has dense span, so it suffices to show that vectors in
$\psi(\PP) VH$ can be norm-approximated by vectors
of the prescribed form.  Let $\tau_1, \dots, \tau_n \geq
0$, $y_1, \dots, y_n \in \AAA$, and
$k \in H$.  By the triangle inequality, we have for any $h \in H_0$, any $t_1, \dots, t_n \in \Q_+$, and any $x_1, \dots, x_n \in \AAA_0$ that
\begin{align*}
&\Big\| \psi\big(\sigma_{\tau_1}(i(y_1)) \dots \sigma_{\tau_n}(i(y_n))
\big) Vk - \psi\big(\sigma_{t_1}(i(x_1)) \dots \sigma_{t_n}(i(x_n))\big) Vh \Big\|\\
&\leq \Big\|\psi\big(\sigma_{\tau_1}(i(y_1))\dots \sigma_{\tau_n}(i(y_n))\big)\Big\| \|h-k\|\\
&\qquad + \Big\| \psi \Big[ \sigma_{\tau_1}(i(y_1)) \cdots
\sigma_{\tau_n}(i(y_n)) - \sigma_{\tau_1}(i(x_1)) \cdots
\sigma_{\tau_n}(i(x_n)) \Big] V h \Big\|\\
&\qquad + \Big\| \psi \Big[ \sigma_{\tau_1}(i(x_1)) \cdots
\sigma_{\tau_n}(i(x_n)) - \sigma_{t_1}(i(x_1)) \cdots
\sigma_{t_n}(i(x_n)) \Big] V h \Big\|.
\end{align*}
The first term can be made small by choosing $h$
sufficiently close to $k$.  For the second,
 note that each composition $\psi \circ \sigma_t \circ i$ is normal, since it equals the composition
$\widetilde{\sigma}_t \circ \psi \circ i$; hence
$\psi(\sigma_t(i(\AAA_0)))$ is weakly
dense in $\psi(\sigma_t(i(\AAA)))$.  By Kaplansky's theorem,
it follows that the unit ball of $\psi(\sigma_t(i(\AAA_0)))$
is strongly dense in the unit ball of $\psi(\sigma_t(i(\AAA)))$; this plus the
joint strong continuity of multiplication implies that \\
${\Big\{\psi\big(\sigma_{t_1}(i(x_1))
\dots \sigma_{s_n}(i(x_n)) \big) \mid  s_1, \dots, s_n \geq 0; \ x_1, \dots, x_n \in \AAA_0\Big\}}$ is
strongly dense in \\ ${\Big\{\psi\big(\sigma_{t_1}(i(y_1))
\dots \sigma_{t_n}(i(y_n)) \big) \mid  t_1, \dots, t_n \geq 0; \ y_1, \dots, y_n \in \AAA\Big\}}$.  Hence, once $h$ has
  been fixed, an appropriate choice of $x_1, \dots, x_n$ makes the second term arbitrarily small.  So far we have shown that vectors of the form
\begin{equation} \label{eqnvectorform}
\psi \big(\sigma_{\tau_1}(i(x_1)) \cdots \sigma_{\tau_n}(i(x_n))
\big) Vh \qquad \tau_1, \dots, \tau_n \geq 0; \,
x_1, \dots, x_n \in \AAA_0; \, h \in H_0
\end{equation}
are total in $\Hh$.  It remains to prove
that such vectors remain total under the added restriction
that the $\tau_i$ be rational.  Let $\xi \in \psi(\PP) VH$
be orthogonal to all vectors of the form (\ref{eqnvectorform}).
That is, we let $z_1, \dots, z_m \in \AAA$,
$\eta \in H$, and $s_1, \dots, s_m \geq 0$ such that,
for all $x_1, \dots, x_n \in \AAA_0$, all
$0 \leq t_1, \dots, t_n \in \Q$, and all $h \in H_0$,
\begin{align*}
0 &= \left\la
\psi \big(\sigma_{t_1}(i(x_1)) \cdots \sigma_{t_n}(i(x_n))\big)
Vh, \psi \big(\sigma_{s_1}(i(z_1)) \cdots \sigma_{s_m}(i(z_m))
\big) V \eta \right\ra\\
&= \left\la V^* \psi \big( \sigma_{s_m}(i(z_m^*))
\dots \sigma_{s_1}(i(z_1^*)) \sigma_{t_1}(i(x_1)) \cdots \sigma_{t_n}(i(x_n)) \big) Vh, \eta \right\ra\\
&= \la \Ss(\vec{s}^{\, *} \vee \vec{t}; \vec{z}^{\, *} \vee
\vec{x}) \xi, \eta \ra
\end{align*}
where we introduce the
notation $(s_1, \dots, s_m)^* = (s_m, \dots, s_1)$
for $s_1, \dots, s_m \geq 0$
and $(z_1, \dots, z_m)^* = (z_m^*, \dots, z_1^*)$
for $z_1, \dots, z_m \in \AAA$.  Now for any $\vec{t} \in [0,\infty)^n$, let $\{\vec{t}_k \} \subset \Q_+^n$
such that $\vec{s}^* \vee \vec{t}_k \to \vec{s}^* \vee \vec{t}$ non-crossingly; then
by Proposition
(\ref{propjointcontinuitymoments}),
\[
\la \Ss(\vec{s}^* \vee \vec{t}; \vec{z}^* \vee
\vec{x}) \xi, \eta \ra
= \lim_{k \to \infty}
\la \Ss(\vec{s}^* \vee \vec{t}_k; \vec{z}^* \vee
\vec{x}) \xi, \eta \ra = 0.
\]

We thus see that $\xi$ must be orthogonal to a known total
set and hence zero.
\end{proof}

\begin{theorem}  \label{thmultraweakcontinuity}
For any $a \in \widetilde{\Aa}$,
$t \mapsto \widetilde{\sigma}_t(a)$  is ultraweakly continuous
for all $t > 0$.
\end{theorem}

\begin{proof}
We establish this in a series of steps.
\begin{enumerate}
    \item For any $a \in \AAA_0$ and $\xi, \eta \in \psi(\PP) VH$,
    the value of $\la \widetilde{\sigma}_t(\psi(a))
    \xi, \eta \ra = \la \psi(\sigma_t(a)) \xi, \eta \ra$ is
    given by a certain Sauvageot moment polynomial; explicitly, if\\ $a = \sigma_{\tau_1}(i(x_1))
    \cdots \sigma_{\tau_n}(i(x_n))$,
    $\xi = \sigma_{s_1}(i(y_1)) \cdots \sigma_{s_m}
    (i(y_n)) V \xi_0$, and\\
    $\eta = \sigma_{u_1}(i(z_1)) \cdots \sigma_{u_\ell}
    (i(z_\ell)) V \eta_0$, then
    \[
    \la \widetilde{\sigma}_t(\psi(a)) \xi, \eta \ra
    = \left\la \pi\Big(\Ss(\vec{u}^{\, *} \vee (\vec{\tau} +t)
    \vee \vec{s}; \vec{z}^{\, *} \vee
    \vec{x} \vee \vec{y}) \Big) \xi_0, \eta_0 \right \ra.
    \]

    \item Given $\vec{\tau}$ and a time $t_0\geq 0$, let
    $F$ be the set of times in $\vec{\tau} + t_0$.
    Taking any $\xi_0, \eta_0 \in \psi(P^{(F)}) V H$,
    which is dense by lemma (\ref{lemfamilydense}),
    we see by proposition (\ref{propjointcontinuitymoments})
    that the above expression is continuous at $t_0$,
    since if $t \to t_0$ within a sufficiently small
    neighborhood of $t_0$ then $\vec{u}^* \vee (\vec{\tau}
    + t) \vee \vec{s} \to \vec{u}^* \vee (\vec{\tau} + t_0)
    \vee \vec{s}$ non-crossingly.
    We therefore have that $t \mapsto \la \widetilde{\sigma}_t(\psi(a))
    \xi, \eta \ra$ is continuous at $t_0$
    for all $\xi, \eta \in
    \psi(\PP^{(F)}) VH$ and all $a \in \AAA_0$.

    \item By Lemma (\ref{lemWOTcontinuitydense}),
    this implies that $t \mapsto \la \widetilde{\sigma}_t(\psi(a))
    \xi, \eta \ra$ is continuous at $t_0$ for all
    $\xi, \eta \in \Hh$ and all $a \in \AAA_0$.

    \item Now let $a \in \widetilde{\Aa}$.  By Kaplansky
    density, there is a sequence $\{a_n\} \subset \Aa_0$
    such that $\psi(a_n) \to a$ in SOT.  We can use
    a sequence rather than a net because the separability
    of $\Hh$, established
     in Proposition (\ref{propseparabilityofHh}), implies the SOT-metrizability of $B(\Hh)$ (\cite{Blackadar} III.2.2.27).  Then for any $\xi, \eta \in \Hh$,
     \[
     \la \widetilde{\sigma}_t(a) \xi, \eta
     \ra = \lim_n \la \widetilde{\sigma}_t(\psi(a_n))
     \xi, \eta \ra
     \]
     so that the left-hand side, as a function
      of $t$, is a pointwise limit of
     a sequence of continuous functions, hence
     measurable.  That is,
     $t \mapsto \widetilde{\sigma}_t(a)$ is WOT-measurable;
     as the $\widetilde{\sigma}$ are contractions and the WOT
     agrees with the ultraweak topology on bounded subsets,
     $t \mapsto \widetilde{\sigma}_t(a)$ is ultraweakly
     measurable at all $t \geq 0$.

     \item Since each $\widetilde{\sigma}_t$ is normal,
     there is a corresponding preadjoint semigroup
     $\{\rho_t\}$ on $\widetilde{\Aa}_*$ given
     by $\rho_t f = f \circ \widetilde{\sigma}_t$,
     as discussed in section \ref{secC0semigroups},
     such that for each $f \in \widetilde{\Aa}_*$,
     $t \mapsto \rho_t(f)$ is weakly measurable
     at all $t \geq 0$.
     
     \item Since $\Hh$ is separable and
     $\widetilde{\Aa} \subset B(\Hh)$, it follows
     that $\widetilde{\Aa}_*$ is a separable Banach
     space.  By section \ref{secC0semigroups},
     the weak measurability
     of $\{\rho_t\}$ is therefore equivalent to its
     weak continuity at times $t > 0$.
     This is then equivalent to the
     ultraweak continuity of $t \mapsto \widetilde{\sigma}_t$.
 \end{enumerate}
\end{proof}

\begin{theorem} \label{thmextensionstrongdilation}
$(\widetilde{\Aa}, \psi \circ i, \widetilde{\E},
\{\widetilde{\sigma}_t\})$ is a strong dilation
of $(\AAA, \{\phi_t\})$.
\end{theorem}

\begin{proof}
By the definition of
$\widetilde{\E}$, equation \ref{eqncovarianceofextension},
 and theorem \ref{thmstrongdilation},
\begin{align*}
\widetilde{\E} \circ \widetilde{\sigma}_t \circ \psi
&=\widetilde{\E} \circ \psi \circ \sigma_t\\
&= \E \circ \sigma_t \\
&= \phi_t \circ \E\\
&= \phi_t \circ \widetilde{\E} \circ \psi.
\end{align*}
Since both $\phi_t \circ \widetilde{\E}$ and
$\widetilde{\E} \circ \widetilde{\sigma}_t$ are normal, and
since they are equal on the ultraweakly dense subset $\psi(\Aa)
\subset \widetilde{\Aa}$, they must be equal.
\end{proof}

So far, theorem \ref{thmultraweakcontinuity}
leaves open the question whether
$\{\widetilde{\sigma}_t\}$ is point-weakly continuous
at $t = 0$.  I do not know when, if ever, that
would fail to be the case; however, in case it does,
we can remedy the situation by taking a suitable quotient.

\begin{lemma} \label{lemcontinuityat0}
Let $A$ be a separable W$^*$-algebra and $\{\alpha_t\}$
an e$_0$-semigroup on $A$ which is point-weakly
continuous at all $t > 0$.  Then $\alpha_t$
is point-weakly continuous at 0 iff
\[
\bigcap_{t > 0} \ker \alpha_t = \{0\}.
\]
\end{lemma}

\begin{proof}
The point-weak continuity of $\alpha_t$ at $t=0$ is equivalent
to the weak (equivalently, strong) continuity at $t=0$ of
the adjoint semigroup $\{\rho_t\}$ on $A_*$ defined
by $(\rho_t f) = f \circ \alpha_t$.  As mentioned
in section \ref{secC0semigroups}, this is equivalent to
the condition
\[
\overline{ \bigcup_{t > 0} \rho_t A_*} = A_*
\]
since $A_*$ is assumed separable.  Now
the annihilator of the left-hand side is
\begin{align*}
\overline{ \bigcup_{t > 0} \rho_t A_*}^\perp
&= \{a \in A \mid \forall t > 0: \ \forall f \in A_*:\
(\rho_t f)(a) = 0\}\\
&= \{a \in A \mid \forall t > 0: \ \forall f \in A_*: \
f(\alpha_t(a)) = 0\}\\
&= \{a \in A \mid \forall t > 0: \alpha_t(a) = 0\}\\
&= \bigcap_{t > 0 } \ker \alpha_t
\end{align*}
because $A_*$ separates points on $A$.
\end{proof}

\begin{theorem} \label{thmkahuna}
Let $\AAA$ be a separable W$^*$-algebra and $\{\phi_t\}$
a CP$_0$-semigroup on $\AAA$.  Then there exists a
unital strong dilation of $\{\phi_t\}$ to an E$_0$-semigroup on a
separable W$^*$-algebra.
\end{theorem}

\begin{proof}
The dilation $(\widetilde{\Aa}, \psi \circ i,
\widetilde{\E}, \{\widetilde{\sigma}_t\})$ constructed
in this chapter satisfies all the requirements
except possibly point-ultraweak continuity at $t=0$.

We now let
\[
\RR = \bigcap_{t > 0} \ker \widetilde{\sigma}_t.
\]
This is an ultraweakly closed ideal in $\widetilde{\Aa}$;
we use $\widehat{\Aa}$ for the quotient
$\widetilde{\Aa}/\RR$, which is another
separable W$^*$-algebra.  Because
$\widetilde{\sigma}_t(\RR) \subset
\RR$ for each $t > 0$, we obtain for each
$t > 0$ a map $\widehat{\sigma}_t: \widehat{\Aa} \to\widehat{\Aa}$
characterized by the commutative diagram
\[ \xymatrix{
\widetilde{\Aa} \ar[d] \ar[r]^{\widetilde{\sigma}_t} &
\widetilde{\Aa} \ar[d] \\
\widehat{\Aa} \ar[r]_{\widehat{\sigma}_t} &
\widehat{\Aa}
}\]
Defining also $\widehat{\sigma}_0 = \dss{\text{id}}{\widehat{\Aa}}$, we see that $\{\widehat{\sigma}_t\}$ inherits from
$\{\widetilde{\sigma}_t\}$ the properties of being an
e$_0$-semigroup and of point-ultraweak continuity
at $t > 0$.  Furthermore, \\
$\displaystyle\bigcap_{t > 0}
\ker \widehat{\sigma}_t = \{0\}$ by construction,
so that $\{\widehat{\sigma}_t\}$ is point-weakly continuous at $t=0$ and hence is an E$_0$-semigroup.  Our embedding
of $\AAA$ into $\widehat{\Aa}$ is given by
$q \circ \psi \circ i$, where $\widetilde{\Aa} \sa{q}
\widehat{\Aa}$ is the quotient map; this is injective because,
if $a \in \AAA$ is such that $q(\psi(i(a)))=0$, then
$\psi(i(a)) \in \RR$, so that for all $t > 0$ one has
\begin{align*}
\widetilde{\sigma}_t(\psi(i(a))) &= 0\\
\psi(\sigma_t(i(a))) &= 0\\
\sigma_t(i(a)) &= 0\\
\E[ \sigma_t(i(a))] &= 0\\
\phi_t(\E[i(a)]) &= 0\\
\phi_t(a) &= 0
\end{align*}
and since $\phi_t(a) \to a$ as $t \to 0^+$ this implies
$a = 0$.  To construct our retraction, we first note that
$\RR \subset \ker \widetilde{\E}$; indeed, if
$a \in \RR$ then for all $t > 0$ we have
\begin{align*}
\widetilde{\sigma}_t(a) &= 0\\
\widetilde{\E} \circ \widetilde{\sigma}_t(a) &= 0\\
\phi_t \circ \widetilde{\E}(a) &= 0
\end{align*}
and by letting $t \to 0^+$ we conclude $\widetilde{\E}(a) = 0$.
Hence, $\ker q \subset \ker \widetilde{\E}$, so there is
a unique map $\widehat{\E}: \widehat{\Aa} \to \AAA$
with $\widetilde{\E} = \widehat{\E} \circ q$.
This map satisfies $\widehat{\E} \circ q \circ \psi \circ i
= \widetilde{\E} \circ \psi \circ i = \dss{\text{id}}{\AAA}$,
so it is a retraction with respect to the given embedding.
Finally,
\[
\widehat{\E} \circ \widehat{\sigma}_t \circ
q = \widehat{\E} \circ q \circ \widetilde{\sigma}_t
= \widetilde{\E} \circ \widetilde{\sigma}_t
= \phi_t \circ \widetilde{\E}
= \phi_t \circ \widehat{\E} \circ q,
\]
and since the image of $q$ generates $\widehat{\Aa}$ this
implies $\widehat{\E} \circ \widehat{\sigma}_t =
\phi_t \circ \widehat{\E}$.  We therefore have
a strong dilation of the original semigroup.
\end{proof}


\chapter{Covariant Filtrations for Sauvageot Dilations} \label{covariantfiltrations}
\section{Introduction}
In the commutative Daniell-Kolmogorov construction, the
retraction $\E: \Aa \to \AAA$ can be interpreted as follows:
Given a function $f$ on the path space $\SSSS$,
$\E f$ is the function on the state space $S$
with the property that $(\E f)(x)$ is the best
guess at the value of $f(p)$ if the only information we
know about path $p$ is that it starts at the point $x$.  One can generalize this: For any time $t \geq 0$ one can
define a retraction $\E_t$ from $\Aa$ to the functions
on \textbf{stopped path space} $S^{[0,t]}$.  The value of
$\E_t f$ at a stopped path $q$ is the best guess at
the value of $f(p)$ if the only information known about
path $p$ is that its history up to time $t$ is given
by $q$.  The conditional expectations $\{E_t\}$
on $\Aa$ corresponding to the retractions $\{\E_t\}$
satisfy the \textbf{filtration property} $E_t E_s
= E_s E_t = E_s$ for $s \leq t$, and the fact that the
process is Markov implies the \textbf{covariance property}
$\sigma_s E_t = E_{t+s} \sigma_s$ for $s,t \geq 0$.

If one is interested in dilations of cp$_0$-semigroups in the context of a theory of noncommutative Markov processes, it may be desirable to construct not only a dilation of the given cp$_0$-semigroup, but also a covariant filtration on the dilation algebra.  This was done
in the paper \cite{Sauvageot}, in a manner that we now relate.

\section{Strong Right-Liberation}
As we shall see shortly, our filtration will depend upon a method of constructing the Sauvageot product $A \star B$ which takes into account the additional information of a conditional expectation on $A$.  In preparation for this we develop a modification of the liberation properties from chapter \ref{chapliberation}.

\begin{definition}
Let $C, \nu, \eE, A, B, \rho$ be as in Definition \ref{defliberated}.
Let $R: A \to A$ be a linear map
such that $\eE \circ R = \eE$,
and let $A_0 \subset A$ denote the range of $R$.  We say
that we say
that $(A, B, \rho, R)$ is
\textbf{strongly right-liberated}
    if $\eE$ is a left $\la A_0, B\ra$-module map
    and for every $n \geq 1$, $a_2, \dots, a_n \in A$, and $b_1, \dots, b_n \in B$
    satisfying $\nu(b_1) = \dots = \nu(b_{n-1}) = 0$
    and every $x \in A_0$,
    \[
    \eE \Big[ x b_1 \mathring{a}_2
    b_2 \dots \mathring{a}_n b_n
    \Big] = 0.
    \]
A \textbf{strongly right-liberating representation} is
the corresponding analogue of Definition \ref{defliberatingrep}.
\end{definition}

As with right and left liberation, strong right liberation
implies an algorithm for calculating $\eE$ on $\la A, B \ra$.
Accordingly, we introduce three types of words in $\la A, B \ra$:
A word of the \textbf{first type} is, as usual, of
the form $b_0 a_1 b_1 \dots a_\ell b_\ell$ for
some $a_1, \dots, a_\ell \in A$ and $b_0, \dots b_\ell \in B$.  A word of the \textbf{second type} is
of the form $b_0 \tilde{a}_1 b_1 \mathring{a}_2
b_2 \mathring{a}_3 b_3 \dots \mathring{a}_\ell b_\ell$,
where we retain the notation $\mathring{a} =
a - \rho(a)$, and introduce the notation
$\tilde{a} = a - R(a)$.  A word of the \textbf{third type} is of the form $b_0 R(a_1)
b_1 \mathring{a}_2 b_2 \mathring{a}_3 b_3 \dots \mathring{a}_\ell b_\ell$.  As previously, words of all
types are said to be in \textbf{standard form} if
$\nu(b_1) = \dots = \nu(b_{\ell-1}) = 0$.

The relevant center-expand-simplify strategy for calculating $\eE$ on words of the first type is
as follows:
\begin{itemize}
    \item Center the $b_i$ for $0 < i < \ell$,
    expand, and collapse.  The result is a sum
    of standard-form words of the first type, each
    with length at most $\ell$.

    \item Center $a_1$ with respect to $R$
    (that is, write $a_1 = \tilde{a}_1
    + R(a_1)$) and the $a_i$ for
    $i > 1$ with respect to $\rho$.
    Expand and collapse.  The result is a sum
    of words of the second and third types,
    with length at most $\ell$, and with
    length equal to $\ell$ only when in standard
    form.

    \item For words not in standard form, un-center
    the $a_i$, expand, and simplify, thereby
    obtaining a sum of words of the first
    type with length strictly less than $\ell$.
    The procedure can then be recursively
    applied to these.

    \item On standard-form words of
    the second type, $\eE$ vanishes by
    the hypothesis of strong right liberation.

    \item On standard-form words of the third
    type, $\eE$ can be calculated
    using the left $\la A_0, B \ra$-module property:
    \[
    \eE [ b_0 R(a_1) w]
    = b_0 R(a_1) \eE[w]
    \]
    where $w$ is a strictly shorter word,
    on which $\eE$ can be calculated recursively.
\end{itemize}

We now formalize this strategy into a recursive definition.
\begin{itemize}
    \item We modify our definitions of right
    collapse and un-collapse to account for the
    fact that $a_1$ is now centered by $R$
    rather than $\rho$.  Given $\ell \geq 1$,
    $\vec{x} \in \WW_\ell$,
    and a subset $\vec{\iota} \subset [\ell-1]$,
    write $\vec{\iota} = (i_1, \dots, i_m)$.  For
    $k = 1, \dots, m$ we define
    $\beta_k = b_{i_k} \prod\limits_{j=i_k+1}^{i_{k+1}-1} \rho(a_j) b_j$
    as before.  However, $\beta_0$ is now
    always equal to $b_0$.
    In the case $1 \in \vec{\iota}$ we define
    the modified right collapse
    $\RC_s(\vec{x}; \vec{\iota})
    = (\beta_0, a_1, \beta_1, \dots, a_{i_m}, \beta_m)$ as before, but in the case
    $1 \notin \vec{\iota}$ we define
    $\RC_s(\vec{x}; \vec{\iota}) =
    (\beta_0, R(a_1), \beta_1, a_{i_1}, \dots, a_{i_m}, \beta_m)$.  Similar remarks apply to
    the modified un-collapse $\UC_s(\vec{x}; \vec{\iota})$.

    \item We also define, for $\ell > 0$ and
     $\vec{x} \in \WW_\ell$, the \textbf{tail}
    $\vec{x}_+$ to be
    the tuple obtained by truncating the
    first two entries.  That is,
    $(b_0, a_1, b_1\dots, a_\ell, b_\ell)_+ = (b_1, a_2, b_2,
    \dots, a_\ell, b_\ell)$.

    \item We now recursively define
    \begin{align}
    \LM_s(b_0) &=
    \RM_s(b_0) = \UM_s(b_0) = b_0\\
    \LM_s(\vec{x}) &= \sum_{\vec{\iota} \,
    \subseteq [\ell]} \RM_s(\LC_s(\vec{x}; \vec{\iota})) \prod_{j \in [\ell+1] \setminus \vec{\iota}} \nu(x_{2j+1})\\
    \RM_s(\vec{x}) &= \sum_{1 \in \vec{\iota} \, \subsetneq [\ell+1]} \UM_s(\RC_s(\vec{x};\vec{\iota})) + \sum_{1 \notin \vec{\iota} \, \subset [\ell+1]} x_1
    R(x_2) \UM_s(\vec{x}_+; \vec{i}-1)\\
    \UM_s(\vec{x}) &= \sum_{\vec{\iota} \,
    \subseteq [\ell+1]} \LM_s(\UC_s(\vec{x};\vec{\iota}))
    \end{align}
    for $\vec{x} \in \WW_{\ell+1}$.
\end{itemize}

We then arrive at the following theorems,
which we state without proof.

\begin{theorem}
Let $(A, B, \rho, R)$ be strongly right-liberated
in $\AAA$ with respect to $\eE, \nu$.  Then
for any $\vec{x} \in \WW_I$,
\[
\eE \Big[ \Pi( \vec{x}) \Big] = \LM_s(\vec{x})
\eE[\one].
\]
\end{theorem}

\begin{corollary} \label{corstronglyrightlibgenerated}
Let $(A, B, \rho, R)$ be strongly right-liberated
in $\AAA$ with respect to $\eE, \nu$.  Then
\[
\eE \Big[ \la A, B \ra \Big] = \eE \Big[
\la A_0, B \ra \Big].
\]
\end{corollary}

\begin{corollary} \label{corSRLunique}
Let $\AAA$ be a unital algebra, $A, B \subset \AAA$ unital subalgebras, $\rho: A \to B$ a unital
linear map, $R$ a linear transformation on $A$,
 $\nu: \AAA \to \com$ a unital linear
functional.  Suppose $\eE_1, \eE_2: \AAA \to \AAA$
are linear maps satisfying
\begin{itemize}
    \item $\eE_i \circ \rho = \eE_i$
    \item $\eE_i \circ R =\eE_i$
    \item $(A, B, \rho, R)$ is strongly
    right-liberated with respect to $\eE_i, \nu$
\end{itemize}
for $i = 1,2$.

If $\eE_1[\one] = \eE_2[\one]$,
then $\eE_1 = \eE_2$ on
$\la A, B \ra$.
\end{corollary}

\section{Lifting a Retraction to the Sauvageot Product}
We begin this section with a lemma which is of independent
interest.

\begin{proposition} \label{propstinespringretraction}
Let $A_0$ and $A$ be unital C$^*$-algebras, $\iota: A_0 \to A$ a unital embedding, and $\epsilon: A \to A_0$
a unital retraction with respect to $\iota$.  Let $\pi: A_0 \to B(H)$ be a representation.  Then
the minimal Stinespring triple $(K, \psi, V)$ for the map $\pi \circ \epsilon: A \to B(H)$ has the property
that the image of $V$ is invariant under $\psi \circ \iota(A_0)$; that is, $V V^* \psi(\iota(\cdot)) V = \psi(\iota(\cdot)) V$.
\end{proposition}

\begin{proof}
For any $a_0 \in A_0$ and $h \in H$,
\begin{align*}
\| \psi(\iota(a_0)) V h - V V^* \psi(\iota(a_0)) V h\|^2
&= \la \psi(\iota(a_0)) Vh, \psi(\iota(a_0)) Vh \ra\\
&\quad - 2 \realp \la \psi(\iota(a_0)) Vh, V V^* \psi(\iota(a_0)) Vh \ra\\
&\quad+ \la V V^* \psi(\iota(a_0)) Vh, VV^* \psi(\iota(a_0)) Vh \ra\\
&= \la V^* \psi(\iota(a_0^* a_0)) V h, h \ra
- \la V^* \psi(\iota(a_0)) Vh, V^* \psi(\iota(a_0)) Vh \ra\\
&= \la \pi(\epsilon(\iota(a_0^* a_0))) h, h \ra
- \la \pi(\epsilon(\iota(a_0))) h, \pi(\epsilon(\iota(a_0)))h \ra\\
&= \la \pi(a_0^* a_0) h, h \ra - \la \pi(a_0) h, \pi(a_0) h \ra = 0.
\end{align*}
\end{proof}

We now apply this in the context of lifting a Sauvageot retraction.

\begin{theorem} \label{thmliftingretraction}
Let $(A, B, \phi, \omega)$ be a CPC$^*$-tuple (resp.\ CPW$^*$-tuple),
 $A_0$ another unital C$^*$-algebra (resp.\ W$^*$-algebra),
 $A_0 \sa{\iota} A$ a (normal) unital embedding, $A \sa{\epsilon} A_0$
 a corresponding (normal) retraction such that
 $\phi \circ \iota \circ \epsilon = \phi$.
 Forming the Sauvageot
 products of the CP-tuples $(A, B, \phi, \omega)$ and
 $(A_0, B, \phi \circ \iota, \omega)$ with Sauvageot retractions
 $\theta$ and $\theta_0$ respectively, and letting
 $\iota^*$ denote the map $\iota \star \dss{\text{id}}{B}: A_0 \star B \to A \star B$,
 there exists a unique retraction $\epsilon^*: A \star B \to A_0 \star B$ such that the diagrams
 \[ \xymatrix{
 A \ar[r] \ar[d]^\epsilon & A \star B \ar[d]_{\epsilon^*} \\
 A_0 \ar[r] & A_0 \star B
 } \qquad \qquad \xymatrix{
 A \star B \ar[rd]^{\theta} \ar[d]_{\epsilon^*} \\
 A_0 \star B \ar[r]_{\theta_0} &B
 }\]
 commute, and such that $(A \star B, \dss{\psi}{L},
 \dss{\psi}{R}, \theta \circ \dss{\psi}{R},
 \varpi, \epsilon^* \circ \iota^*)$ is a strongly right-liberating
 representation of $(A, B, \phi, \epsilon \circ \iota)$.
\end{theorem}

\begin{proof}
First we show that there exists a faithful representation
$(H, \Omega,\dss{\pi}{R},  K, V, \dss{\pi}{L})$ of
$(A, B, \phi, \omega)$ and an $\dss{\pi}{L}(\iota(A_0))$-invariant subspace $VH \subset K_0 \subset K$ with the property that
$(H, \Omega, \dss{\pi}{R}, K_0, V, \dss{\pi}{L} \circ \iota)$
is a faithful representation of $(A_0, B, \phi \circ \iota, \omega)$.
For this, we begin with a faithful representation $(H, \Omega, \dss{\pi}{R})$ of $(B, \omega)$.  Then, applying Stinespring's
theorem to the completely positive map $\dss{\pi}{R} \circ \phi \circ \iota: A_0 \to B(H)$, we obtain a triple $(K_0, V, \dss{\pi}{L}^{(0)})$
such that $V^* \dss{\pi}{L}^{(0)}(a_0) V = \dss{\pi}{R}(\phi(\iota(a_0)))$.  We may assume WLOG that
$(H, \Omega, \dss{\pi}{R}, K_0, V, \dss{\pi}{L}^{(0)})$ is
faithfully decomposable, as
otherwise we take its direct sum with some faithful representation of $A_0$.  Applying
Stinespring again to the completely positive map
$\dss{\pi}{L}^{(0)} \circ \epsilon: A \to B(K_0)$, we obtain
another triple $(K, W, \dss{\pi}{L})$ such that
$W^* \dss{\pi}{L}(a) W = \dss{\pi}{L}^{(0)}(\epsilon(a))$.  For simplicity, we suppress the notation and regard $K_0$ as a subspace
of $K$; by Proposition \ref{propstinespringretraction},
$(K_0, \dss{\pi}{L}^{(0)})$ is a subrepresentation of $(K,\dss{\pi}{L} \circ \iota)$.  Again we assume without loss of generality
that $(H, \Omega, \dss{\pi}{R}, K, V, \dss{\pi}{L})$ is
faithfully decomposable, as otherwise we replace $(K, \dss{\pi}{L})$
by its direct sum with some faithful representation of $A$.

Having completed this task, we now let $A \star B$ be the Sauvageot product realized by this representation on the Hilbert space
$\Hh = H^- \star L$; by \ref{corunique}, the C$^*$-subalgebra
(resp.\ von Neumann subalgebra) generated by $A_0$ and $B$ acting
on $\Hh$ is $A_0 \star B$.  Letting $\Hh_0 = H^- \star L_0$,
which by Remark \ref{remSPsubspaces} may be regarded as
a subspace of $\Hh$, we see from Proposition \ref{prop3UEreps}
that $\Hh_0$ is $A_0 \star B$-invariant.  It follows that this and the
faithfulness of $(H, \Omega, \dss{\pi}{R}, K_0, V, \dss{\pi}{L}^{(0)})$
that, letting $C: B(\Hh) \to B(\Hh_0)$ be the natural compression,
then the image $C(A_0 \star B)$ is also isomorphic to
$A_0 \star B$, and (modulo this isomorphism) $C$ restricts
to the identity map on $A_0 \star B$.  Then the restriction
of $C$ to $A \star B$, which we denote $\epsilon^*$, is a retraction
satisfying the commuting diagrams in the statement of the theorem.
It remains to check strong right-liberation.  Beginning with
a vector in any summand $L_0^{+ \otimes n} \otimes L_0$
or $H \otimes L_0^{+ \otimes n} \otimes L_0$ and applying
a word of the form
\[
\left[\dss{\psi}{L}(a_1) - \dss{\psi}{R}(\phi(a_1))\right]
\dss{\psi}{R}(b_1) \dots
\left[\dss{\psi}{L}(a_n) - \dss{\psi}{R}(\phi(a_n))\right]
\dss{\psi}{R}(b_n)
\]
produces a vector in a sum of subspaces of the same form.
But then applying an element $\dss{\psi}{R}(a)
- \dss{\psi}{R}(\epsilon(\iota(a)))$ yields a sum of vectors
in subspaces $(L \ominus L_0) \otimes L^{+ \otimes n} \otimes L$,
so that projection onto $\Hh_0 \subset \Hh$ again returns zero.
\end{proof}

The next two results are more easily stated in terms of
conditional expectations, but may be translated into the
language of retractions if desired.

\begin{proposition} \label{propepsilonstarsubalg}
Let $(A, B, \phi, \omega)$ be a CP-tuple, $E$ a conditional
expectation on $A$, $A_1 \subset A$ a C$^*$-subalgebra
(resp.\ W$^*$-subalgebra) such that $E(A_1) \subset A_1$.
Then viewing $A_1 \star B$ as a subalgebra of $A \star B$,
we have $E^*(A_1 \star B) \subset A_1 \star B$.
\end{proposition}

\begin{proof}
This follows from the uniqueness of Sauvageot products
and of the lifted retractions provided by Theorem \ref{thmliftingretraction}, because the restriction of $E$
is a conditional expectation on $A_1$.
\end{proof}

\begin{corollary} \label{corfunctoriallifts}
Let $(A, B, \phi, \omega)$ be a CP-tuple and
$E_1, E_2$ conditional expectations on $A$ such
that $E_1 \circ E_2 = E_1$ (resp.\ $E_1 \circ E_2 = E_2$).  Then
$E_1^* \circ E_2^* = E_1^*$ (resp.\
$E_1^* \circ E_2^* = E_2^*$).
\end{corollary}

\begin{proof}
This follows from the uniqueness of $E_1^*$ (resp.\ $E_2^*$)
in Theorem \ref{thmliftingretraction}, as $E_1^* \circ E_2^*$
satisfies the same properties.
\end{proof}

\section{The Inductive System: Retractions Onto Initial Segments}
In the construction of our inductive system in chapter \ref{chapiteratedproducts}, we produce retractions $\epsilon_\gamma: \AAA_\gamma \to \AAA$ for each $\gamma \in \FF$.  We have now
the tools to produce a consistent family of retractions
$\epsilon_{\beta,\gamma}: \AAA_\gamma \to \AAA_\beta$ whenever
$\beta \leq \gamma$ is an initial segment.

\begin{definition} \label{defretractioninitialsegment}
Let $\gamma = \{t_1, \dots, t_n\} \in \FF$ and
$\beta = \{t_1, \dots, t_m\}$ for some $1 \leq m \leq n$.
We define a retraction $\epsilon_{\beta,\gamma}:
\AAA_\gamma \to \AAA_\beta$ as follows:
\begin{itemize}
    \item We proceed by reverse induction to define
    retractions $\epsilon_{\beta(\ell),\gamma(\ell)}$
    for $\ell = m, \dots, 1$.
    \item In the base
    case $\ell=m$ we define $\epsilon_{\beta(m),\gamma(m)}$
    to be $\epsilon_{\gamma(m)}$ as defined in chapter \ref{chapiteratedproducts}.
    \item Inductively, we define
    \[
    \epsilon_{\beta(\ell),\gamma(\ell)}
    = \epsilon_{\beta(\ell+1),\gamma(\ell+1)}^*
    \]
    in the sense of Theorem \ref{thmliftingretraction}.
\end{itemize}
\end{definition}

\begin{proposition} \label{propconsistentretractionsinitialsegments}
\
\begin{enumerate}
    \item Each map $\epsilon_{\beta,\gamma}$ is a retraction with
    respect to $\dss{f}{\gamma,\beta}$.
    \item If $\beta \leq \gamma
    \leq \delta$ are initial segments, then $\epsilon_{\beta,\gamma} \circ
    \epsilon_{\gamma,\delta} = \epsilon_{\beta,\delta}$.
\end{enumerate}
\end{proposition}

\begin{proof} \
\begin{enumerate}
    \item Given $\beta, \gamma$ as in Definition
    \ref{defretractioninitialsegment}, we
    inductively prove that $\epsilon_{\beta(\ell),
    \beta(\ell)} \circ \dss{f}{\gamma(\ell),\beta(\ell)} =
    \text{id}$.  The base case $\epsilon_\gamma \circ
    \iota_\gamma$ was established in Chapter \ref{chapiteratedproducts}, while the inductive
    step is just the property $\epsilon^* \circ \iota^* =
    \text{id}$ from Theorem \ref{thmliftingretraction}.

    \item Suppose $\beta = \{t_1, \dots, t_m\}$,
    $\gamma = \{t_1, \dots, t_{m+n}\}$, and
    $\delta = \{t_1, \dots, t_{m+n+k}\}$.  We first prove inductively
    that $\epsilon_{\gamma(\ell)} \circ \epsilon_{\gamma(\ell),
    \delta(\ell)} = \epsilon_{\delta(\ell)}$ for
    $\ell = m+n, \dots, 1$.  The base case $\ell=m+n$ is
    trivial, since $\gamma(\ell)$ is a singleton
    and hence $\epsilon_{\gamma(\ell),\delta(\ell)} =
    \epsilon_{\delta(\ell)}$ and $\epsilon_{\gamma(\ell)}
    = \dss{\text{id}}{\AAA}$.  Inductively, recall
    that $\AAA_{\gamma(\ell)}$ is the Sauvageot
    product $\AAA_{\gamma(\ell+1)} \star \AAA$ with
    respect to the map $\phi_{t_{\ell+1}-t_\ell}
    \circ \epsilon_{\gamma(\ell+1)}$, and similarly
    $\AAA_{\delta(\ell)} = \AAA_{\delta(\ell+1)}
    \star \AAA_{\delta(\ell)}$ with respect
    to $\phi_{t_{\ell+1}-t_\ell} \circ \epsilon_{\delta(\ell)}$.
    We also have $\epsilon_{\delta(\ell+1)}
    \circ \dss{f}{\delta(\ell+1),\gamma(\ell+1)}
    = \epsilon_{\gamma(\ell+1)}$ by Proposition
    \ref{propconsistentwith0}.  It follows that
    \begin{align*}
    \phi_{t_{\ell+1}-t_\ell} \circ \epsilon_{\delta(\ell+1)}
    \circ \dss{f}{\delta(\ell+1),\gamma(\ell+1)}
    \circ \epsilon_{\gamma(\ell+1),\delta(\ell+1)}
    &= \phi_{t_{\ell+1}-t_\ell} \circ
    \epsilon_{\gamma(\ell+1)} \circ \epsilon_{\gamma(\ell+1),
    \delta(\ell+1)} \\
    &= \phi_{t_{\ell+1}-t_\ell} \circ \epsilon_{\delta(\ell+1)}
    \end{align*}
    and an application of Corollary \ref{corcommutingsquare}
    implies that $\epsilon_{\gamma(\ell)} \circ \epsilon_{\gamma(\ell),\delta(\ell)} = \epsilon_{\delta(\ell)}$.

    As a particular case we obtain $\epsilon_{\gamma(m)}
    \circ \epsilon_{\gamma(m),\delta(m)} = \epsilon_{\delta(m)}$,
    which then becomes the base case of a new induction,
    proving that $\epsilon_{\beta(\ell), \gamma(\ell)}
    \circ \epsilon_{\gamma(\ell),\delta(\ell)}
    = \epsilon_{\beta(\ell),\delta(\ell)}$ for
    $\ell = m, \dots, 1$.  In the case $\ell=m$
    we have $\beta(\ell)$ a singleton, so that
    $\epsilon_{\beta(\ell),\gamma(\ell)} = \epsilon_{\gamma(\ell)}$ and we reduce to the result just established.
    The induction on $\ell$ follows by Corollary \ref{corfunctoriallifts}.
\end{enumerate}
\end{proof}

Fix now any $\tau \geq 0$.  Let $\FF_\tau \subset \FF$
denote the finite subsets of $[0,\tau]$, and for
$\gamma \in \FF$ let $\gamma \la \tau \ra$ denote
$\gamma \cap [0,\tau]$.  Now $\{\AAA_\gamma \mid \gamma \in \FF_\tau\}$
with the same embeddings $\dss{f}{\gamma, \beta}$ from before
is an inductive system; let $\Aa_\tau$ denote its limit,
with embeddings $\dss{g}{\infty,\gamma}: \AAA_\gamma \to \Aa_\tau$.
(We note that we can equivalently obtain $\Aa_\tau$ as the limit of a system indexed by $\FF$, with the object $\AAA_{\gamma \la \tau \ra}$ corresponding to the set $\gamma$ and the morphism
$\dss{f}{\gamma \la \tau \ra, \beta \la \tau \ra}$ corresponding to the inclusion $\beta \leq \gamma$.)
Through the universal property of the limit we obtain an
embedding $h_\tau: \Aa_\tau \to \Aa$ characterized by
\begin{equation} \label{eqndefhtau}
\dss{f}{\infty,\gamma} = h_\tau \circ \dss{g}{\infty,\gamma}
\qquad \text{ for all } \gamma \in \FF_\tau.
\end{equation}

Now a slight adaptation of the proof of Proposition \ref{propconsistentwith0} shows that for any inclusion
$\beta \leq \gamma$ in $\FF$ and any $\ell$ such
that $t_\ell \in \beta$, one has
\[
\epsilon_{\gamma(\ell),\gamma} \circ \dss{f}{\gamma,\beta}
= \dss{f}{\gamma(\ell),\beta(\ell)} \circ \epsilon_{\beta(\ell),
\beta}.
\]
It follows that if $\beta \leq \gamma$ are sets both of which
contain $\tau$, then
\[
\epsilon_{\gamma \la \tau \ra, \gamma} \circ \dss{f}{\gamma, \beta}
= \dss{f}{\gamma \la \tau \ra, \beta \la \tau \ra} \circ
\epsilon_{\beta \la \tau \ra, \beta}.
\]
Holding $\beta$ fixed and taking a limit in $\gamma$, this implies
the existence of a retraction $\E_\tau: \Aa \to \Aa_\tau$
characterized by
\begin{equation} \label{eqndefEtau}
\E_\tau \circ \dss{f}{\infty,\beta} = \dss{g}{\infty, \beta \la \tau \ra} \circ \epsilon_{\beta \la \tau \ra, \beta} \qquad \text{ for all }
\beta \in \FF.
\end{equation}
We let $E_\tau = h_\tau \circ \E_\tau$ denote the corresponding
conditional expectation on $\Aa$, which is characterized by the
property
\begin{equation} \label{eqndefEtauexpectation}
E_\tau \circ \dss{f}{\infty,\beta} = \dss{f}{\infty, \beta \la \tau \ra} \circ \epsilon_{\beta \la \tau \ra, \beta} \qquad \text{ for all }
\beta \in \FF.
\end{equation}

\newpage
\begin{proposition}[Filtration] \label{propfiltration}
For all $s \leq t$, $E_s E_t = E_s = E_t E_s$.
\end{proposition}

\begin{proof}
Since $s \leq t$, we have $\beta \la s \ra \la t \ra
= \beta \la s \ra = \beta \la t \ra \la s \ra$ for all
$\beta \in \FF$.  Then
\begin{align*}
E_s \circ E_t \circ \dss{f}{\infty,\beta} &=
E_s \circ \dss{f}{\infty,\beta \la t \ra} \circ \epsilon_{\beta \la t \ra, \beta}\\
&= \dss{f}{\infty, \beta \la s \ra} \circ \epsilon_{\beta \la s \ra, \beta \la t \ra} \circ \epsilon_{\beta \la t \ra, \beta}\\
&= \dss{f}{\infty, \beta \la s \ra} \circ \epsilon_{\beta\la s \ra, \beta} \\
&= E_s \circ \dss{f}{\infty, \beta}
\end{align*}
and
\begin{align*}
E_t \circ E_s \circ \dss{f}{\infty,\beta} &=
E_t \circ \dss{f}{\infty, \beta \la s \ra} \circ
\epsilon_{\beta \la s \ra, \beta} \\
&= \dss{f}{\infty, \beta \la s \ra} \circ \epsilon_{\beta \la s \ra, \beta}\\
&= E_s \circ \dss{f}{\infty,\beta}
\end{align*}
and as the images of the $\dss{f}{\infty,\beta}$ generate
$\Aa$, this implies $E_t \circ E_s = E_s = E_s \circ E_t$.
\end{proof}

\begin{proposition}[Covariance] \label{propcovariance}
For any $s,t \geq 0$,
\[
\sigma_s E_t = E_{t+s} \sigma_s.
\]
\end{proposition}

\begin{proof}
First, we note that for any initial segment $\beta \leq \gamma$ in $\FF$
and any $s \geq 0$, $\epsilon_{\beta + s, \gamma+s}
= \epsilon_{\beta, \gamma}$, since the definition of $\epsilon_{\beta,\gamma}$ depends only on the time differences in $\gamma$.  We also have $\gamma \la t \ra +s
= (\gamma + s) \la t+s \ra$ for any $s,t \geq 0$ and
$\gamma \in \FF$.  Then
\begin{align*}
E_{t+s} \circ \sigma_s \circ \dss{f}{\infty,\gamma}
&= E_{t+s} \circ \dss{f}{\infty,\gamma+s}\\
&= \dss{f}{\infty, (\gamma+s) \la t+s \ra} \circ \epsilon_{(\gamma+s) \la t+s \ra, \gamma +s}\\
&= \dss{f}{\infty, \gamma \la t \ra +s} \circ
\epsilon_{\gamma \la t \ra+s, \gamma +s}\\
&= \dss{f}{\infty, \gamma \la t \ra + s} \circ
\epsilon_{\gamma \la t \ra, \gamma}\\
&= \sigma_s \circ \dss{f}{\infty, \gamma \la t \ra}
\circ \epsilon_{\gamma \la t \ra, \gamma}\\
&= \sigma_s \circ E_t \circ \dss{f}{\infty,\gamma}
\end{align*}
which implies the result.
\end{proof}

\section{Covariant Filtrations on W$^*$-Dilations?}
So far we have produced a filtration of conditional
expectations $\{E_t\}_{t \geq 0}$ on the C$^*$-dilation algebra
$\Aa$, which is covariant with respect to the semigroup $\{\sigma_t\}$.  In chapter \ref{chapcontinuous}, we showed how, when our initial semigroup acts on a W$^*$-algebra, we can modify the dilation to achieve a W$^*$-algebra $\widetilde{\Aa}$ and a continuous semigroup of normal
endomorphisms $\{\tilde{\sigma}_t\}$ of $\widetilde{\Aa}$.  It is
natural, then, to seek a filtration $\{\widetilde{E}_t\}$ of
normal conditional expectations on $\widetilde{\Aa}$ which is
covariant with respect to $\{\tilde{\sigma}_t\}$, which is continuous
 in the sense that $t \mapsto E_t(a)$ is strongly continuous for fixed $a \in \widetilde{\Aa}$, and which
is also related to our C$^*$-filtration by the diagram
\[ \xymatrix{
\Aa \ar[r]^{E_t} \ar[d]_\psi & \Aa \ar[d]^\psi \\
\widetilde{\Aa} \ar[r]_{\widetilde{E}_t} & \widetilde{\Aa}
}\]
This is very similar in spirit to the question of how to define
the maps $\tilde{\sigma}_t$, which was addressed in Theorem \ref{thmexistsnormalextension}.  A strategy for answering it through similar methods would be as follows:
\begin{itemize}
    \item For each $\tau \geq 0$, define ``$\tau$-moment polynomials''
    $\Ss_\tau(\vec{t}; \vec{a})$ by modifying the recursion
    in Definition \ref{defmomentpolynomials}.  The point of this
    would be to show that
    \[
    \Ss_\tau(\vec{t}; \vec{a})
    = \E_\tau \big[ \sigma_{t_1}(i(a_1)) \cdots
    \sigma_{t_n}(i(a_n)) \big].
    \]
    These $\tau$-moment polynomials should also be weakly
    continuous in each entry of $\vec{a}$, and
    jointly strongly continuous in $\vec{t}$ and
    $\vec{a}$ subject to the non-crossing
    restriction (indeed, this restriction
    may only be needed among those times greater than
    or equal to $\tau$) and possibly with an additional
    restriction that times not cross $\tau$.  Finally,
    for any fixed $\vec{t}$ and $\vec{a}$,
    $\Ss_\tau(\vec{t}; \vec{a})$ should be strongly
    continuous in $\tau$ for those $\tau$ not equal to
    any entry of $\vec{t}$.

    \item In the spirit of Lemma \ref{lemexistsQ}, one
    could find for each $y,z \in \PP$ and $t \geq 0$ elements
    $y_0, z_0 \in \PP$ and a normal linear map $Q$ on $\AAA$ such that
    $\E [ y E_t(x) z] = Q \circ \E[y_0 x z_0]$ for all
    $x \in \Aa$.

    \item As in Theorem \ref{thmexistsnormalextension}, one
    would have for all $x \in \Aa$, $y_0,z_0 \in \PP$,
    and $\xi', \eta' \in H$ that
    \[
    \la \psi(E_t(x)) \psi(y) V \xi', \psi(z) V \eta' \ra
    = \la Q (\E[z_0^* \psi(x) y_0]) \xi', \eta'\ra
    \]
    and could therefore define $\widetilde{E}_t$ by its sesquilinear form
    \[
    \la \widetilde{E}_t(X) \psi(y) V \xi', \psi(z) V \eta'
    \ra = \la Q(\widetilde{\E}[z_0^* X y_0]) \xi', \eta'.
    \]
\end{itemize}

At the time of thesis submission (March 28, 2013), I have not verified the success of this approach.


\chapter{Product Systems} \label{chapproductsystems}

\section{Introduction}
For a Hilbert space $H$, the Fock space
\[
\FF(H) = \bigoplus_{n=0}^\infty H^{\otimes n}
\]
can be understood as follows: The collection $\{H^{\otimes n}\}$
forms a \textbf{bundle} of Hilbert spaces over $\N$, and furthermore,
a bundle which \textbf{tensors associatively}: there is a family
of unitary equivalences $H^{\otimes n} \otimes H^{\otimes m} \to H^{\otimes (m+n)}$ which compose in such a way that \\
$H^{\otimes n} \otimes H^{\otimes m} \otimes H^{\otimes k}
\to H^{\otimes (m+n+k)}$ is unambiguous.  The Fock space
then consists of the \textbf{square summable sections} of this bundle.

When formulated in this way, one can naturally define a
``continuous analogue of Fock space'' (\cite{ArvesonContinuousAnalogues1}-\cite{ArvesonContinuousAnalogues4}),
called a \textbf{product system of Hilbert spaces}, to be a bundle of Hilbert spaces over $(0,\infty)$ which tensors associatively.  As usual when one forms bundles of
Hilbert spaces, measurability hypotheses come into play, of which we omit the details here.  (The analogue of Fock space is
not the bundle itself, but rather the associated Hilbert space of its square-integrable sections.)  Without stating all the results precisely, here are some of the relevant aspects of the theory.
\begin{enumerate}
    \item Product systems of Hilbert spaces are naturally associated
    with E$_0$-semigroups; indeed, once appropriate morphisms have
    been defined, there is an equivalence of categories between
    (equivalence classes of) product systems of Hilbert spaces
    on the one hand, and
    E$_0$-semigroups on $B(H)$ on the other.

    \item There are several ways to construct a product system
    from a given E$_0$-semigroup, one of which we mention is
    the following technique
    of Arveson: Given the semigroup $\{\alpha_t\}$ on $B(H)$,
    form the intertwining spaces
    \[
    E_t = \{X \in B(H) \mid \forall Y \in B(H): \
    XY = Y \alpha_t(X).\}
    \]
    A straightforward calculation shows that for $X,Y \in E_t$,
    the operator $X^* Y$ must commute with everything in $B(H)$,
    so it corresponds to a scalar which we define to be
    $\la X, Y\ra$.  Furthermore, $E_t$ turns out to be complete
    in this inner product, hence a Hilbert space; moreover, these
    spaces ``tensor'' associatively, where for $X \in E_t$ and
    $Y \in E_s$ the ``tensor product'' $X \otimes Y \in E_{t+s}$
    is just the composition of operators $XY$.

    \item There is a classification of product systems into
    types I, II, and III, similar in spirit to the type theory
    of von Neumann algebras.  The classification is based
    on the notion of a \textbf{unit} for a product system
    $\{E_t\}$, which is a family of (unit) vectors $u_t \in E_t$
    that follow the given embeddings, that is, such that
    $u_t \otimes u_s$ is identified with $u_{s+t}$.  If
    the units ``span'' the product system in the appropriate sense,
    it is type I; if there exists at least one unit but they
    do not span the system, it is type II; if there are no
    units, it is type III.  Type I systems are further classified
    according to their \textbf{index}, which is a
    number in $\N \cup \{\infty\}$ defined as the dimension
    of a certain Hilbert space associated to the set
    of units, and equal to the index
    (as defined by Powers in \cite{PowersIndexTheory}) of the
    associated E$_0$-semigroup.
\end{enumerate}

\section{Hilbert C$^*$-Modules and Correspondences}
Given a C$^*$-algebra $A$, a \textbf{Hilbert $A$-module} is
a right $A$-module $E$ with an ``$A$-valued inner product''
such that $E$ is complete in the associated norm.  A Hilbert
space is precisely a Hilbert $\com$-module.  As
 another notable example, if $X$ is a locally compact Hausdorff space,    a vector bundle over $X$ (in which each fiber
    is a closed subspace of some fixed Hilbert space)
    is a Hilbert $C(X)$-module; for this reason Hilbert C$^*$-modules are sometimes conceptualized as
    ``noncommutative vector bundles.''  The basic
    theory can be found in the seminal papers \cite{Paschke}
    \cite{Rieffel}, \cite{Kasparov} and the more recent
    sources \cite{Lance}, \cite{ManuilovTroitsky}, and
    \cite{RaeburnWilliams}.  Among the notable features
    are the replacement of bounded operators on $E$
    with the more restrictive notion of \textbf{adjointable}
    operators, the set of which is denoted $\LL(E)$ and
    forms a C$^*$-algebra; if $A$ is a W$^*$-algebra and
    $E$ is \textbf{self-dual} (a property not enjoyed by
    all C$^*$-modules), then $\LL(E)$ is a W$^*$-algebra.
    If $A$ is a W$^*$-algebra and $E$ is not self-dual, one
    typically works instead with its \textbf{self-dual completion}.

Given C$^*$-algebras $A$ and $B$, a Hilbert $B$-module equipped
with a left action of $A$ (by which one means a *-homomorphism
from $A$ to $\LL(E)$) is called an \textbf{$(A,B)$-correspondence}.
As with balanced tensor products of bimodules over rings, one can form the tensor product (sometimes called the \textbf{internal tensor product}) of an $(A,B)$-correspondence with a $(B, C)$-correspondence to obtain an $(A, C)$-correspondence.  In particular, given
an $(A,A)$-correspondence $E$, one can form tensor powers
$E^{\otimes n}$, and thereby also the Fock correspondence $\FF(E) = \bigoplus_{n=0}^\infty E^{\otimes n}$.  More generally, one
can form a \textbf{product system of $(A,A)$-correspondences}
in a fashion analogous to the last section.

As outlined above, Arveson's intertwining technique produces a product system of Hilbert spaces from an E$_0$-semigroup on $B(H)$.  When
one considers an E$_0$-semigroup on a general von Neumann algebra $\MM$, however, the same technique produces a product system of correspondences over the commutant $\MM'$.
Hence, product systems of correspondences arise naturally
in the study of E$_0$-semigroups on von Neumann algebras.

The classification theory of product systems of correspondences is more complicated than that of product systems of Hilbert spaces, in part because it is more complicated even to define what a unit is.  Following \cite{SkeideWhiteNoises}, we define a unit to be a family
of elements $\xi_t \in E_t$ which tensor associatively, a
\textbf{unital} unit to be one for which $\la \xi_t, \xi_t \ra = \one$
for all $t$ (this is not automatic even with the hypothesis
that $\|\xi_t\| = 1$), and a \textbf{central} unit to be one for which
the left and right actions of $\MM$ agree, i.e. $m \cdot \xi_t
= \xi_t \cdot m$ for all $t$ and all $m \in \MM$.  Central unital units are of particular importance in classification theory; fortunately, it is known that for product systems of von Neumann modules, the existence of a central \textbf{contractive} unit (meaning $\la \xi_t, \xi_t \ra \leq \one$ for all $t$) implies the existence of a central unital unit.
A product system having such a unit is called either \textbf{spatial} or \textbf{non-type-III}, though sometimes those terms are distinguished.

\section{Units for Product Systems Associated With Sauvageot Dilations}
In this section we present some calculations toward the construction
of a unit for the product system arising from a Sauvageot dilation.
For convenience we work with the dilation $(\widetilde{\Aa}, \{\widetilde{\sigma}_t\})$ instead of its quotient $(\widehat{\Aa},
\{\widehat{\sigma}_t\})$.
\begin{itemize}
    \item Define for each $t \geq 0$ an operator  $U_t \in B(\Hh)$
    by its action on $\psi(\PP) VH$:
    \[
    U_t \psi(\sigma_{t_1}(i(a_1))) \cdots
    \psi(\sigma_{t_n}(i(a_n))) V h = \psi(\sigma_{t_1+t}(i(a_1)))
    \cdots \psi(\sigma_{t_n+t}(i(a_n))) Vh.
    \]

    \item We omit the verification that this extends to
    a well-defined map on the linear span of $\psi(\PP)VH$,
    and show its contractivity.  For simplicity, consider
    an element of $\psi(\PP)VH$, and compute
    \begin{align*}
    \left\| U_t \psi(\sigma_{t_1}(i(a_1))) \cdots
    \psi(\sigma_{t_n}(i(a_n))) V h\right\|^2&=
    \left \|\psi(\sigma_{t_1+t}(i(a_1))) \cdots
    \psi(\sigma_{t_n+t}(i(a_n))) V h\right\|^2\\
    &= \left\la V^* \psi \Big( \sigma_{t_n+t}(i(a_n))^* \cdots
    \sigma_{t_n+t}(i(a_n)) \Big) V h, h \right \ra\\
    &= \left\la \pi \circ \E \circ \sigma_t \Big(
    \sigma_{t_n}(i(a_n))^* \cdots \sigma_{t_n}(i(a_n)) \Big) h,h
    \right\ra\\
    &= \left\la \pi \circ \phi_t \circ \E \Big(
    \sigma_{t_n}(i(a_n))^* \cdots \sigma_{t_n}(i(a_n)) \Big) h,h
    \right\ra\\
    &\leq \left\la \pi \circ \E \Big(
    \sigma_{t_n}(i(a_n))^* \cdots \sigma_{t_n}(i(a_n)) \Big) h,h
    \right\ra\\
    &= \| \psi (\sigma_{t_1}(i(a_1))) \cdots
    \psi(\sigma_{t_n}(i(a_n))) Vh\|^2
    \end{align*}
    since each $\phi_t$ is contractive.

    \item Clearly $U_t U_s = U_{t+s}$.  To show that $\{U_t\}$
    is a unit, we need to verify that $U_t Y = \widetilde{\sigma}_t(Y) U_t$ for all $Y \in \widetilde{\Aa}$.  By the normality of both sides, it suffices to establish this for $Y$ in the weakly dense subalgebra $\psi(\Aa)$, which reduces again to consideration of
    elements $\psi(\PP)$.  Now letting $\sigma_{t_1}(i(a_1))
    \cdots \sigma_{t_n}(i(a_n))$ be a typical element of $\psi(\PP)$
    and $\sigma_{s_1}(i(b_1)) \cdots \sigma_{s_m}(i(b_m)) Vh$
    a typical vector in $\Hh$,
    \begin{align*}
    &U_t \psi \big( \sigma_{t_1}(i(a_1)) \cdots
    \sigma_{t_n}(i(a_n))  \sigma_{s_1}(i(b_1)) \cdots
     \sigma_{s_m}(i(b_m)) \big) Vh\\
      &\qquad \qquad \qquad = \psi \big( \sigma_{t_1+t}(i(a_1))
    \cdots \sigma_{t_n+t}(i(a_n)) \sigma_{s_1+t}(i(b_1))
    \cdots \sigma_{s_m+t}(i(b_m)) \big) Vh\\
    &\qquad \qquad \qquad = \bigg[ \widetilde{\sigma}_t
    \big(\psi \big( \sigma_{t_1}(i(a_1)) \cdots \sigma_{t_n}(i(a_n)) \big) \big) \bigg] U_t \bigg[
    \psi \big( \sigma_{s_1}(i(b_1)) \cdots \sigma_{s_m}(i(b_m))
    \big) Vh \bigg]
    \end{align*}
    as desired.

    \item So far we have a contractive unit.  To show
    centrality, consider an operator $X \in \widetilde{\Aa}'$
    and a typical vector $\psi(p) Vh$ for $p \in \PP$.  Note
    that $X$ commutes with $\psi(p)$ and $\psi(\sigma_t(p))$,
    as both are elements of $\widetilde{\Aa}$.  Then
    \[
    U_t X \psi(p) Vh = U_t \psi(p) X Vh = \psi(\sigma_t(p)) X Vh
    = X \psi(\sigma_t(p)) Vh = X U_t \psi(p) Vh
    \]
    so that $U_t X = X U_t$.
\end{itemize}

The existence of a central contractive unit implies that the
dilation semigroup is non-type-III, as referenced above.  In particular, this is the case even when the original CP-semigroup happens to be a type III E$_0$-semigroup---a striking result indeed!

\appendix
\chapter{Table of Values of Collapse and Moment Functions}

In the tables to come, $\vec{x}_\ell$ denotes
the tuple $(b_0, a_1, b_1, \dots, a_\ell, b_\ell)$,
and $\vec{y}_\ell$ the tuple $(a_0, b_1, a_1, \dots, b_\ell,
a_\ell)$.

\section{Right-Liberation Collapse Functions}

\begin{align*}
\LC(\vec{x}_1; \emptyset) &= \vec{x}_1 &
 \LC(\vec{x}_3; \emptyset) &= (b_0, a_1 a_2 a_3, b_3)\\
\RC(\vec{x}_1; \emptyset) &= b_0 \rho(a_1) b_1 &
 \LC(\vec{x}_3; \{1\}) &= (b_0, a_1, \mathring{b}_1, a_2 a_3,
b_3) \\
\RC(\vec{x}_1; \{1\}) &= \vec{x}_1 &
 \LC(\vec{x}_3; \{2\}) &= (b_0, a_1 a_2, \mathring{b}_2, a_3,
b_3) \\
\UC(\vec{x}_1; \emptyset) &= -b_0 \rho(a_1) b_1 &
 \LC(\vec{x}_3; \{1,2\}) &= (b_0, a_1, \mathring{b}_1, a_2, \mathring{b}_2, a_3, b_3) \\
\UC(\vec{x}_1; \{1\}) &= \vec{x}_1 &
 \RC(\vec{x}_3; \emptyset) &= b_0 \rho(a_1) b_1 \rho(a_2) b_2 \rho(a_3) b_3\\
\LC(\vec{x}_2; \emptyset) &= (b_0, a_1 a_2, b_1)
&
 \RC(\vec{x}_3; \{1\}) &= (b_0, a_1, b_1 \rho(a_2) b_2 \rho(a_3)b_3)\\
\LC(\vec{x}_2; \{1\}) &= (b_0, a_1, \mathring{b}_1, a_2, b_2) &
 \RC(\vec{x}_3; \{2\}) &= (b_0 \rho(a_1) b_1, a_2, b_2 \rho(a_3)b_3)\\
\RC(\vec{x}_2; \emptyset) &= (b_0 \rho(a_1)
b_1 \rho(a_2) b_2) &
 \RC(\vec{x}_3; \{3\}) &= (b_0 \rho(a_1) b_1 \rho(a_2) b_2,a_3,b_3)\\
\RC(\vec{x}_2; \{1\}) &= (b_0, a_1, b_1 \rho(a_2)
b_2) &
 \RC(\vec{x}_3; \{1,2\}) &= (b_0, a_1, b_1, a_2, b_2 \rho(a_3) b_3)\\
\RC(\vec{x}_2; \{2\}) &= (b_0 \rho(a_1) b_1, a_2, b_2) &
 \RC(\vec{x}_3; \{1,3\}) &= (b_0, a_1, b_1 \rho(a_2) b_2, a_3, b_3)\\
\RC(\vec{x}_2; \{1,2\}) &= \vec{x}_2 &
 \RC(\vec{x}_3; \{2,3\}) &= (b_0 \rho(a_1) b_1, a_2, b_2, a_3, b_3)\\
\UC(\vec{x}_2; \emptyset) &= (b_0 \rho(a_1)
b_1 \rho(a_2) b_2) &
 \RC(\vec{x}_3; \{1,2,3\}) &=(b_0, a_1, b_1, a_2,b_2, a_3,b_3)\\
\UC(\vec{x}_2; \{1\}) &= (b_0, a_1, -b_1 \rho(a_2)
b_2) &
 \UC(\vec{x}_3; \emptyset) &= -b_0 \rho(a_1) b_1 \rho(a_2) b_2 \rho(a_3) b_3\\
\UC(\vec{x}_2; \{2\}) &= (-b_0 \rho(a_1) b_1, a_2, b_2)  &  \UC(\vec{x}_3; \{1\}) &= (b_0, a_1, b_1 \rho(a_2) b_2 \rho(a_3)b_3)\\
\UC(\vec{x}_2; \{1,2\}) &= \vec{x}_2
& \UC(\vec{x}_3; \{2\}) &= (-b_0 \rho(a_1) b_1, a_2, -b_2 \rho(a_3)b_3)\\
& & \UC(\vec{x}_3; \{3\}) &= (b_0 \rho(a_1) b_1 \rho(a_2) b_2,a_3,b_3)\\
& & \UC(\vec{x}_3; \{1,2\}) &= (b_0, a_1, b_1, a_2, -b_2 \rho(a_3) b_3)\\
& & \UC(\vec{x}_3; \{1,3\}) &= (b_0, a_1,- b_1 \rho(a_2) b_2, a_3, b_3)\\
& & \UC(\vec{x}_3; \{2,3\}) &= (-b_0 \rho(a_1) b_1, a_2, b_2, a_3, b_3)\\
& & \UC(\vec{x}_3; \{1,2,3\}) &=(b_0, a_1, b_1, a_2,b_2, a_3,b_3)
\end{align*}

\newpage

\begin{align*}
\LC(\vec{x}_4; \emptyset) &= (b_0, a_1 a_2 a_3 a_4, b_4)\\
\LC(\vec{x}_4; \{1\}) &= (b_0, a_1, \mathring{b}_1, a_2 a_3 a_4, b_4)\\
\LC(\vec{x}_4; \{2\}) &= (b_0, a_1 a_2, \mathring{b}_2, a_3 a_4, b_4)\\
\LC(\vec{x}_4; \{3\}) &= (b_0, a_1 a_2 a_3, \mathring{b}_3, a_4, b_4)\\
\LC(\vec{x}_4; \{1,2\}) &= (b_0, a_1, \mathring{b}_1, a_2,
 \mathring{b}_2, a_3 a_4, b_4)\\
\LC(\vec{x}_4; \{1,3\}) &= (b_0, a_1, \mathring{b}_1, a_2 a_3,
 \mathring{b}_3, a_4, b_4)\\
 \LC(\vec{x}_4; \{2,3\}) &= (b_0, a_1 a_2, \mathring{b}_2, a_3,
 \mathring{b}_3, a_4, b_4)\\
\LC(\vec{x}_4; \{1,2,3\}) &= (b_0, a_1, \mathring{b}_1,
a_2, \mathring{b}_2, a_3, \mathring{b}_3, a_4, b_4)\\
\RC(\vec{x}_4; \emptyset) &= b_0 \rho(a_1) b_1 \rho(a_2)
b_2 \rho(a_3) b_3 \rho(a_4) b_4\\
\RC(\vec{x}_4; \{1\}) &= (b_0, a_1, b_1 \rho(a_2) b_2 \rho(a_3) b_3 \rho(a_4) b_4)\\
\RC(\vec{x}_4; \{2\}) &= (b_0 \rho(a_1) b_1, a_2, b_2 \rho(a_3) b_3
\rho(a_4) b_4)\\
\RC(\vec{x}_4; \{3\}) &= (b_0 \rho(a_1) b_1 \rho(a_2) b_2,
a_3, b_3 \rho(a_4) b_4)\\
\RC(\vec{x}_4; \{4\}) &= (b_0 \rho(a_1) b_1 \rho(a_2) b_2
\rho(a_3) b_3, a_4, b_4)\\
\RC(\vec{x}_4; \{1,2\}) &= (b_0, a_1, b_1, a_2, b_2 \rho(a_3)
b_3 \rho(a_4) b_4)\\
\RC(\vec{x}_4; \{1,3\}) &= (b_0, a_1, b_1 \rho(a_2) b_2,
a_3, b_3 \rho(a_4) b_4)\\
\RC(\vec{x}_4; \{1,4\}) &= (b_0, a_1, b_1 \rho(a_2) b_2
\rho(a_3) b_3, a_4, b_4)\\
\RC(\vec{x}_4; \{2,3\}) &= (b_0 \rho(a_1) b_1, a_2, b_2,
a_3, b_3 \rho(a_4) b_4)\\
\RC(\vec{x}_4; \{2,4\}) &= (b_0 \rho(a_1) b_1, a_2, b_2
\rho(a_3) b_3, a_4, b_4)\\
\RC(\vec{x}_4; \{3,4\}) &= (b_0 \rho(a_1) b_1 \rho(a_2)
b_2, a_3, b_3, a_4, b_4)
\end{align*}

\begin{align*}
\RC(\vec{x}_4; \{1,2,3\}) &= (b_0, a_1, b_1, a_2, b_2, a_3,
b_3 \rho(a_4) b_4)\\
\RC(\vec{x}_4; \{1,2,4\}) &= (b_0, a_1, b_1, a_2, b_2 \rho(a_3)
b_3, a_4, b_4)\\
\RC(\vec{x}_4; \{1,3,4\}) &= (b_0, a_1, b_1 \rho(a_2) b_2, a_3,
b_3, a_4,b_4)\\
\RC(\vec{x}_4; \{2,3,4\}) &= (b_0 \rho(a_1) b_1, a_2, b_2, a_3,
b_3, a_4, b_4)\\
\RC(\vec{x}_4; \{1,2,3,4\}) &= (b_0, a_1, b_1, a_2, b_2, a_3,
b_3, a_4, b_4)\\
\UC(\vec{x}_4; \emptyset) &= b_0 \rho(a_1) b_1 \rho(a_2)
b_2 \rho(a_3) b_3 \rho(a_4) b_4\\
\UC(\vec{x}_4; \{1\}) &= (b_0, a_1, -b_1 \rho(a_2) b_2 \rho(a_3) b_3 \rho(a_4) b_4)\\
\UC(\vec{x}_4; \{2\}) &= (-b_0 \rho(a_1) b_1, a_2, b_2 \rho(a_3) b_3
\rho(a_4) b_4)\\
\UC(\vec{x}_4; \{3\}) &= (b_0 \rho(a_1) b_1 \rho(a_2) b_2,
a_3, -b_3 \rho(a_4) b_4)\\
\UC(\vec{x}_4; \{4\}) &= (-b_0 \rho(a_1) b_1 \rho(a_2) b_2
\rho(a_3) b_3, a_4, b_4)\\
\UC(\vec{x}_4; \{1,2\}) &= (b_0, a_1, b_1, a_2, b_2 \rho(a_3)
b_3 \rho(a_4) b_4)\\
\UC(\vec{x}_4; \{1,3\}) &= (b_0, a_1, -b_1 \rho(a_2) b_2,
a_3, -b_3 \rho(a_4) b_4)\\
\UC(\vec{x}_4; \{1,4\}) &= (b_0, a_1, b_1 \rho(a_2) b_2
\rho(a_3) b_3, a_4, b_4)\\
\UC(\vec{x}_4; \{2,3\}) &= (-b_0 \rho(a_1) b_1, a_2, b_2,
a_3, -b_3 \rho(a_4) b_4)\\
\UC(\vec{x}_4; \{2,4\}) &= (-b_0 \rho(a_1) b_1, a_2, -b_2
\rho(a_3) b_3, a_4, b_4)\\
\UC(\vec{x}_4; \{3,4\}) &= (b_0 \rho(a_1) b_1 \rho(a_2)
b_2, a_3, b_3, a_4, b_4)\\
\UC(\vec{x}_4; \{1,2,3\}) &= (b_0, a_1, b_1, a_2, b_2, a_3,
-b_3 \rho(a_4) b_4)\\
\UC(\vec{x}_4; \{1,2,4\}) &= (b_0, a_1, b_1, a_2, -b_2 \rho(a_3)
b_3, a_4, b_4)\\
\UC(\vec{x}_4; \{1,3,4\}) &= (b_0, a_1, -b_1 \rho(a_2) b_2, a_3,
b_3, a_4,b_4)\\
\UC(\vec{x}_4; \{2,3,4\}) &= (-b_0 \rho(a_1) b_1, a_2, b_2, a_3,
b_3, a_4, b_4)\\
\UC(\vec{x}_4; \{1,2,3,4\}) &= (b_0, a_1, b_1, a_2, b_2, a_3,
b_3, a_4, b_4)
\end{align*}

\newpage
\section{Right-Liberation Moment Functions}
\begin{align*}
\RM(\vec{x}_1) &= b_0 \rho(a_1) b_1\\ 
\LM(\vec{x}_1) &= b_0 \rho(a_1) b_1\\
\UM(\vec{x}_1) &= 0 \\
\RM(\vec{x}_2) &= b_0 \rho(a_1) b_1 \rho(a_2) b_2\\
\LM(\vec{x}_2) &= \nu(b_1) b_0 \big[\rho(a_1 a_2)
- \rho(a_1) \rho(a_2)\big] b_2 + b_0 \rho(a_1) b_1 \rho(a_2) b_2\\
\UM(\vec{x}_2) &= \nu(b_1) b_0 \big[\rho(a_1 a_2)
- \rho(a_1) \rho(a_2)\big] b_2\\
\RM(\vec{x}_3) &= b_0 \rho(a_1) b_1 \rho(a_2) b_2 \rho(a_3) b_3
+ \nu(b_1) b_0 \big[ \rho(a_1 a_2) - \rho(a_1) \rho(a_2) \big] b_3\\
&+ \nu\big(b_1 \rho(a_2) b_2 \big) b_0 \big[ \rho(a_1 a_3) - \rho(a_1)
\rho(a_3) \big] b_3
+ \nu(b_2) b_0 \big[ \rho(a_2 a_3) - \rho(a_2) \rho(a_3) \big] b_3\\
\LM(\vec{x}_3) &= \nu(b_1) \nu(b_2) b_0 \rho(a_1 a_2 a_3) b_3
+ \nu(b_2) b_0 \rho(a_1) \mathring{b}_1 \rho(a_2 a_3) b_3\\
&\quad + \nu(b_1) b_0 \rho(a_1 a_2) \mathring{b}_2 \rho(a_3) b_3
+ b_0 \rho(a_1) \mathring{b}_1 \rho(a_2) \mathring{b}_2 \rho(a_3)b_3\\
&\quad + \nu \Big( \mathring{b}_1 \rho(a_2) \mathring{b}_2 \Big)
b_0 \big[ \rho(a_1 a_3) - \rho(a_1) \rho(a_3) \big] b_3\\
&= b_0 \rho(a_1) b_1 \rho(a_2) b_2 \rho(a_3) b_3\\
&\quad +\nu(b_1) \nu(b_2) b_0 \Big[ \rho(a_1 a_2 a_3)-\rho(a_1) \rho(a_2a_3)
- \rho(a_1 a_2) \rho(a_3) + \rho(a_1) \rho(a_2) \rho(a_3) \Big]b_3\\
&\quad + \nu(b_1) b_0 \big[\rho(a_1 a_2) - \rho(a_1) \rho(a_2) \big]
b_2 \rho(a_3) b_3 + \nu(b_2) b_0 \rho(a_1) b_1 \big[ \rho(a_2 a_3) - \rho(a_2)
\rho(a_3) \big] b_3\\
&\quad + \Big[ \nu(b_1) \nu(b_2) \nu(\rho(a_2)) -
\nu(b_1) \nu \big( \rho(a_2) b_2 \big)
- \nu\big(b_1 \rho(a_2)\big) \nu(b_2) \\
&\qquad \qquad \qquad + \nu \big( b_1 \rho(a_2) b_2 \big) \Big]b_0 \big[
\rho(a_1 a_3) - \rho(a_1) \rho(a_3) \big] b_3\\
\UM(\vec{x}_3) &= \nu(b_1) \nu(b_2) b_0 \Big[ \rho(a_1 a_2 a_3)-\rho(a_1) \rho(a_2a_3)
- \rho(a_1 a_2) \rho(a_3) + \rho(a_1) \rho(a_2) \rho(a_3) \Big]b_3\\
&\quad +  \Big[ \nu(b_1) \nu(b_2) \nu(\rho(a_2)) -
\nu(b_1) \nu \big( \rho(a_2) b_2 \big)
- \nu\big(b_1 \rho(a_2)\big) \nu(b_2) \Big]b_0 \big[
\rho(a_1 a_3) - \rho(a_1) \rho(a_3) \big] b_3
\end{align*}

\begin{align*}
\RM(\vec{x}_4) &= b_0 \rho(a_1) b_1 \rho(a_2) b_2
\rho(a_3) b_3 \rho(a_4) b_4  + \nu(b_1) b_0 \big[ \rho(a_1 a_2) - \rho(a_1) \rho(a_2)\big]
b_2 \rho(a_3) b_3 \rho(a_4) b_4\\
&\quad + \nu \big(b_1 \rho(a_2) b_2 \big) b_0 \big[
\rho(a_1 a_3) - \rho(a_1) \rho(a_3) \big] b_3 \rho(a_4) b_4 \\
&\quad + \nu \big(b_1 \rho(a_2) b_2 \rho(a_3) b_3 \big) b_0
\big[ \rho(a_1 a_4) - \rho(a_1) \rho(a_4) \big] b_4\\
&\quad + \nu(b_2) b_0 \rho(a_1) b_1 \big[ \rho(a_2 a_3)
- \rho(a_2) \rho(a_3) \big] b_3 \rho(a_4) b_4\\
&\quad + \nu \big( b_2 \rho(a_3) b_3 \big) b_0 \rho(a_1) b_1
\big[ \rho(a_2 a_4) - \rho(a_2) \rho(a_4) \big] b_4\\
&\quad + \nu(b_3) b_0 \rho(a_1) b_1 \rho(a_2) b_2
\big[ \rho(a_3 a_4) - \rho(a_3) \rho(a_4) \big] b_4\\
&\quad + \nu \big( b_1 \rho(a_2) b_2 \big) \nu(b_3) b_0
\big[ \rho(a_1 a_3 a_4) - \rho(a_1) \rho(a_3 a_4) - \rho(a_1 a_3)
\rho(a_4) + \rho(a_1) \rho(a_3) \rho(a_4) \big] b_4\\
&\quad + \big[\nu \big( b_1 \rho(a_2) b_2 \big) \nu(b_3) \nu(\rho(a_3))
- \nu \big(b_1 \rho(a_2) b_2 \big) \nu \big( \rho(a_3) b_3 \big)\\
&\qquad \qquad \qquad - \nu \big( b_1 \rho(a_2) b_2 \rho(a_3) \big) \nu(b_3) \big]
b_0 \big[ \rho(a_1 a_4) - \rho(a_1) \rho(a_4) \big] b_4\\
&\quad + \nu(b_1) \nu \big( b_2 \rho(a_3) b_3 \big) b_0
\big[ \rho(a_1 a_2 a_4) - \rho(a_1) \rho(a_2 a_4)
- \rho(a_1 a_2) \rho(a_4) + \rho(a_1) \rho(a_2) \rho(a_4) \big]b_4\\
&\quad + \big[ \nu(b_1) \nu \big( b_2 \rho(a_3) b_3 \big) \nu(\rho(a_2))
- \nu(b_1) \nu \big( \rho(a_2) b_2 \rho(a_3) b_3 \big)\\
&\qquad \qquad \qquad - \nu \big( b_1 \rho(a_2) \big) \nu \big(b_2 \rho(a_3) b_3 \big)\big]
b_0 \big[ \rho(a_1 a_4) - \rho(a_1) \rho(a_4) \big] b_4\\
&\quad + \nu(b_1) \nu(b_2) b_0 \big[ \rho(a_1 a_2 a_3)
- \rho(a_1) \rho(a_2 a_3) - \rho(a_1 a_2) \rho(a_3)
+ \rho(a_1) \rho(a_2) \rho(a_3) \big] b_3 \rho(a_4) b_4\\
&\quad + \big[\nu(b_1) \nu(b_2) \nu(\rho(a_2)) - \nu(b_1) \nu \big(
\rho(a_2) b_2 \big) \\
&\qquad \qquad \qquad - \nu \big( b_1 \rho(a_2) \big) \nu(b_2) \big]
b_0 \big[ \rho(a_1 a_3) - \rho(a_1) \rho(a_3) \big] b_3 \rho(a_4)b_4\\
&\quad + \nu(b_2) \nu(b_3) b_0 \rho(a_1) b_1 \big[ \rho(a_2 a_3 a_4)
- \rho(a_2) \rho(a_3 a_4) - \rho(a_2 a_3) \rho(a_4)
+ \rho(a_2) \rho(a_3) \rho(a_4) \big] b_4\\
&\quad + \big[ \nu(b_2) \nu(b_3) \nu(\rho(a_3)) - \nu(b_2) \nu
\big( \rho(a_3) b_3 \big) \\
&\qquad \qquad \qquad - \nu \big(b_2 \rho(a_3) \big) \nu(b_3)\big] b_0 \rho(a_1) b_1 \big[ \rho(a_2 a_4) - \rho(a_2) \rho(a_4) \big] b_4
\end{align*}

\begin{align*}
\LM(\vec{x}_4) &= \nu(b_1) \nu(b_2) \nu(b_3) b_0 \rho(a_1 a_2
a_3 a_4) b_4 + \nu(b_2) \nu(b_3) b_0 \rho(a_1) \mathring{b}_1
\rho(a_2 a_3 a_4) b_4\\
&\quad +\nu(b_1) \nu(b_3) b_0 \rho(a_1 a_2) \mathring{b}_2
\rho(a_3 a_4) b_4 + \nu(b_1) \nu(b_2) b_0
\rho(a_1 a_2 a_3) \mathring{b}_3 \rho(a_4) b_4\\
&\quad + \nu(b_3) b_0 \rho(a_1) \mathring{b}_1 \rho(a_2) \mathring{b}_2
\rho(a_3 a_4) b_4 + \nu \big( \mathring{b}_1 \rho(a_2)
\mathring{b}_2 \big) \nu(b_3) b_0 \big[ \rho(a_1 a_3 a_4)
- \rho(a_1) \rho(a_3 a_4) \big] b_4\\
&\quad + \nu(b_2) b_0 \rho(a_1) \mathring{b}_1 \rho(a_2 a_3)
\mathring{b}_3 \rho(a_4) b_4 + \nu \big( \mathring{b}_1
\rho(a_2 a_3) \mathring{b}_3 \big) \nu(b_2)  b_0 \big[ \rho(a_1 a_4)
- \rho(a_1) \rho(a_4) \big] b_4\\
&\quad + \nu(b_1) b_0 \rho(a_1 a_2) \mathring{b}_2 \rho(a_3) \mathring{b}_3
\rho(a_4) b_4 + \nu(b_1) \nu \big( \mathring{b}_2 \rho(a_3) \mathring{b}_3
\big) b_0 \big[ \rho(a_1 a_2 a_4) - \rho(a_1 a_2) \rho(a_4)\big]b_4\\
&\quad +b_0 \rho(a_1) \mathring{b}_1 \rho(a_2) \mathring{b}_2
\rho(a_3) \mathring{b}_3 \rho(a_4) b_4  \\
&\quad + \nu \big(\mathring{b}_1 \rho(a_2) \mathring{b}_2 \big) b_0 \big[
\rho(a_1 a_3) - \rho(a_1) \rho(a_3) \big] \mathring{b}_3 \rho(a_4) b_4 \\
&\quad + \nu \big(\mathring{b}_1 \rho(a_2) \mathring{b}_2 \rho(a_3) \mathring{b}_3 \big) b_0
\big[ \rho(a_1 a_4) - \rho(a_1) \rho(a_4) \big] b_4\\
&\quad + \nu \big( \mathring{b}_2 \rho(a_3) \mathring{b}_3 \big) b_0 \rho(a_1) \mathring{b}_1
\big[ \rho(a_2 a_4) - \rho(a_2) \rho(a_4) \big] b_4\\
&\quad - \nu \big(\mathring{b}_1 \rho(a_2) \mathring{b}_2 \big) \nu \big( \rho(a_3) \mathring{b}_3 \big)
b_0 \big[ \rho(a_1 a_4) - \rho(a_1) \rho(a_4) \big] b_4\\
&\quad - \nu \big( \mathring{b}_1 \rho(a_2) \big) \nu \big(\mathring{b}_2 \rho(a_3) \mathring{b}_3 \big)
b_0 \big[ \rho(a_1 a_4) - \rho(a_1) \rho(a_4) \big] b_4\\
\end{align*}

\newpage
\begin{align*}
\LM(\vec{x}_4) &= b_0 \rho(a_1) b_1 \rho(a_2) b_2 \rho(a_3) b_3 \rho(a_4) b_4 + \nu(b_1)\nu(b_2) \nu(b_3) b_0 \dss{\rho}{[1,2,3,4]}(a) b_4\\
&+ \nu(b_1) \nu(b_2) \nu(b_3) \nu(\rho(a_2)) b_0 \dss{\rho}{[1,3,4]}(a) b_4
+ \nu(b_1) \nu(b_2) \nu(b_3) \nu(\rho(a_3)) b_0 \dss{\rho}{[1,2,4]}(a) b_4\\
&+ \nu(b_1) \nu(b_3) \nu(b_3) \Big[ \nu(\rho(a_2 a_3)) + 2 \nu(\rho(a_2))
\nu(\rho(a_3)) - \nu \big( \rho(a_2) \rho(a_3) \big) \Big] b_0
\dss{\rho}{[1,4]}(a) b_4\\
&+ \nu(b_1) \nu(b_2) \nu(\rho(a_2)) b_0 \dss{\rho}{[1,3]}(a) b_3
\rho(a_4) b_4- \nu(b_1) \nu(b_2) \nu \big( \rho(a_3) b_3 \big) b_0
\dss{\rho}{[1,2,4]}(a) b_4\\
&- \nu(b_1) \nu(b_2)  \Big[ \nu \big( \dss{\rho}{[2,3]}(a) b_3
\big) - 2 \nu(\rho(a_2)) \nu \big( \rho(a_3) b_3 \big) \Big]
b_0 \dss{\rho}{[1,4]}(a) b_4\\
&+\nu(b_1) \nu(b_3) b_0 \dss{\rho}{[1,2]}(a) b_2
\dss{\rho}{[3,4]}(a) b_4 - \nu(b_1) \nu(b_3) \nu \big( \rho(a_2) b_2 \big) \dss{\rho}{[1,3,4]}(a) b_4\\
&- \nu(b_1) \nu(b_3) \nu \big(b_2 \rho(a_3) \big) b_0
\dss{\rho}{[1,2,4]} b_4\\
&+ \nu(b_1) \nu(b_3) \Big[ \nu \big(\rho(a_2) b_2 \rho(a_3) \big)
- \nu \big( \rho(a_2) b_2 \big) \nu(\rho(a_3))
- \nu(\rho(a_2)) \nu \big( b_2 \rho(a_3) \big) \Big]
b_0 \dss{\rho}{[1,4]} b_4\\
&+ \nu(b_2) \nu(b_3) b_0 \rho(a_1) b_1 \dss{\rho}{[2,3,4]}(a) b_4
- \nu(b_2) \nu(b_3) \nu \big( b_1 \rho(a_2) \big) b_0
\dss{\rho}{[1,3,4]} b_4\\
&- \nu(b_2) \nu(b_3) \Big[ \nu \big( b_1 \dss{\rho}{[2,3]}(a)
\big) + 2 \nu \big( b_1 \rho(a_2) \big) \nu(\rho(a_3)) \Big]
b_0 \dss{\rho}{[1,4]} b_4\\
&+ \nu(b_1) b_0 \dss{\rho}{[1,2]}(a) b_2 \rho(a_3)
b_3 \rho(a_4) b_4
- \nu(b_1) \nu \big( \rho(a_2) b_2 \big) b_0
\dss{\rho}{[1,3]}(a) b_3 \rho(a_4) b_4\\
&+ \nu(b_1) \nu \big( b_2 \rho(a_3) b_3 \big) b_0
\dss{\rho}{[1,2,4]}(a) b_4\\
&- \nu(b_1) \Big[ \nu \big( \rho(a_2) b_2 \rho(a_3) b_3
\big) - \nu(\rho(a_2)) \nu \big( b_2 \rho(a_3) b_3 \big)
- \nu \big( \rho(a_2) b_2 \big)
\nu \big( \rho(a_3) b_3 \big) \Big] b_0
\dss{\rho}{[1,4]}(a) b_4\\
&+ \nu(b_2) b_0 \rho(a_1) b_1 \dss{\rho}{[2,3]}(a) b_3
\rho(a_4) b_4 - \nu(b_2) \nu \big( b_1 \rho(a_2) \big) b_0
\dss{\rho}{[1,3]}(a) b_3 \rho(a_4) b_4\\
&- \nu(b_2) \nu \big( \rho(a_3) b_3 \big) b_0
\rho(a_1) b_1 \dss{\rho}{[2,4]}(a) b_4\\
&+ \nu(b_2) \Big[ \nu \big( b_1 \dss{\rho}{[2,3]}(a) b_3 \big)
+ 2 \nu \big( b_1 \rho(a_2) \big)
\nu \big( \rho(a_3) b_3 \big) \Big] b_0
\dss{\rho}{[1,4]} b_4\\
&+ \nu(b_3) b_0 \rho(a_1) b_1 \rho(a_2) b_2
\dss{\rho}{[3,4]} b_4
+ \nu(b_3) \nu \big( b_1 \rho(a_2) b_2 \big) b_0
\dss{\rho}{[1,3,4]}(a) b_4\\
&- \nu(b_3) \nu \big( b_2 \rho(a_3) \big) b_0 \rho(a_1)
b_1 \dss{\rho}{[2,4]}(a) b_4\\
&- \nu(b_3) \Big[ \nu \big( b_1 \rho(a_2) b_2 \rho(a_3)
\big) - \nu \big( b_1 \rho(a_2) b_2 \big) \nu(\rho(a_3))
- \nu \big( b_1 \rho(a_2) \big) \nu \big(b_2 \rho(a_3)\big)
\Big] b_0 \dss{\rho}{[1,4]} b_4\\
&+ \Big[ \nu \big( b_1 \rho(a_2) b_2 \rho(a_3) b_3 \big)
- \nu \big( b_1 \rho(a_2) b_2 \big) \nu \big(
\rho(a_3) b_3 \big) - \nu \big( b_1 \rho(a_2) \big)
\nu \big( b_2 \rho(a_3) b_3 \big) \Big] b_0
\dss{\rho}{[1,4]} b_4\\
&+ \nu \big( b_1 \rho(a_2) b_2 \big) b_0
\dss{\rho}{[1,3]}(a) b_3 \rho(a_4) b_4
+ \nu \big( b_2 \rho(a_3) b_3 \big) b_0
\rho(a_1) b_1 \dss{\rho}{[2,4]}(a) b_4
\end{align*}

\section{Left-Liberation Collapse Functions}
\begin{align*}
\LC'(\vec{y}_1; \emptyset) &=  a_0 a_1&
 \LC'(\vec{y}_3; \emptyset) &= a_0 a_1 a_2 a_3\\
\LC'(\vec{y}_1; \{1\}) &= (a_0, \mathring{b}_1, a_1) &
 \LC'(\vec{y}_3; \{1\}) &= (a_0, \mathring{b}_1, a_1 a_2 a_3)\\
\RC'(\vec{y}_1; \emptyset) &= \vec{y}_1 &
 \LC'(\vec{y}_3; \{2\}) &=(a_0 a_1, \mathring{b}_2, a_2 a_3)\\
\UC'(\vec{y}_1; \emptyset) &= \vec{y}_1 &
 \LC'(\vec{y}_3; \{3\}) &=(a_0 a_1 a_2, \mathring{b}_3, a_3)\\
\LC'(\vec{y}_2; \emptyset) &= a_0 a_1 a_2 &
 \LC'(\vec{y}_3; \{1,2\}) &= (a_0, \mathring{b}_1, a_1,
\mathring{b}_2,a_2 a_3)\\
\LC'(\vec{y}_2; \{1\}) &= (a_0, \mathring{b}_1, a_1 a_2)&
 \LC'(\vec{y}_3; \{1,3\}) &= (a_0, \mathring{b}_1, a_1 a_2, \mathring{b}_3, a_3)\\
\LC'(\vec{y}_2; \{2\}) &= (a_0 a_1, \mathring{b}_2, a_2) &
 \LC'(\vec{y}_3; \{2,3\}) &= (a_0 a_1, \mathring{b}_2, a_2, \mathring{b}_3, a_3)\\
\LC'(\vec{y}_2; \{1,2\}) &= (a_0, \mathring{b}_1, a_1, \mathring{b}_2, a_2) &
 \LC'(\vec{y}_3; \{1,2,3\}) &= (a_0, \mathring{b}_1, a_1,
\mathring{b}_2, a_2, \mathring{b}_3, a_3)\\
\RC'(\vec{y}_2; \emptyset) &=(a_0, b_1 \rho(a_1) b_2, a_2)&
 \RC'(\vec{y}_3; \emptyset) &= (a_0, b_1 \rho(a_1) b_2
 \rho(a_2) b_3, a_3)\\
\RC'(\vec{y}_2; \{1\}) &= \vec{y}_2 &
 \RC'(\vec{y}_3; \{1\}) &= (a_0, b_1, a_1, b_2 \rho(a_2) b_3, a_3)\\
\UC'(\vec{y}_2; \emptyset) &=(a_0, -b_1 \rho(a_1) b_2, a_2)&
 \RC'(\vec{y}_3; \{2\}) &= (a_0, b_1 \rho(a_1) b_2, a_2, b_3, a_3)\\
\UC'(\vec{y}_2; \{1\}) &= \vec{y}_2 &
 \RC'(\vec{y}_3; \{1,2\}) &= \vec{y}_3 \\
& & \UC'(\vec{y}_3; \emptyset) &= (a_0, b_1 \rho(a_1) b_2
 \rho(a_2) b_3, a_3)\\
& & \UC'(\vec{y}_3; \{1\}) &= (a_0, b_1, a_1, -b_2 \rho(a_2) b_3, a_3)\\
& & \UC'(\vec{y}_3; \{2\}) &= (a_0, -b_1 \rho(a_1) b_2, a_2, b_3, a_3)\\
& & \UC'(\vec{y}_3; \{1,2\}) &= \vec{y}_3
\end{align*}

\section{Left-Liberation Moment Functions}
\begin{align*}
\RM'(\vec{y}_1) &= 0\\
\LM'(\vec{y}_1) &= \nu(b_1) a_0 a_1\\
\UM'(\vec{y}_1) &= \nu(b_1) a_0 a_1\\
\RM'(\vec{y}_2) &= \nu \big( b_1 \rho(a_1) b_2 \big) a_0 a_2\\
\LM'(\vec{y}_2) &= \nu(b_1) \nu(b_2) a_0 a_1 a_2
+ \nu \big( \mathring{b}_1 \rho(a_1) \mathring{b}_2 \big)
a_0 a_2\\
&= \nu(b_1) \nu(b_2) a_0 a_1 a_2 + \Big[
\nu \big( b_1 \rho(a_1) b_2 \big) - \nu(b_1)
\nu\big( \rho(a_1) b_2 \big) \\
&\qquad \qquad - \nu \big(b_1 \rho(a_1)\big)
+ \nu(b_1) \nu(\rho(a_1)) \nu(b_2) \Big] a_0 a_2\\
\UM'(\vec{y}_2) &= \nu(b_1) \nu(b_2) a_0 a_1 a_2
+ \Big[\nu \big( \mathring{b}_1 \rho(a_1) \mathring{b}_2 \big)
-\nu \big(b_1 \rho(a_1) b_2 \big) \Big]
a_0 a_2\\
&= \nu(b_1) \nu(b_2) a_0 a_1 a_2 + \Big[ \nu(b_1) \nu(\rho(a_1)) \nu(b_2) - \nu(b_1) \nu\big( \rho(a_1) b_2 \big)
 - \nu \big(b_1 \rho(a_1)\big)
 \Big] a_0 a_2\\
\RM'(\vec{y}_3) &= \nu(b_1) \nu \big(b_2 \rho(a_2) b_3 \big)
a_0 a_1 a_3 + \nu \big( b_1 \rho(a_1) b_2 \big) \nu(b_3) a_0
a_2 a_3 \\
&\qquad + \Big[ \nu \big( b_1 \rho(a_1) b_2 \rho(a_2) b_3\big)
- \nu \big( b_1 \rho(a_2)\big) \nu \big( b_2 \rho(a_2) b_3\big)
- \nu(b_1) \nu \big( \rho(a_1) b_2 \rho(a_2) b_3 \big)\\
&\qquad \qquad - \nu \big( b_1 \rho(a_1) b_2 \rho(a_2) \big) \nu(b_3)
- \nu \big( b_1 \rho(a_1) b_2 \big) \nu \big( \rho(a_2)b_3\big)\\
&\qquad \qquad + \nu(b_1) \nu \big( \rho(a_1)\big) \nu \big( b_2 \rho(a_2) b_3 \big)
+\nu \big( b_1 \rho(a_1) b_1 \big) \nu \big( \rho(a_2) \big)
\nu(b_3) \Big] a_0 a_3
\end{align*}

\begin{align*}
\LM'(\vec{y}_3) &=\nu(b_1) \nu(b_2) \nu(b_3) a_0 a_1 a_2 a_3
+ \nu(b_1) \nu \big( \mathring{b}_2 \rho(a_2) \mathring{b}_3 \big)
a_0 a_1 a_3\\
&+ \nu(b_3) \nu \big( \mathring{b}_1 \rho(a_1) \mathring{b}_2 \big) a_0 a_2 a_3 + \Big[ \nu \big( \mathring{b}_1 \rho(a_1)
\mathring{b}_2 \rho(a_2) \mathring{b}_3 \big)\\
&\qquad \qquad - \nu \big( \mathring{b}_1 \rho(a_1) \big)
\nu \big( \mathring{b}_2 \rho(a_2) \mathring{b}_3 \big)
- \nu \big( \mathring{b}_1 \rho(a_1) \mathring{b}_2 \big)
\nu \big( \rho(a_2) \mathring{b}_3 \big)\\
&\qquad \qquad + \nu(b_2) \nu \big( \mathring{b}_1 \rho(a_1 a_2)
\mathring{b}_3 \big) \Big] a_0 a_3\\
&= \nu(b_1) \nu(b_2) \nu(b_3) a_0 a_1 a_2 a_3
+ \nu(b_1) \Big[  \nu \big( b_2 \rho(a_2) b_3 \big)
-  \nu(b_2) \nu \big( \rho(a_2) b_3 \big)\\
&\qquad \qquad -  \nu(b_3) \nu \big( b_2 \rho(a_2) \big)
+  \nu(b_2) \nu(b_3) \nu \big( \rho(a_2) \big) \Big] a_0 a_1 a_3 + \nu(b_3) \Big[ \nu \big( b_1 \rho(a_1) b_2 \big)\\
&\qquad \qquad- \nu(b_1) \nu \big( \rho(a_1) b_2 \big)
- \nu(b_2) \nu \big( b_1 \rho(a_1) \big)
+ \nu(b_1) \nu \big( \rho(a_1) \big) \nu(b_2) \Big] a_0 a_2 a_3\\
&\qquad+ \Big[\nu(b_2) \nu \big( b_1 \rho(a_1 a_2) b_3 \big)
- \nu(b_1) \nu(b_2) \nu \big( \rho(a_1 a_2) b_3 \big)
- \nu(b_2) \nu(b_3) \nu \big( b_1 \rho(a_1 a_2) \big)\\
&\qquad \qquad + \nu(b_1) \nu(b_2) \nu(b_3) \nu \big( \rho(a_1 a_2) \big) + \nu \big( b_1 \rho(a_1) b_2 \rho(a_2) b_3 \big)
- \nu(b_1) \nu \big( \rho(a_1) b_2 \rho(a_2) b_3 \big)\\
&\qquad \qquad - \nu(b_2) \nu \big( b_1 \rho(a_1) \rho(a_2)
b_3 \big) - \nu(b_3) \nu \big( b_1 \rho(a_1) b_2 \rho(a_2) \big) + \nu(b_1) \nu(b_2) \nu \big( \rho(a_1) \rho(a_2) b_3\big)\\
&\qquad \qquad + \nu(b_1) \nu(b_3) \nu \big( \rho(a_1)
b_2 \rho(a_2) \big) + \nu(b_2) \nu(b_3) \nu \big(b_1 \rho(a_1) \rho(a_2) \big) \\
&\qquad \qquad- \nu \big( b_1 \rho(a_1)\big) \nu \big( b_2 \rho(a_2) b_3 \big) + \nu(b_1) \nu \big(\rho(a_1)\big)
\nu\big( b_2 \rho(a_2) b_3 \big)\\
&\qquad \qquad +2 \nu(b_2) \nu \big( b_1 \rho(a_1)\big)
\nu \big(\rho(a_2) b_3 \big) + \nu(b_3) \nu \big( b_1 \rho(a_1)\big) \nu \big( b_2 \rho(a_2)\big)\\
&\qquad \qquad - 2 \nu(b_1) \nu(b_2) \nu \big( \rho(a_1)\big)
\nu \big(\rho(a_2) b_3 \big) - \nu(b_1) \nu(b_3) \nu \big(\rho(a_1)\big) \nu \big(b_2 \rho(a_2) \big)\\
&\qquad \qquad- \nu \big(b_1 \rho(a_1) b_2 \big) \nu \big( \rho(a_2) b_3 \big) + \nu(b_1) \nu \big(\rho(a_1) b_2 \big) \nu \big( \rho(a_2) b_3 \big)\\
&\qquad \qquad + \nu(b_3) \nu \big( b_1 \rho(a_1) b_2 \big) \nu \big( \rho(a_2) \big) - \nu(b_1) \nu(b_3) \nu \big( \rho(a_1) b_2 \big)\nu \big(\rho(a_2)\big)\Big] a_0 a_3
\end{align*}

\newpage
\section{Moment Polynomials}
For the sake of brevity, we use $1,2,3$ to denote
$t_1, t_2, t_3$, with the standing assumption
that $0 < t_1 < t_2 < t_3$, and omit listing $a_1, \dots, a_n$;
  hence $\Ss(1,0,3,2)$ is an abbreviation for
  $\Ss(t_1, 0, t_3, t_2; a_1, a_2, a_3, a_4)$,
  and $\phi_{2-1}$ for $\phi_{t_2-t_1}$.

  After the first few, we omit
polynomials in which 0 appears as the first or last index,
since the bimodule property easily reduces these to others, viz. \begin{align*}
\Ss(0, s_1, \dots, s_k; a_0, a_1, \dots, a_k) &= a_0 \phi_\tau
\big(\Ss(s_1-\tau, \dots, s_k-\tau; a_1, \dots, a_k) \big)\\
\Ss(s_1, \dots, s_k, 0;  a_1, \dots, a_k, a_{k+1}) &= \phi_\tau
\big(\Ss(s_1-\tau, \dots, s_k-\tau; a_1, \dots, a_k) \big) a_{k+1}
\end{align*}
where $\tau = \min (s_1, \dots, s_k)$.

We also omit polynomials with consecutive time indices equal, since these can be reduced by multiplying consecutive terms with the same time index; for instance, \\
$\Ss(t_1, t_1, t_2, t_3, t_3; a_1, a_2, a_3, a_4, a_5) = \Ss(t_1, t_2, t_3; a_1 a_2, a_3, a_4 a_5)$.

\begin{align*}
\Ss(0) &= a_1  \\
\Ss(0,1) &= a_1 \phi_1(a_2)\\
\Ss(1,0) &= \phi_1(a_1) a_2\\
\Ss(0,1,0) &= a_1 \phi_1(a_2) a_3\\
\Ss(1,0,1) &= \phi_1(a_1) a_2 \phi_1(a_3) +\omega(a_2) \big[ \phi_1(a_1 a_3) - \phi_1(a_1) \phi_1(a_3) \big] \\
\Ss(0,1,2) &= a_1 \phi_1 \big( a_2 \phi_{2-1}( a_3) \big)\\
\Ss(0,2,1) &= a_1 \phi_1 \big( \phi_{2-1}(a_2) a_3 \big)\\
\Ss(1,0,2) &= \phi_1(a_1) a_2 \phi_2(a_3) + \omega(a_2) \big[ \phi_1 \big(a_1 \phi_{2-1}(a_3)
\big) - \phi_1(a_1) \phi_2(a_3) \big]\\
\Ss(1,2,0) &= \phi_1 \big( a_1 \phi_{2-1}(a_2)\big)a_3\\
\Ss(2,0,1) &= \phi_2(a_1) a_2 \phi_1(a_3)
+ \omega(a_2) \big[ \phi_1 \big( \phi_{2-1}(a_1) a_3 \big)
- \phi_2(a_1) \phi_1(a_3) \big]\\
\Ss(2,1,0) &= \phi_1 \big( \phi_{2-1}(a_1) a_2 \big) a_3\\
\Ss(1,0,1,2) &= \phi_1(a_1) a_2 \phi_1 \big( a_3 \phi_{3-1}(a_4) \big) + \omega(a_2) \big[ \phi_1 \big( a_1 a_3 \phi_{2-1}(a_4) \big) - \phi_1(a_1) \phi_1 \big( a_3 \phi_{2-1}(a_4) \big) \big]\\
\Ss(1,0,2,1) &= \phi_1(a_1) a_2 \phi_1 \big( \phi_{2-1}(a_3) a_4 \big) +\omega(a_2) \big[ \phi_1 \big( a_1 \phi_{2-1}(a_3) a_4 \big) - \phi_1(a_1) \phi_1 \big( \phi_{2-1}(a_3) a_4 \big)\big]\\
\Ss(1,2,0,1) &= \phi_1 \big( a_1 \phi_{2-1}(a_2)\big)
a_3 \phi_1(a_4) + \omega(a_3) \big[ \phi_1 \big(
a_1 \phi_{2-1}(a_2) a_4 \big) - \phi_1 \big( a_1 \phi_{2-1}(a_2)\big) \phi_1(a_4) \big]\\
\Ss(2,1,0,1) &= \phi_1 \big( \phi_{2-1}(a_1) a_2 \big) a_3 \phi_1(a_4) + \omega(a_3) \big[ \phi_1\big( \phi_{2-1}(a_1)
a_2 a_4 \big) - \phi_1 \big( \phi_{2-1}(a_1) a_2\big)
\phi_1(a_4) \big]\\
\Ss(1,2,0,2) &=\phi_1 \big( a_1 \phi_{2-1}(a_2) \big)
a_3 \phi_2(a_4) + \omega(a_3) \big[ \phi_1 \big( a_1 \phi_{2-1}(a_2a_4) \big) - \phi_1 \big( a_1 \phi_{2-1}(a_3)\big) \phi_2(a_4) \big]\\
\Ss(2,0,1,2) &= \phi_2(a_1) a_2\phi_1 \big( a_3 \phi_{2-1}(a_4)\big) + \omega(a_2) \big[ \phi_1 \big(
\phi_{2-1}(a_1) a_3 \phi_{2-1}(a_4) \big) -
\phi_2(a_1) \phi_1 \big( a_3 \phi_{2-1}(a_4)\big)\big]\\
&\qquad \qquad \qquad + \omega(a_2) \omega(a_3) \big[
\phi_2(a_1 a_4) - \phi_2(a_1) \phi_2(a_4)\big]\\
\Ss(2,0,2,1) &= \phi_2(a_1) a_2 \phi_1 \big( \phi_{2-1}(a_3) a_4 \big) + \omega(a_2) \big[ \phi_1 \big( \phi_{2-1}(a_1 a_3) a_4\big) - \phi_2(a_1) \phi_1 \big( \phi_{2-1}(a_3) a_4\big)\big]\\
\Ss(2,1,0,2) &= \phi_1 \big( \phi_{2-1}(a_1) a_2 \big)
a_3 \phi_2(a_4) + \omega(a_3) \big[ \phi_1 \big( \phi_{2-1}(a_1) a_2 \phi_{2-1}(a_4) \big) - \phi_1 \big( \phi_{2-1}(a_1) a_2
\big) \phi_2(a_4) \big]\\
&\qquad \qquad \qquad + \omega(a_2) \omega(a_3)
\big[ \phi_2(a_1 a_4) - \phi_2(a_1) \phi_2(a_4)\big]
\end{align*}

\begin{align*}
\Ss(1,0,2,3) &= \phi_1(a_1) a_2 \phi_2 \big( a_3
\phi_{3-2}(a_4) \big) + \omega(a_2) \Big[ \phi_1 \big(
a_1 \phi_{2-1} \big( a_3 \phi_{3-2}(a_4)\big)\big)
- \phi_1(a_1) \phi_2 \big( a_3 \phi_{3-2}(a_4)\big)\Big]\\
\Ss(1,0,3,2) &= \phi_1(a_1) a_2 \phi_2 \big( \phi_{3-2}(a_3) a_4 \big) + \omega(a_2) \Big[ \phi_1\big(a_1 \phi_{2-1}\big(
\phi_{3-2}(a_3) a_4 \big) \big) - \phi_1(a_1)
\phi_2 \big(\phi_{3-2}(a_3)a_4\big) \Big]\\
\Ss(1,2,0,3) &= \phi_1 \big( a_1 \phi_{2-1}(a_2)\big)
a_3 \phi_3(a_4) + \omega(a_3) \Big[ \phi_1 \big(
a_1 \phi_{2-1}\big(a_2 \phi_{3-2}(a_4)\big)\big)
- \phi_1 \big( a_1 \phi_{2-1}(a_2) \big) \phi_3(a_4)\Big]\\
\Ss(1,3,0,2) &= \phi_1 \big( a_1 \phi_{2-1}(a_3)\big)
a_3 \phi_2(a_4) + \omega(a_3) \Big[ \phi_1 \big( a_1
\phi_{2-1} \big( \phi_{3-2}(a_2) a_4 \big) \big)
- \phi_1 \big( a_1 \phi_{3-1}(a_2)\big) \phi_2(a_4) \Big]\\
\Ss(2,0,1,3) &= \phi_2(a_1) a_2 \phi_1 \big( a_3 \phi_{3-1}(a_4)\big) + \omega(a_2) \Big[ \phi_1 \big( \phi_{2-1}(a_1) a_3 \phi_{3-1}(a_4) \big) - \phi_2(a_1)
\phi_1 \big( a_3 \phi_{3-1}(a_4) \big) \Big]\\
&\qquad \qquad \qquad + \omega(a_2) \omega(a_3) \Big[
\phi_2\big(a_1 \phi_{3-2}(a_4)\big) - \phi_2(a_1)
\phi_3(a_4) \Big]\\
\Ss(2,0,3,1) &= \phi_2(a_1) a_2 \phi_1 \big(
\phi_{3-1}(a_3) a_4 \big) + \omega(a_2) \Big[
\phi_1 \big( \phi_{2-1} \big( a_1 \phi_{3-2}(a_3) \big)
a_4 \big) - \phi_2(a_1) \phi_1 \big( \phi_{3-1}(a_3) a_4\big)
\Big]\\
\Ss(2,1,0,3) &= \phi_1 \big( \phi_{2-1}(a_1) a_2 \big)
a_3 \phi_3(a_4) + \omega(a_3) \big[ \phi_1 \big(
\phi_{2-1}(a_1) a_2 \phi_{3-1}(a_2)\big) - \phi_1 \big(
\phi_{2-1}(a_1) a_2 \big) \phi_3(a_4) \Big]\\
&\qquad \qquad \qquad + \omega(a_2) \omega(a_3)
\Big[ \phi_2 \big( a_1 \phi_{3-2}(a_4)\big)
- \phi_2(a_1) \phi_3(a_4) \Big]\\
\Ss(2,3,0,1) &= \phi_2 \big( a_1 \phi_{3-2}(a_2)\big)
a_3 \phi_1(a_4) + \omega(a_3) \Big[ \phi_1 \big(
\phi_{2-1}(a_1 \phi_{3-2}(a_3)\big) a_4 \big)
- \phi_2 \big( a_1 \phi_{3-2}(a_2)\big) \phi_1(a_4)\Big]\\
\Ss(3,0,1,2) &= \phi_3(a_1) \phi_1 \big( a-3 \phi_{2-1}(a_4)\big) + \omega(a_2) \Big[ \phi_1 \big( \phi_{3-1}(a_1) a_3 \phi_{2-1}(a_4)\big) - \phi_3(a_1) \phi_1 \big( a_3 \phi_{2-1}(a_4)\big) \Big]\\
&\qquad \qquad \qquad + \omega(a_2) \omega(a_3) \Big[
\phi_2 \big( \phi_{3-2}(a_1) a_4 \big) - \phi_3(a_1)
\phi_2(a_4) \Big]\\
\Ss(3,0,2,1) &= \phi_3(a_1) a_2 \phi_1 \big( \phi_{2-1}(a_3) a_4 \big) + \omega(a_2) \Big[ \phi_1 \big( \phi_{2-1} \big( \phi_{3-2} (a_1) a_3 \big) a_4 \big)
- \phi_3(a_1) \phi_1 \big( \phi_{2-1}(a_3) a_4 \big) \Big]\\
\Ss(3,1,0,2) &= \phi_1 \big( \phi_{3-1}(a_1) a_2 \big)
a_3 \phi_2(a_4) + \omega(a_3) \Big[ \phi_1 \big( \phi_{3-1}(a_1) a_2 \phi_{2-1}(a_4) \big) - \phi_1 \big( \phi_{3-1}(a_1) a_2 \big) \phi_2(a_4) \Big]\\
&\qquad \qquad \qquad + \omega(a_2) \omega(a_3) \Big[
\phi_2 \big( \phi_{3-2}(a_1) a_4 \big) - \phi_3(a_1)
\phi_2(a_4) \Big]\\
\Ss(3,2,0,1) &= \phi_2 \big( \phi_{3-2}(a_1) a_2 \big) a_3
\phi_1(a_4) + \omega(a_3) \Big[ \phi_1 \big( \phi_{2-1}
\big( \phi_{3-2}(a_1) a_2 \big) a_4 \big)
- \phi_2 \big( \phi_{3-2}(a_1) a_2 \big) \phi_1(a_4)\Big]
\end{align*}

\newpage
To illustrate possibilities of discontinuity,
we consider the following \\
for $0 < \tau < t_1
< t_2 < t_3$:

\begin{align*}
\Ss(t_1, \tau, t_3, 0, t_2) &= \phi_\tau \big( \phi_{t_1-\tau}(a_1) a_2
\phi_{t_3-\tau}(a_3) \big) a_4 \phi_{t_2}(a_5)\\
&\qquad + \omega(a_2) \phi_\tau \Big( \phi_{t_1-\tau}
\big( a_1 \phi_{t_3-t_1}(a_3) \big) - \phi_{t_1-\tau}(a_1)
\phi_{t_3-\tau}(a_3) \Big) a_4 \phi_{t_2}(a_5) \\
&\qquad + \omega(a_4)\phi_\tau \Big( \phi_{t_1-\tau}(a_1) a_2
 \phi_{t_2-\tau} \big( \phi_{t_3-t_2}(a_3) a_5 \big)
 \Big) \\
&\qquad + \omega(a_2) \omega(a_4)
 \phi_\tau \Big[ \phi_{t_1-\tau} \big( a_1
 \phi_{t_2-\tau} \big( \phi_{t_2-t_2}(a_3) a_5 \big) \big)
 - \phi_{t_1-\tau}(a_1) \phi_{t_2-\tau} \big(\phi_{t_3-t_2}(a_3) a_5\big) \Big]\\
&\qquad - \omega(a_4) \phi_\tau \big(\phi_{t_1-\tau}(a_1) a_2
\phi_{t_3-\tau}(a_3) \big) \phi_{t_2}(a_5)\\
&\qquad - \omega(a_2) \omega(a_4)\phi_\tau \Big( \phi_{t_1-\tau}
\big( a_1 \phi_{t_3-t_1}(a_3) \big) - \phi_{t_1-\tau}(a_1)
\phi_{t_3-\tau}(a_3) \Big) \phi_{t_2}(a_5)\\
\Ss(t_1, 0, t_3, 0, t_2) &= \phi_{t_1}(a_1)
a_2 \phi_{t_3}(a_3) a_4 \phi_{t_2}(a_5)\\
&\quad + \omega(a_2) \omega(a_4)
\Big[ \phi_{t_1} \big(a_1 \phi_{t_2-t_1} \big(\phi_{t_3-t_2}(a_3)
\big) a_5 \big) - \phi_{t_1}(a_1) \phi_{t_2} \big(
\phi_{t_3-t_2}(a_3) a_5 \big)\\
&\qquad \qquad \qquad - \phi_{t_1}\big(a_1 \phi_{t_3-t_1}(a_3) \big) \phi_{t_2}(a_5) +\phi_{t_1}(a_1) \phi_{t_3}(a_3) \phi_{t_2}(a_5) \Big]\\
&\quad + \omega(a_2) \Big[ \phi_{t_1}\big(a_1 \phi_{t_3-t_1}(a_3)\big) - \phi_{t_1}(a_1) \phi_{t_3}(a_3) \Big] a_4
\phi_{t_2}(a_5)\\
&\quad + \omega(a_4) \phi_{t_1}(a_1) a_2 \Big[ \phi_{t_2}\big(\phi_{t_3-t_2}(a_3) a_5 \big) - \phi_{t_3}(a_3) \phi_{t_2}(a_5) \Big]\\
&\quad + \Big[\omega(a_2) \omega(a_4) \omega(a_3)
- \omega(a_2) \omega \big(\phi_{t_3}(a_3) a_4 \big)
- \omega \big( a_2 \phi_{t_3}(a_3) \big) \omega(a_4)\\
&\qquad \qquad \qquad + \omega \big( a_2 \phi_{t_3}(a_3) a_4 \big) \Big] \big[ \phi_{t_1}\big(a_1 \phi_{t_2-t_1}(a_5) \big)
- \phi_{t_1}(a_1) \phi_{t_3}(a_3) \big]\\
\Ss(t_1, 0, t_3, 0, t_2) &- \lim_{\tau \to 0^+}
\Ss(t_1, \tau, t_3, 0, t_2) = \Big[\omega(a_2) \omega(a_4) \omega(a_3)
 - \omega(a_2) \omega \big(\phi_{t_3}(a_3) a_4 \big)\\
&\qquad - \omega \big( a_2 \phi_{t_3}(a_3) \big) \omega(a_4)
+ \omega \big( a_2 \phi_{t_3}(a_3) a_4 \big) \Big]
\Big[ \phi_{t_1}\big(a_1 \phi_{t_2-t_1}(a_5) \big)
- \phi_{t_1}(a_1) \phi_{t_3}(a_3) \Big]
\end{align*}

\bibliographystyle{amsalpha}
\biblio{Thesis}

\makeatletter
\makeatother
\end{document}